\documentclass[a4paper,8pt,table]{article}
\usepackage[a4paper,margin=1in,footskip=0.25in]{geometry}
\usepackage{amssymb,amsmath,latexsym,color,amsthm,authblk,mathpazo,palatino,bbding,setspace,datetime,breakcites,hyphenat,microtype,diagbox,tikz,xcolor,subcaption,graphicx,titlesec,amsbsy}
\usepackage{titlesec}
\usepackage{mathptmx}
\usepackage[round,authoryear]{natbib}
\usepackage{hyperref,theoremref}
\usepackage[titletoc,title]{appendix}
\usepackage[shortlabels]{enumitem}
\usepackage{float}
\usepackage{amssymb}
\restylefloat{table}
\everymath{\displaystyle}
\usepackage{pstricks}
\usepackage{bigints}
\newcommand{\vertiii}[1]{{\left\vert\kern-0.25ex\left\vert\kern-0.25ex\left\vert #1 
		\right\vert\kern-0.25ex\right\vert\kern-0.25ex\right\vert}}

\usepackage{accents}

\definecolor{armygreen}{rgb}{0.29, 0.8, 0.13}
\definecolor{auburn}{rgb}{0.43, 0.21, 0.1}
\definecolor{burgundy}{rgb}{0.5, 0.0, 0.13}
\definecolor{medium red}{rgb}{.490,.298,.337}
\definecolor{dark red}{rgb}{.235,.141,.161}

\hypersetup{
	colorlinks = true,
	linkcolor = {burgundy},
	citecolor = {burgundy}, 		
	linkbordercolor = {white},
}

\let\OLDthebibliography\thebibliography
\renewcommand\thebibliography[1]{
	\OLDthebibliography{#1}
	\setlength{\parskip}{0pt}
	\setlength{\itemsep}{0pt plus 0.1ex}
}

\DeclareFontFamily{U}{mathx}{\hyphenchar\font45}
\DeclareFontShape{U}{mathx}{m}{n}{<-> mathx10}{}
\DeclareSymbolFont{mathx}{U}{mathx}{m}{n}
\DeclareMathAccent{\widebar}{0}{mathx}{"73}

\titleformat{\section}[block]{\normalfont\scshape\large\filcenter}{\thesection .}{1em}{}
\titleformat{\subsection}{\normalfont\scshape\large}{\thesubsection}{1em}{}
\titleformat{\subsubsection}{\normalfont\scshape\large}{\thesubsubsection}{1em}{}
\titleformat{\paragraph}
{\normalfont\scshape\large}{\theparagraph}{1em}{}
\titlespacing*{\paragraph}
{0pt}{3.25ex plus 1ex minus .2ex}{1.5ex plus .2ex}

\newtheorem{theorem}{Theorem}[section]
\newtheorem{proposition}{Proposition}[section]

\newtheorem{lemma}{Lemma}[section]
\newtheorem{corollary}{Corollary}[section]

\theoremstyle{definition}
\newtheorem{definition}{Definition}[section]
\newtheorem{example}{Example}[section]
\theoremstyle{remark}
\newtheorem{remark}{\textsc{Remark}}[section]

\renewenvironment{proof}[1][\proofname]{{\bfseries #1: }}{\qed}

\newcommand{\sr}{\textcolor{red}}

\newdateformat{monthyeardate}{\monthname[\THEMONTH], \THEYEAR}

\title{\textsc{\textbf{Regularized bayesian best response learning in finite games}}}

\author{\textsc{Sayan Mukherjee}\footnote{(Corresponding Author) Economic Research Unit, Indian Statistical Institute, Kolkata. e-mail: sm.isical@gmail.com} \hspace{1mm} and \hspace{0.8mm} \textsc{Souvik Roy}\footnote{Applied Statistics Unit, Indian Statistical Institute, Kolkata. e-mail: souvik.2004@gmail.com}}
\date{February 7, 2023}

\begin{document}
	
	\maketitle
	
	\begin{abstract}\singlespacing
		\noindent  We introduce the notion of regularized Bayesian best response ({RBBR}) learning dynamic in heterogeneous population games. We obtain such a dynamic via perturbation by an arbitrary lower semicontinuous, strongly convex regularizer in Bayesian population games with finitely many strategies. We provide a sufficient condition for the existence of rest points of the {RBBR} learning dynamic, and hence the existence of regularized Bayesian equilibrium in Bayesian population games. These equilibria are shown to approximate the original Bayesian equilibria for vanishingly small perturbations. We also explore the fundamental properties of the {RBBR} learning dynamic, which includes the existence of unique continuous solutions from arbitrary initial conditions, as well as the continuity of the solution trajectories thus obtained with respect to the initial conditions. Finally, as application the the theory we introduce the notions of Bayesian potential and Bayesian negative semidefinite games and provide convergence results for such games.

		\bigskip
		\noindent\textbf{Keywords}: Regularizer; Bayesian Strategies; Regularized Bayesian Best Response Learning; Bayesian Potential Games, Bayesian Negative Semidefinite Games.
		
		\bigskip
		\noindent\textbf{Mathematics Subject Classification}: 91A07, 91A14, 91A22, 91A26, 91A27.

	\end{abstract}
	
	\newpage
	\tableofcontents
	
	\section{{Introduction}}
	
	In this paper, we seek to introduce the notion of the regularized Bayesian best response ({RBBR}) learning dynamic for heterogeneous populations, where the agents in the population are diverse. Unlike usual homogeneous population games, players are distributed over types taking values in a complete separable metric space. Extrapolating the notion of best response for homogeneous population games, the Bayesian best response also seems to be the plausible behavioral norm in evolutionary game theory. However, again, due to the multiplicity of the Bayesian best response map and the fact that it can change abruptly, makes the study of the map challenging exercise. Thus, the dynamic generated via the Bayesian best response map is difficult to analyze (\cite{gilboa1991social}, \cite{hofbauer1995stability}). In case of regularized best response dynamics, one can circumvent the aforementioned problem by considering a uniquely defined smoothed version of the Bayesian best response, which can then be studied via the standard theory of ordinary differential equations.
	
	As a result, we obtain the {RBBR} learning dynamic, where the agents' payoffs are subject to perturbation by an arbitrary, but fixed regularizer, and subsequently the agents best respond with respect to these perturbed payoffs. The logit dynamic (\cite{fudenberg1998theory}) is a prototypical example of evolutionary dynamics in population games with homogeneous payoff structure, where the agents' payoff are regularized using the well--known Shannon-Gibbs entropy. In our set--up, such a dynamic is called the logit {RBBR} learning dynamic. However, instead of assuming a particular form of the regularizer, {RBBR} learning dynamic is generated via perturbation by an arbitrary strongly convex regularizer, which includes the Tsallis entropy and the Burg entropy.\footnote{Both Tsallis and Burg entropy has appeared in the literature of (deterministically) perturbed best response dynamics, as well as in the context of online non--convex optimization problems. See \cite{hofbauer2002global}, and \cite{heliou2020online} for further details.}
	
	We now provide some motivations behind our study. Regularized best response dynamics are one of the most important classes of dynamics in the literature of evolutionary game theory. The logit dynamic is the most common and widely studied dynamic in evolutionary games. Several authors have, since then, extended the notion of logit dynamic to the class of more general classes of perturbed best response dynamics as a means to study the social behavior of the underlying population (\cite{benaim1999mixed}, \cite{hofbauer2005learning}, \cite{hofbauer2002global}, \cite{hofbauer2007evolution}). 
	To the best of our knowledge, regularized best response have been studied primarily in the context of homogeneous population games. This paper is an attempt to extend the notion of perturbed best response dynamics to the general class of Bayesian games. More formally, as opposed to population games with homogeneous payoff structure, we consider Bayesian population games, where the agents are distributed over types. The purpose of considering Bayesian population games is that in real life scenarios it is highly unlikely that all agents in the population are of the same type. 
	
	We now provide a brief heuristics of our model, which we borrow from \cite{ely2005evolution}. We consider a population where the agents are distributed over a complete separable metric space consisting of types, equipped with the usual Borel sigma-algebra, and a probability measure. Such a probability space is called the type space of the population. The state space of the {RBBR} learning dynamic is the space of Bayesian strategies, which are B\"ochner measurable maps from the type space to the space of all distributions over the strategies. However, the space of Bayesian strategies do not form a normed linear space. We therefore consider the normed linear space of integrable signed Bayesian strategies which are measurable maps from the type space onto $\mathbb{R}^{n}$. Following \cite{ely2005evolution}, we endow this space with the strong norm, which turns it into a Banach space. It then follows from the relevant literature (\cite{diestel1978vector}) that the space of signed Bayesian strategies is actually the Banach space of B\"ochner integrable functions. We then define the regularized Bayesian best response dynamic on the space of signed Bayesian strategies and subsequently study the properties of the dynamic.

	Some overview of our results are as follows. First, we establish the existence of rest points of the {RBBR} learning dynamic, which are fixed points of the regularized Bayesian best response map, and which we call regularized Bayesian equilibria. Next, we show that these regularized equilibria which we obtain, act as approximations to the Bayesian equilibria of the original game for vanishingly small regularizations. This approximation result is somewhat non-trivial and calls for some added conditions on the regularizers apart from being strong convexity. Loosely speaking, we require the regularizer to satisfy any one of the two conditions, namely Condition C1 or Condition C2. Condition C1 ensures that the population concentrates on the strategy which provides maximum payoff as the regularization vanishes. Condition C2 requires the regularizer to satisfy some mild continuity conditions in the sense that small changes in the population state leads only to a small change in the value of the regularizer. We also show that the well-known regularizers considered in the literature of evolutionary game theory satisfies one of these conditions.
	Proceeding, we show the existence of a unique continuous solution trajectory to the RBBR dynamic is continuous for every given initial condition. We further show that the solutions to the RBBR dynamic are continuous with respect to the initial states (that is, if we start from two initial states which are very close, then the respective solutions to the RBBR dynamic will also be enough close). As applications of our results, we introduce the notions of Bayesian potential and Bayesian negative semidefinite games. We bring to notice to the reader that the notion on Bayesian potential games have been considered before in \cite{zusai2018evolutionary}. Both these games are extensions of potential games and negative semidefinite games for homogeneous populations, respectively, for heterogeneous populations.\footnote{See \cite{monderer1996potential}, \cite{cheung2014pairwise}, \cite{lahkar2015logit} for a complete exposition of homogeneous population potential and negative semidefinite games.}

	Unlike evolutionary games with finitely many strategies, where the state space of the underlying dynamic is finite dimensional, the space of Bayesian strategies on which the {RBBR} learning dynamic is defined is an infinite dimensional Banach space. As a result, the evolutionary dynamics on the space of Bayesian strategies in a non trivial extension of the existing results on homogeneous population games with finite strategies. The choice of topology on the state space plays an important role in deriving the above mentioned results. For example, the weak topology on the space of signed Bayesian strategies turns out to be the appropriate topology to prove the result related to the existence of regularized Bayesian equilibrium. Whereas, on the other hand, the strong topology (induced by the strong norm) is necessary to show existence of solution trajectories to the {RBBR} learning dynamic. In establishing convergence results for Bayesian potential and Bayesian negative semidefinite games requires a further exercise of characterizing (strong) norm compact and forward invariant subsets of the space of signed Bayesian strategies in view of constructing a suitable Lyapunov function and then applying relevant results from theory of dynamical systems.

	As far as we know, evolution in Bayesian games have not been explored much in the literature of evolutionary game theory. Such a theory of evolution in Bayesian games was first introduced by \cite{ely2005evolution} in the view of studying the Bayesian best response dynamics where the agents in the population has diverse preferences. The authors did not consider regularization of the Bayesian best response. Instead, they provided sufficient conditions for the Bayesian best response correspondence to be single-valued. They showed that the equilibrium behavior of such a dynamic can be studied via the aggregation, though the disequilibrium behavior of both the dynamics are quite dissimilar. In principle, they showed that the equilibrium behavior and stability of the Bayesian best response dynamic can be fully characterized by that of the corresponding aggregation dynamic. This paper was followed by \cite{sandholm2007evolution} which studies the evolutionary stability of purified equilibria of two-player normal form games under the Bayesian best response dynamic. \cite{staudigl2013co} considers a network evolution model where players have diverse preferences and provides a closed-form expression for the invariant distribution of the underlying co-evolutionary process for potential games. \cite{zusai2018evolutionary} considers evolutionary dynamics in heterogeneous populations obtained via a mean dynamic based on revision protocol functions. They show that under some regularity conditions on the revision protocol function, the evolutionary dynamic admits a unique solution trajectory along with equilibrium stability in heterogeneous potential games. Further, they show equilibrium stability in heterogeneous population games. Our approach is different from \cite{zusai2018evolutionary} in the sense that we focused primarily on generating the best response evolutionary dynamic obtained via regularization of the agents' payoffs. We find this important to study mainly due to two major reasons; firstly, a vast amount of literature in evolutionary game theory devoted to perturbed best response learning, and secondly, to address and rigorously provide insights to similar questions, for eg. existence and approximation of regularized equilibria, as well as convergence of certain classes of games, namely Bayesian potential and Bayesian negative semidefinite games.
	
	The paper is organized as follows. Section \ref{SectionPreliminaries} contains the preliminaries of the model. In Section \ref{SectionPerturbedBestResponse}, we first introduce the notion of a regularizer, followed by the definition of regularized Bayesian best response function, and the associated {RBBR} learning dynamic. We then establish the existence of rest points of the {RBBR} learning dynamic, in particular, the existence of regularized Bayesian equilibrium followed by approximation result concerning regularized Bayesian equilibria. In Section \ref{SectionFundamentalProp}, we establish some fundamental properties of the {RBBR} learning dynamic subject to some conditions on the type space. We apply the theory from the aforementioned sections to prove convergence results pertaining to Bayesian potential games and Bayesian negative semidefinite games in Section \ref{SectionBayesianPotential} and Section \ref{SectionBayesianNegative}, respectively.

	\section{{Preliminaries}}\label{SectionPreliminaries}

	\subsection{Basic notations}
	For $n \geq 1$, let $\mathbb{R}^n$ denote the $n$-dimensional Euclidean space endowed with its Borel sigma algebra $\mathcal{B}_n$. Let $(\Omega,{d}_{\Omega})$ be a complete separable metric space equipped with its Borel sigma algebra $\mathcal{B}_{\Omega}$ and a probability measure ${\xi}$. Let $S:=\{1,\ldots,n\}$ be a finite set and let $\Delta$ be the collection of all probability distributions on $S$ which is identified by the simplex $\Delta:=\{{x}=(x_1,\ldots,x_n)\in [0,1]^n:\sum_{1 \leq j \leq n}x_{j}=1\}$. We use the notations $\Delta^{\circ}$ and $\partial \Delta$ to denote the relative interior and boundary of $\Delta$ respectively. Let $\mathcal{C}_n(\Delta)$ be the collection of all continuous maps from $\Delta$ into $\mathbb{R}^n$ endowed with the topology of uniform convergence.\footnote{Suppose $(f_m)_{m \geq 1}, f \in \mathcal{C}_n(\Delta)$. Then $f_m \to f$ uniformly if $\|f_m-f\|_{\infty}:=\sup_{x \in \Delta}\|f_m(x)-f(x)\| \to 0$ as $m \to \infty$.}
	
	\subsection{Model}
	Consider a large population of heterogeneous agents whose mass is normalized to unity. By a heterogeneous population, we mean that the agents are of different types, and that they are distributed over $\Omega$ according to the given probability measure ${\xi}$. The probability measure $\xi$ thus accounts for the degree of heterogeneity within the population. We call the probability space $(\Omega,\xi)$ the \textbf{type space}, and the probability measure $\xi$, the \textbf{type measure} of the population. The set ${S}$ denotes the strategy space of the players in the population. In homogeneous population games, a probability distribution $\textbf{x} \in \Delta$ denotes the state/aggregate behavior of the population. Thus, if we say that the aggregate behavior of the population is $\textbf{x}$, we mean that $\textbf{x}_{j}$ proportion of players in the population are exercising strategy $j \in {S}$. However in case of heterogeneous populations, we instead consider a more refined version of aggregate behavior which is called a Bayesian strategy.\footnote{The notion of a Bayesian strategy was first introduced in the context of evolutionary games by \cite{ely2005evolution}.}

	A \textit{Bayesian strategy} is a ${\xi}$-measurable mapping $\sigma:\Omega \rightarrow \Delta$, where $\sigma(\omega)$ denotes the belief (distribution) induced by the players of type $\omega$ over the strategy space $S$.\footnote{See Appendix \ref{SubSecWeakTopRelWeakCpt}, Definition \ref{Def_mu_mble}.} Loosely speaking, a Bayesian strategy can be viewed as a random vector on the probability space $(\Omega,{\xi})$, taking values in the simplex $\Delta$. Let $\Sigma$ be the collection of all Bayesian strategies. We say that two Bayesian strategies $\sigma$ and ${\rho}$ are equivalent and write $\sigma \sim {\rho}$ if $\sigma(\omega)={\rho}(\omega)$, for ${\xi}$-a.e $\omega \in \Omega$. Thus, with an abuse of notation, we denote the equivalence classes of all Bayesian strategies by ${\Sigma}$ itself. We conclude this section with the definition of Bayesian population games followed by a related example on Bayesian aggregative population games. As we shall observe later that Bayesian population games are extensions of usual population games to heterogeneous populations.
	
	\begin{definition}
		{A \textbf{Bayesian population game} is a map ${\mathcal{G}}:{\Sigma} \times \Omega \rightarrow \mathbb{R}^{n}$ such that for every $\sigma \in \Sigma$, the mapping} $$\omega \mapsto \mathcal{G}(\sigma,\omega) \hspace{0.5cm} \textit{ is } \mathcal{B}_{\Omega}/\mathcal{B}_n \text{ measurable}.$$ 
	\end{definition}
	
	The definition of Bayesian population games have the following interpretation. For a Bayesian strategy $\sigma \in \Sigma$, $\mathcal{G}(\sigma,\omega)$ denotes the payoff vector of the players in the population of type $\omega \in \Omega$. More precisely, if the population is playing Bayesian strategy $\sigma$, then a player of type $\omega$ exercising strategy $i$ receives a payoff $\mathcal{G}^{i}(\sigma, \omega)$. \cite{zusai2018evolutionary} calls such a game a heterogeneous population game. The following is a special example of a Bayesian population game, where the space of types $\Omega$ are identified by the payoff functions (population games) which the players in the population use to calculate their payoff. We call such a game, a {Bayesian aggregative population game}. Such games have been considered earlier in \cite{ely2005evolution}. We clarify that our definition of a Bayesian population game has a more general structure and all our results in the subsequent sections of the paper holds true for Bayesian populations games with the general set-up.
	
	\begin{example}[{\textbf{Bayesian aggregative population games}}, \cite{ely2005evolution}]\label{DefAggregateBayPopGame}
		\textit{In Bayesian aggregative population games, the type space of the agents in the population is the collection of continuous maps from $\Delta$ to $\mathbb{R}^n$. Mathematically, we have $(\Omega, d_{\Omega}):=(\mathcal{C}_n(\Delta), \|\cdot\|_{\infty})$, where $\|\cdot\|_{\infty}$ denotes the sup-norm metric, and is defined as $\|\omega-\bar{\omega}\|_{\infty}:=\sup_{{x} \in \Delta}\|\omega({x})-\bar{\omega}({x})\|$ for all $\omega,\bar{\omega} \in \Omega$. Since $\Delta$ is compact, we have that $\mathcal{C}_n(\Delta)$, and thus, $\Omega$ is a complete and separable metric space. Let ${\xi}$ be a probability measure on $\Omega$. Note that from the standard theory of evolutionary games, the elements of $\Omega$ are population games. The \textit{expectation operator} is a mapping ${\mathcal{E}}:{\Sigma} \rightarrow \Delta$ such that for every ${\sigma} \in {\Sigma}$,} $${\mathcal{E}}^{i}({\sigma}):= \int_{\Omega}{\sigma}^{i}(\omega){\xi}(d\omega), \hspace{3mm} \text{ for all } i \in {S}.\footnote{Since the Bayesian strategies take values in the compact set $\Delta$, the expectation operator $\mathcal{E}$ is well--defined.}$$
		
		\textit{Thus, given $\sigma \in \Sigma$, $\mathcal{E}(\sigma)$ denotes the aggregate belief of the population over the strategies with respect to the probability measure ${\xi}$. 
			Following \cite{ely2005evolution}, we define the Bayesian aggregative population game as a mapping $\mathcal{G}:\Sigma \times \Omega \rightarrow \mathbb{R}^{n}$ such that for every $\sigma \in \Sigma$, $$\mathcal{G}(\sigma,\omega):=\omega(\mathcal{E}(\sigma)), \hspace{5mm} \text{ for all } \omega \in \Omega.$$ 
			The definition of Bayesian aggregative population game has the following interpretation. Given a Bayesian strategy $\sigma$, $\omega^i(\mathcal{E}(\sigma))$ denotes the payoff to a player of type $\omega \in \Omega$ exercising strategy $i \in S$ at aggregate population state $\mathcal{E}(\sigma)$. Thus, the payoff to a player of type $\omega$ exercising strategy $i$ depends only on the aggregate belief $\mathcal{E}(\sigma)$ over the strategies.
			Note that $\mathcal{G}$ above is well-defined since for every $\sigma \in \Sigma$, the mapping $\omega \mapsto \omega(\mathcal{E}(\sigma))$ is continuous, and hence measurable.} 
	\end{example}

	\section{{Regularization of Bayesian best response}}\label{SectionPerturbedBestResponse}
	In this section we seek to define the notion of regularized Bayesian best response learning dynamic, which in turn requires us to define the notion of Bayesian best response map. It is an extension of best response maps to the context of Bayesian population games. For a Bayesian strategy $\sigma \in \Sigma$, the expected payoff to a player of type $\omega$ against an aggregate population state (aggregate belief) $\textbf{y} \in \Delta$ is given by $\langle \textbf{y},\mathcal{G}(\sigma,\omega)\rangle$,	where $\langle\cdot,\cdot\rangle$ denotes the usual inner product on $\mathbb{R}^{n}$. Quite naturally, the \textbf{Bayesian best response correspondence} is a mapping ${\beta}:\Sigma \twoheadrightarrow \Sigma$  such that for all $\sigma \in \Sigma$,
	
	\begin{equation}\label{EqBayesianBestResponse}
		{\beta}[\sigma](\omega):=\arg\max_{\textbf{y} \in \Delta}\langle \textbf{y},\mathcal{G}(\sigma,\omega)\rangle, \hspace{3mm} \text{ for all } \omega \in \Omega.
	\end{equation}
	
	Note that from the definition of Bayesian population games, $\mathcal{G}(\sigma,\omega)$ denotes the payoff vector a player of type $\omega$ faces at Bayesian strategy $\sigma$. Thus, the best response to a player of type $\omega$ at Bayesian strategy $\sigma$ is any probability distribution $\textbf{y}$ on $S$ which maximizes the expected payoff $\langle \textbf{y},\mathcal{G}(\sigma,\omega)\rangle$.
	
	A Bayesian strategy $\sigma^{\circ} \in \Sigma$ is said to be a \textbf{Bayesian equilibrium} of the game $\mathcal{G}$ is $\sigma^{\circ} \in \beta(\sigma^{\circ})$. In terms of the game $\mathcal{G}$, the above condition can equivalently be expressed in the following way: $\sigma^{\circ}$ is a Bayesian equilibrium if for $\xi$-a.e $\omega \in \Omega$
	\begin{equation}
		\sigma^{\circ}_j(\omega)>0 \implies \mathcal{G}_{j}(\sigma^{\circ},\omega) \geq \mathcal{G}_{i}(\sigma^{\circ},\omega), \hspace{3mm} \text{ for all } i \in S.
	\end{equation}
	
	In principle, the above condition suggests that a Bayesian strategy $\sigma^{\circ}$ is an equilibrium provided the following condition holds: if a strategy $j \in S$ has positive probability in the induced belief $\sigma^{\circ}(\omega)$, then the payoff vector $\mathcal{G}(\sigma,\omega)$ should have a component-wise maximum at strategy
	$j$. Thus, strategies in the support of the induced belief must have maximum payoff. Let $\text{BE}(\mathcal{G},\xi)$ denote the collection of Bayesian equilibria corresponding to the game $\mathcal{G}$.
	In order to state the result on existence of Bayesian equilibrium, we require the notion of an appropriate topology on the space of Bayesian strategies, which essentially is the weak topology in this case (see Definition \ref{DefWeakTopOnSigma}).
	
	\begin{definition}
		{A Bayesian population game $\mathcal{G}:\Sigma \times \Omega \to \mathbb{R}^n$ is said to be \textbf{weakly continuous} if the mapping $\sigma \mapsto \mathcal{G}(\sigma,\omega)$ is continuous relative to the weak topology on $\Sigma$, for ${\xi}$-a.e $\omega \in \Omega$.} 
	\end{definition}

	\begin{proposition}
		Suppose that the population game $\mathcal{G}$ is weakly continuous. Then $\text{BE}(\mathcal{G},\xi)$ is a non-empty and compact subset of $\Sigma$.
	\end{proposition}
	\begin{proof}
		Follows from Kakutani-Glicksberg-Fan theorem (see \cite{charalambos2013infinite}).
	\end{proof}
	
	Suppose now that we fix a type $\omega \in \Omega$. We then define the \textbf{game sliced at type $\omega$} as a mapping $\mathcal{G}_{\omega}:\Sigma \to \mathbb{R}^n$ such that $\mathcal{G}_{\omega}(\sigma):=\mathcal{G}(\sigma,\omega)$ for all $\sigma \in \Sigma$. In words, the sliced game determines how payoffs behave as a function of Bayesian strategies, keeping a type fixed. Thus, we can view the sliced game $\mathcal{G}_{\omega}$ as the game corresponding to the underlying sub-population consisting of players of exactly one type $\omega$. We then introduce the \textbf{$\omega$-sliced best response} for the game $\mathcal{G}_{\omega}$ as a mapping $\beta_{\omega}:\Sigma \twoheadrightarrow \Delta$ such that 
	\begin{equation}\label{EqSlicedBayesianBestResponse}
		\beta_{\omega}(\sigma):=\beta[\sigma](\omega), \hspace{3mm} \text{ for all } \sigma \in \Sigma.
	\end{equation}
	
	We say that a Bayesian strategy $\sigma^{\circ}$ is an \textbf{$\omega$-sliced Bayesian equilibrium} of the $\omega$-sliced game $\mathcal{G}_{\omega}$ if $\sigma^{\circ}(\omega) \in \beta_{\omega}(\sigma^{\circ})$.
	We then have the following characterization result for Bayesian equilibria of the original game $\mathcal{G}$ and the equilibria of the sliced game counterparts $\mathcal{G}_{\omega}$.
	
	\begin{proposition}\label{Prop-OriGame-SliceGame}
		A Bayesian strategy $\sigma^{\circ}$ is a Bayesian equilibrium of the game $\mathcal{G}$ iff $\sigma^{\circ}$ is an $\omega$-sliced Bayesian equilibrium of the game $\mathcal{G}_{\omega}$ for $\xi$-a.e $\omega \in \Omega$.
	\end{proposition}
	\begin{proof}
		See Appendix \ref{App-Prop-OriGame-SliceGame}.
	\end{proof}

	As the title of the paper suggests, we now shift our focus towards the regularized version of the Bayesian best response. Recall that the Bayesian best response is a multi-valued mapping, which, under standard evolutionary dynamics, becomes difficult to analyze. To circumvent this problem, we resort to a regularized (smoothed) version of the Bayesian best response mapping in a view to obtain a unique smoothed best response which can then be analyzed using the existing theory of differential equations and dynamical systems. We begin with the definition of a regularizer.
	
	\begin{definition}\label{DefRegularizer}
		{A mapping $\textbf{v}:\Delta \rightarrow \mathbb{R}\cup\{\infty\}$ is called a \textbf{regularizer} on $\Delta$ if:}
		\vspace{-1mm}
		\begin{itemize}
			\item {it is finite except possibly on the relative boundary $\partial \Delta$ of $\Delta$;}
			\vspace{-2mm}
			\item {it is lower semi-continuous on $\Delta$;} 
			\vspace{-2mm}
			\item {it is strongly convex on $\Delta^{\circ}$, in the sense that there exists ${\gamma}>0$ such that for all $0 \leq \alpha \leq 1$,} $$\textbf{v}(\alpha \textbf{x}+(1-\alpha)\textbf{y})\leq\alpha \textbf{v}(\textbf{x})+(1-\alpha)\textbf{v}(\textbf{y})-\dfrac{{\gamma}\alpha(1-\alpha)}{2}\|\textbf{x}-\textbf{y}\|^{2}, \hspace{3mm} \text{ for all } \textbf{x}, \textbf{y} \in \Delta^{\circ}.$$
		\end{itemize}
	\end{definition}
	
	It is well-known (see \cite{coucheney2015penalty}, \cite{mertikopoulos2016learning}) that regularization ensures the existence of a unique best response. We bring to notice that our definition of a regularizer is weaker as compared to the one used in the references mentioned just prior, in the sense that regularizers in \cite{coucheney2015penalty} and \cite{mertikopoulos2016learning} are smooth in the interior $\Delta^{\circ}$; whereas in our case, we only require lower semicontinuity on $\Delta$.
	
	\begin{definition}[\textbf{Regularized Bayesian best response}]\label{DefRBBRMap}
		\textit{Given a noise parameter $\epsilon>0$, the regularized Bayesian best response ({RBBR}) is a mapping ${\beta}_{\epsilon}:\Sigma \rightarrow \Sigma$ such that for every $\sigma \in \Sigma$,} 
		\begin{equation}\label{RBBRMap}
			{\beta}_{\epsilon}[\sigma](\omega):=\arg \max_{\textbf{y} \in \Delta}\langle \textbf{y},\mathcal{G}(\sigma,\omega)\rangle-\epsilon \textbf{v}(\textbf{y}), \hspace{3mm} \text{ for all } \omega \in \Omega.\footnote{The fact that the perturbation function $\textbf{v}$ is strongly convex implies that for every fixed $\sigma \in \Sigma$, the mapping $\textbf{y} \mapsto \langle \textbf{y},\mathcal{G}(\sigma,\omega)\rangle-\epsilon \textbf{v}(\textbf{y})$ has a unique maximum.} 
		\end{equation}
	\end{definition}
	As we mentioned before, the purpose behind defining the regularized version of the Bayesian best response correspondence is to ensure the existence of a unique maximizer in (\ref{RBBRMap}), when the expected payoff is perturbed via the regularizer $\textbf{v}$, thereby making the study of the evolutionary learning dynamic amenable to the existing theory of ordinary differential equations and dynamical systems. Suppose further that the regularizer is steep, in the sense that $\|\nabla\textbf{v}(\textbf{x}_n)\| \to \infty$ as $\textbf{x}_n \to \partial\Delta$. It then follows that for every $\sigma \in \Sigma$, the best response $\beta_{\epsilon}[\sigma](\omega)$ to type $\omega \in \Omega$ lies in the relative interior $\Delta^{\circ}$ (see  \cite{rockafellar1970convex}, Chapter 26).
	
	Also it worthwhile to mention that the notion of deterministic and stochastically perturbed best response in homogeneous populations with finitely many strategies has been studied in several occasions such as \cite{benaim1999mixed}, \cite{hofbauer2002global}, \cite{hofbauer2005learning}, \cite{hofbauer2007evolution} to name a few. In the recent past, \cite{mertikopoulos2016learning} considers reinforcement learning under regularization. 
	In the context of homogeneous population games with a continuum of strategies, the study of perturbed best response dynamics has been generalized to logit best response in \cite{lahkar2015logit} and \cite{perkins2014stochastic}. Very recently, a general class of perturbed best response dynamics generated via an arbitrary perturbation function (regularizer in our case) has been considered in \cite{lahkar2022generalized}. In this paper, we thus focus on extending the notion of perturbed best response to heterogeneous population games with finitely many strategies.

	\subsection{Some examples of regularizers}
	Below we provide some examples of widely used regularizers which have been extensively used in the literature of evolution and learning in games, both in the context of finite strategy games and  infinite strategy homogeneous population games. They have also been used extensively in reinforcement learning and online non-convex optimization problems (see \cite{mertikopoulos2016learning}, \cite{heliou2020online}). In the case of continuum strategy games, perturbed best response dynamics using general regularizers have been recently considered in \cite{lahkar2022generalized}.
	
	\begin{example}[\textbf{The Shannon-Gibbs entropy}]\label{ExampleShannon}
		\textit{The most classic and prototypical example of a regularizer is the well-known Shannon-Gibbs entropy which is defined as} 
		\begin{equation}
			\textbf{v}_{\textbf{s}}(\textbf{x}):=\sum_{j \in S}\textbf{x}_{j}\log \textbf{x}_{j}, \hspace{3mm} \text{ for all } \textbf{x} \in \Delta.
		\end{equation}
		\textit{As is well-known from the literature (see \cite{fudenberg1998theory}), the regularized Bayesian best response map corresponding to $\textbf{v}_{\textbf{s}}$ is such that for every $\sigma \in \Sigma$,} 
		\begin{equation}\label{LogitClosedForm}
			{\beta}^{\textbf{s}}_{\epsilon,j}[\sigma](\omega)=\dfrac{\exp(\epsilon^{-1}\mathcal{G}_{j}(\sigma,\omega))}{\sum_{i \in S}\exp(\epsilon^{-1}\mathcal{G}_{i}(\sigma,\omega))},\hspace{2mm} \text{ for all } \omega \in \Omega, \text{ and all } j \in S.
		\end{equation}
	\end{example}
	\begin{example}[\textbf{The Tsallis entropy}]\label{ExampleTsallis}
		\textit{A well-known generalization of the Shannon-Gibbs entropy is the so-called Tsallis entropy which is defined as}
		\begin{equation}
			\textbf{v}_{\textbf{t}}(\textbf{x}):=(q-{q}^2)^{-1}\sum_{j \in S}(\textbf{x}_{j}-\textbf{x}^{{q}}_{j}), \hspace{3mm} \text{ for } 0<q<1, \text{ for all } \textbf{x} \in \Delta.
		\end{equation}
		\textit{Notice that the Shannon-Gibbs entropy can be obtained from the Tsallis entropy by letting $q \to 1$. It can be shown (see for example \cite{coucheney2015penalty}, \cite{heliou2020online}) that under perturbation by $\textbf{v}_{\textbf{t}}$, the regularized Bayesian best response map takes the following form: for every $\sigma \in \Sigma$,}
		\begin{equation}
			\beta^{\textbf{t}}_{\epsilon,j}[\sigma](\omega):=\Big[\frac{\epsilon}{1-q}\Big]^{1/(1-q)}\frac{1}{(\theta^{\textbf{t}}_{\epsilon}(\sigma,\omega)-\mathcal{G}_j(\sigma,\omega))^{1/(1-q)}}, \hspace{2mm} \text{ for all } \omega \in \Omega, \text{ and all } j \in S,
		\end{equation}
		\textit{where $\theta^{\textbf{t}}(\sigma,\omega)>\max_{j \in S}\mathcal{G}_{j}(\sigma,\omega)$ is such that $\beta^{\textbf{t}}_{\epsilon}[\sigma](\omega)$ is a valid probability distribution on $S$ for every $\sigma \in \Sigma$ and for every $\omega \in \Omega$.}
	\end{example}\label{ExampleBurg}
	\begin{example}[\textbf{The log-barrier entropy}]
		\textit{Another important example of a regularizer is the logarithmic barrier entropy, which is sometimes also called the Burg entropy (see \cite{hofbauer2002global}) and is defined as}
		\begin{equation}
			\textbf{v}_{\textbf{b}}(\textbf{x}):=-\sum_{j \in S}\log \textbf{x}_{j}, \hspace{3mm} \text{ for all } \textbf{x} \in \Delta.
		\end{equation}
		\textit{In case of log-barrier regularization, the closed form expression of the regularized Bayesian best response map is such that for every $\sigma \in \Sigma$,}
		\begin{equation}
			\beta^{\textbf{b}}_{\epsilon,j}[\sigma](\omega):=\frac{\epsilon}{\theta^{\textbf{b}}_{\epsilon}(\sigma,\omega)-\mathcal{G}_j(\sigma,\omega)}, \hspace{2mm} \text{ for all } \omega \in \Omega, \text{ and all } j \in S,
		\end{equation}
		\textit{where $\theta^{\textbf{b}}_{\epsilon}(\sigma,\omega)>\max_{j \in S}\mathcal{G}_{j}(\sigma,\omega)$ is such that $\beta^{\textbf{b}}_{\epsilon}[\sigma](\omega)$ is a valid probability distribution on $S$ for every $\sigma \in \Sigma$ and for every $\omega \in \Omega$ (see \cite{heliou2020online} for the continuum strategy counterpart).}
	\end{example}

	Before moving ahead with the notion of regularized Bayesian best response, it is necessary to ensure that $\beta_{\epsilon}$ is well-defined in the sense that $\beta_{\epsilon}[\sigma]:\Omega \to \Delta$ is a $\xi$-measurable mapping for every $\sigma \in \Sigma$. To this end, we have the following proposition.
	\begin{proposition}\label{Prop}
		For every $\epsilon>0$ and $\sigma \in \Sigma$, the mapping $\omega \mapsto {\beta}_{\epsilon}[\sigma](\omega)$ is $\xi$-measurable.
	\end{proposition}
	\begin{proof}
		See Appendix \ref{AppProp}.
	\end{proof}

	\subsection{Regularized Bayesian best response dynamic}
	In this section, we seek to define the notion of the regularized Bayesian best response dynamic on the space of Bayesian strategies. In what follows, we first extend the space of Bayesian strategies to the space of \textbf{integrable signed Bayesian strategies} which is defined below as: 
	\begin{equation}\label{IntegrableSignedBayStr}
		\widehat{\Sigma}:=\Big\{\widehat{\sigma}:\Omega \to \mathbb{R}^{n} \text{ such that }\widehat{\sigma} \text{ is } {\xi}-\text{measurable and } \vertiii{\widehat{\sigma}}:=\bigintssss_{\Omega}\|\widehat{\sigma}(\omega)\|{\xi}(d\omega)<\infty\Big\}.\footnote{Thus, we identify the space of integrable signed Bayesian strategies by the B\"ochner space $\mathcal{L}^{1}(\Omega,{\xi},\mathbb{R}^{n})$. See Appendix \ref{SubSecWeakTopRelWeakCpt}.}
	\end{equation}

	\begin{definition}\label{DefRBBRLearning}
		{Let $\epsilon>0$ be a noise parameter and let the {RBBR} map be as defined in (\ref{RBBRMap}). Then the \textbf{regularized Bayesian best response learning dynamic} is defined as}
		\begin{equation}\label{RBBRLearning}
			\dot{\sigma}={{\beta}}_{\epsilon}[\sigma]-\sigma.\footnote{We review the notions of continuity and differentiability of trajectories through the space $(\widehat{\Sigma},\vertiii{\hspace{0.5mm}\cdot\hspace{0.5mm}})$ from \cite{ely2005evolution}. Let $(\widehat{\sigma}_{t})_{t \geq 0}$ be a trajectory through $\widehat{\Sigma}$. An integrable signed Bayesian strategy  $\widehat{\sigma} \in \widehat{\Sigma}$ is a \textit{strong limit} of $(\widehat{\sigma}_{t})_{t \geq 0}$ as $t$ approaches $t_0$ if $\lim_{t \to t_0}\vertiii{\widehat{\sigma}_{t}-\widehat{\sigma}}=0$.
				The  trajectory $(\widehat{\sigma}_{t})_{t \geq 0}$ is \textit{continuous} at time $t$ if $\lim_{s \to t}\vertiii{\widehat{\sigma}_{s}-\widehat{\sigma}_{t}}=0$. We say that the trajectory $(\widehat{\sigma}_{t})_{t \geq 0}$ is continuous if it is continuous at every $t \geq 0$. The trajectory $(\widehat{\sigma}_t)_{t \geq 0}$ is \textit{differentiable} if for every $t \geq 0$, there exists $\dot{\widehat{\sigma}_t} \in \widehat{\Sigma}$ such that $$\lim_{\epsilon \longrightarrow 0}\vertiii{\dfrac{{\widehat{\sigma}_{t+\epsilon}-\widehat{\sigma}_t}}{\epsilon}-\dot{\widehat{\sigma}_t}}=0.$$} 
		\end{equation}
	\end{definition}
	From Definition \ref{DefRBBRLearning}, it follows that under the {RBBR} learning dynamic, a Bayesian strategy $\sigma$ moves towards its regularized Bayesian best response ${{\beta}}_{\epsilon}[\sigma]$. A rest point of the {RBBR} learning dynamic is a Bayesian strategy $\sigma^{\circ}$ satisfying $\dot{\sigma}^{\circ}\equiv 0$. Thus any rest point of the dynamic (\ref{RBBRLearning}) is a Bayesian strategy $\sigma^{\circ}$ which satisfies ${{\beta}}_{\epsilon}[\sigma^{\circ}]=\sigma^{\circ}$, and hence becomes a fixed point of the regularized Bayesian best response map. Such a fixed point is called a \textbf{regularized Bayesian equilibrium}. 
	Let $\text{RBE}(\mathcal{G},{\xi},\epsilon)$ be the collection of all regularized Bayesian equilibrium of the game $\mathcal{G}$ {with noise parameter $\epsilon$} under type measure ${\xi}$, that is,  ${\text{RBE}}(\mathcal{G},{\xi};\epsilon):=\{\sigma^{\circ} \in \Sigma:{{\beta}}_{\epsilon}[\sigma^{\circ}]=\sigma^{\circ}\}$. We now state the existence theorem for regularized Bayesian equilibria in Bayesian population games.

	\begin{theorem}\label{ThmFixedPt}
		Suppose that the Bayesian population game $\mathcal{G}:\Sigma \times \Omega \to \mathbb{R}^n$ is weakly continuous. Then for every $\epsilon>0$, the RBBR mapping $\sigma \mapsto {\beta}_{\epsilon}(\sigma)$ admits a fixed point, that is, there exists $\sigma^{\circ}_{\epsilon} \in \Sigma$ such that ${\beta}_{\epsilon}(\sigma^{\circ}_{\epsilon})=\sigma^{\circ}_{\epsilon}$.
	\end{theorem}
	\begin{proof}
		See Appendix \ref{AppFixedPoint}.
	\end{proof}	

	As a corollary to the above theorem, we have the following result which states the existence of regularized equilibrium in Bayesian aggregative population games. 
	
	\begin{corollary}\label{CorBayEqExist}
		Let $\mathcal{G}:\Sigma \times \mathcal{C}_n(\Delta) \rightarrow \mathbb{R}^n$ be a Bayesian aggregative population game as defined in Example \ref{DefAggregateBayPopGame}. Then for every noise parameter $\epsilon>0$, there exists a regularized Bayesian equilibrium of the game $\mathcal{G}$.
	\end{corollary}
	\begin{proof}
		See Appendix \ref{AppCorBayEqExist}.
	\end{proof}	
	We now aim to prove the convergence of regularized Bayesian equilibria to the Bayesian equilibria of the underlying game $\mathcal{G}$ when we allow the noise parameter to be arbitrarily small. In such a scenario, we can then argue that when the noise parameter is close to $0$, the regularized Bayesian equilibrium acts as a smoothed approximation of the original Bayesian equilibrium relative to the underlying weak topology. Such approximation results have been considered, for example in \cite{hofbauer2002global} for finite strategy games and in \cite{lahkar2015logit}, \cite{lahkar2022generalized} for continuum strategy homogeneous population games. Thus, this calls for an extension of such a result in heterogeneous population games, towards which we dedicate the rest of the section. In what follows, we require some technical conditions on the regularizer. We begin with the following definition.
	\begin{definition}
		For a extended real valued function $f:\mathbb{R}^n \to \mathbb{R}\cup\{-\infty,\infty\}$, the \textit{convex conjugate} $f^{*}:\mathbb{R}^n \to \mathbb{R}\cup\{-\infty,\infty\}$ is defined by
		\begin{equation*}
			f^*(\textbf{y}):=\sup_{\textbf{x} \in \mathbb{R}^n}(\langle \textbf{y},\textbf{x} \rangle -f(\textbf{x})), \hspace{3mm} \text{ for all } \textbf{y} \in \mathbb{R}^n. 
		\end{equation*}
	\end{definition}
	
	We also require two technical conditions on the regularizer, namely Condition C1 and Condition C2. We provide some heuristics behind them prior to the definition of both the conditions. 
	
	Condition C1 requires that as the noise level in the learning dynamic is gradually vanishes, the population, while playing their best response at a particular population state must concentrate themselves over that action which has the maximum payoff at that state in the limit. For $\epsilon>0$, let us denote $\epsilon\textbf{v}$ by $\textbf{v}_{\epsilon}$. Then, Condition C1 is defined as follows.
	
	\begin{definition}\label{ConditionC1}
		{A regularizer $\textbf{v}:\Delta \to \mathbb{R} \cup \{\infty\}$ satisfies \textbf{Condition C1} if for all $i,j \in S$ with $i \neq j$ and $\textbf{u} \in \mathbb{R}^n$ with $\textbf{u}_j>\textbf{u}_i$, there exists a function $\lambda:(0,\infty) \to \mathbb{R}$ with $\lim_{x \to 0}\lambda(x)= \infty$ such that}
		\begin{equation}
			\nabla \textbf{v}^{*}_{\epsilon,j}(\textbf{u}) \geq \lambda(\epsilon) \nabla \textbf{v}^{*}_{\epsilon,i}(\textbf{u}), \hspace{3mm} \text{ for sufficiently small } \epsilon>0,
		\end{equation}
		{where $\nabla \textbf{v}^{*}_{\epsilon}(\textbf{u})$ denotes that gradient of the convex conjugate $\textbf{v}^{*}_{\epsilon}$ of $\textbf{v}_{\epsilon}$ at the point $\textbf{u} \in \mathbb{R}^n$.} 
	\end{definition}

	The following condition (Condition C2) ensures that a small change in the population states can only lead to a small change in the value of the regularizer. Before defining the condition technically, we require the following notion; that of a shift of a Bayesian strategy.
	
	\begin{definition}\label{DefShift}
		{Let $\sigma \in \Sigma$ be a Bayesian strategy, and let $\bar{\Omega} \subseteq \Omega$ be such that ${\xi}(\bar{\Omega})>0$. We say that $\widetilde{\sigma}$ is a \textbf{shift} of $\sigma$ over $\bar{\Omega}$ from $\textbf{i}$ to $\textbf{j}$, where $\textbf{i},\textbf{j}:\bar{\Omega}\to S$ with $\textbf{i}(\omega)\neq \textbf{j}(\omega)$ for all $\omega \in \bar{\Omega}$, if for all $\omega \in \bar{\Omega}$, the following three conditions are satisfied:}
		\begin{itemize}
			\item $\widetilde{\sigma}^{\textbf{i}(\omega)}(\omega) \leq {\sigma}^{\textbf{i}(\omega)}(\omega)$,
			\vspace{-2mm}
			\item $\widetilde{\sigma}^{\textbf{j}(\omega)}(\omega)=\sigma^{\textbf{j}(\omega)}(\omega)+(\sigma^{\textbf{i}(\omega)}(\omega)-\widetilde{\sigma}^{\textbf{i}(\omega)}(\omega))$, {and} 
			\vspace{-2mm}
			\item $\widetilde{\sigma}^{k}(\omega)=\sigma^{k}(\omega)$, {for all} $k \neq \textbf{i}(\omega),\textbf{j}(\omega)$. 
		\end{itemize}	
	\end{definition}
	
	We now proceed to state the Condition C2 which calls for a `mild' continuity condition on the regularizer in the sense that small changes in the population state does not lead to abrupt changes in the regularizer. We clarify the reader that the required notion continuity imposed on the regularizer is much weaker than the usual notion of continuity for extended real valued functions.

	\begin{definition}\label{ConditionC2}
		Let $(\sigma_m)_{m \geq 1}$ be a sequence of Bayesian strategies and let $\bar{\Omega} \subseteq \Omega$ be such that $\xi(\bar{\Omega})>0$. We say that a regularizer $\textbf{v}$ satisfies \textbf{Condition C2} with respect to the pair $((\sigma_m)_{m \geq 1},\bar{\Omega})$ if for all $\omega \in \bar{\Omega}$, all distinct $\textbf{i},\textbf{j}:\bar{\Omega}\to S$, with
		\begin{itemize}
			\item $\lim_{m \to \infty}\sigma^{\textbf{i}(\omega)}_m(\omega) > 0$,
			\vspace{-2mm}
			\item $\mathcal{G}^{\textbf{j}(\omega)}(\sigma_m,\omega)>\mathcal{G}^{\textbf{i}(\omega)}(\sigma_m,\omega)$ for sufficiently large $m$ (depending on $\omega$),
		\end{itemize}
		there exists $(\widetilde{\sigma}_m)_{m \geq 1}$, where $\widetilde{\sigma}_m$ is a shift of $\sigma_m$ over $\bar{\Omega}$ from $\textbf{i}$ to $\textbf{j}$, and a constant $\kappa_i(\omega) >0$ (independent of $m$) such that
		\begin{equation}
			|\textbf{v}(\sigma_{m}(\omega))-\textbf{v}(\widetilde{\sigma}_m(\omega))| \leq \kappa_{\textbf{i}(\omega)}(\omega) (\sigma^{\textbf{i}(\omega)}_m(\omega)-\widetilde{\sigma}^{\textbf{i}(\omega)}_m(\omega)),
		\end{equation}
		for sufficiently large $m$ (depending on $\omega$).	
		
		We say that a regularizer $\textbf{v}$ satisfies \textbf{Condition C2} with respect to a sequence of Bayesian strategies $(\sigma_m)_{m \geq 1}$ if it satisfies Condition C2 with respect to the pair $((\sigma_m)_{m \geq 1},\bar{\Omega})$ for every $\bar{\Omega} \subseteq \Omega$ with $\xi(\bar{\Omega})>0$.
		
	\end{definition}
	
	Thus, we have the following approximation theorem.
	
	\begin{theorem}\label{ThmBayesianAccuPoints}
		Let $(\epsilon_m)_{m \geq 1}$ be a sequence of positive noise parameters converging to $0$ and let $(\sigma^{\circ}_m)_{m \geq 1}$ be the corresponding sequence of $\epsilon_m$-regularized Bayesian equilibria such that $\sigma^{\circ}_m \in \text{RBE}(\mathcal{G},{\xi},\epsilon_m)$ for all $m \geq 1$. Suppose that the regularizer $\textbf{v}$ satisfies either Condition C1, or Condition C2 with respect to the sequence of Bayesian strategies $(\sigma^{\circ}_m)_{m \geq 1}$. Then $(\sigma^{\circ}_m)_{m \geq 1}$ has accumulation points with respect to the weak topology and any such accumulation point is a Bayesian equilibrium of the original game $\mathcal{G}$. 
	\end{theorem}
	\begin{proof}
		See Appendix \ref{AppThmBayesianAccuPoints}. 
	\end{proof}

	\section{{Fundamental properties of {{RBBR}} learning dynamic}}\label{SectionFundamentalProp}
	
	In this section, we establish two fundamental properties of the regularized Bayesian best response dynamic, namely,  the existence of solutions from arbitrary initial conditions in $\Sigma$ and continuity of these solutions with respect to the initial states. We resort to the infinite dimensional Picard-Lindel\"off theorem and the Gr\"onwall's inequality (see \cite{zeidler1986nonlinear}) in order to prove the main results of this section. We begin with the following two definitions.
	
	\begin{definition}
		{A Bayesian population game $\mathcal{G}:\Sigma \times \Omega \rightarrow \mathbb{R}^n$ is said to be \textbf{strongly Lipschitz} if there exists $\kappa>0$ such that for ${\xi}$-{a.e} $\omega \in \Omega$,  $$\|\mathcal{G}(\sigma,\omega)-\mathcal{G}(\rho,\omega)\|\leq \kappa\vertiii{\sigma-\rho}, \hspace{3mm} \text{ for all } \sigma,\rho \in \Sigma.$$}
	\end{definition}
	
	\begin{definition}\label{DefSemiflow}
		{Consider an abstract differential equation  $${\phi}'(t)=\Lambda({\phi}(t))$$ on a Banach space $(\mathbb{X},\|\cdot\|_{\mathbb{X}})$. Suppose that a unique solution to the above differential equation exists for every initial condition ${\phi}(0)=\gamma$ and denote the solution with initial condition $\gamma$ by ${\phi}_\gamma(t)$. Then the \textit{semiflow of the dynamic} is the map ${\zeta}:\mathbb{X} \times [0,\infty) \rightarrow \mathbb{X}$ defined as ${\zeta}(\gamma,t)={\phi}_\gamma(t)$ for all $\gamma \in \mathbb{X}$ and $t \in [0,\infty)$.} 
	\end{definition}
	
	We now state the main theorem of this section.
	\begin{theorem}\label{ThmSolnExistence}
		Consider the type space $(\Omega,{\xi})$. Suppose that the Bayesian population game $\mathcal{G}:\Sigma \times \Omega \rightarrow \mathbb{R}^n$ is strongly Lipschitz. Then the following statements hold true:
		\begin{enumerate}[(i)]
			\item For every initial condition $\sigma_{0} \in \Sigma$, the {RBBR} dynamic admits a unique solution $(\sigma_{t})_{t \geq 0}$.
			\item The semiflow ${\zeta}:\widehat{\Sigma} \times [0,\infty) \to \widehat{\Sigma}$ of the {RBBR} dynamic is continuous in the initial conditions with respect to the strong topology on $\widehat{\Sigma}$.
		\end{enumerate}
	\end{theorem}
	\begin{proof}
		See Appendix \ref{AppThmSolnExistence}.
	\end{proof}	
	
	Following is a corollary, which states that the above result holds true for Bayesian aggregative population games in Example \ref{DefAggregateBayPopGame}. In what follows we require the following definition.
	
	\begin{definition}
		{A population game $\omega:\Delta \rightarrow \mathbb{R}^{n}$ is said to be Lipschitz if there exists a real number $\alpha >0$ such that} 
		$$\|\omega(\textbf{x})-\omega(\textbf{y})\| \leq \alpha \|\textbf{x}-\textbf{y}\|, \hspace{3mm} \text{ for all } \textbf{x},\textbf{y} \in \Delta.$$
	\end{definition}	
	
	Let $\text{Lip}_{\alpha}(\Delta)$ be the collection of $\alpha$-Lipschitz population games which map $\Delta$ into $\mathbb{R}^n$. Observe that under the uniform topology, $\text{Lip}_{\alpha}(\Delta)$ is a closed subset of $\mathcal{C}_n(\Delta)$. Thus $\text{Lip}_{\alpha}(\Delta)$ is a possible candidate for the support of the type measure $\xi$.
	
	\begin{corollary}\label{CorSolnExistence}

		Let $\mathcal{G}:\Sigma \times \mathcal{C}_n(\Delta) \rightarrow \mathbb{R}^n$ be a Bayesian aggregate population game. Suppose that there exists $\alpha>0$ such that $\xi(\text{Lip}_{\alpha}(\Delta))=1$. Then the conclusions of Theorem \ref{ThmSolnExistence} hold true.
		
	\end{corollary}
	\begin{proof}
		The proof of the corollary follows immediately and is left to the reader. 
	\end{proof}	

	We now focus on some applications of the RBBR dynamic in the view to obtain convergence results for certain classes of Bayesian population games. As in case of homogeneous population games with finite or continuum strategies, we consider here the class of Bayesian potential and Bayesian negative semi-definite games. To this end, we need to restrict our attention to norm-compact subsets of the space of Bayesian strategies which are forward invariant under the RBBR dynamic. We thus provide a sufficient condition to identify such subsets on which the the solution trajectories of the RBBR dynamic can lead to nice convergence results. This calls for some regularity on the Bayesian strategies which we present in the following definition.
	
	\begin{definition}\label{DefLipBayStr}
		{A Bayesian strategy $\sigma:\Omega \rightarrow \Delta$ is said to be {Lipschitz} if there exists $\alpha>0$ such that $$\|\sigma(\omega)-\sigma(\bar{\omega})\|\leq\alpha {d}_{\Omega}(\omega,\bar{\omega}), \hspace{3mm} \text{ for all } \omega,\bar{\omega} \in \Omega.$$} 
	\end{definition}
	
	Typically, this means that if the types of any two agents are close, then the induced belief of the the two agents are also close, which is somewhat reasonable. Next we define a subset of the space of Bayesian strategies as follows. Let $M \subseteq \Delta^{\circ}$ be a compact set. Define
	\begin{equation}
		\Sigma_{M}:=\{\sigma \in \Sigma : \sigma(\omega) \in M \text{ for } {\xi}-\text{a.e } \omega \in \Omega\}.
	\end{equation}
	
	In words, $\Sigma_M$ is the collection of those Bayesian strategies such that for $\xi$-almost all types $\omega \in \Omega$, the induced belief over the actions lies in $M$. In principle this means that the beliefs induced by Bayesian strategies in $\Sigma$ are almost surely bounded away from the boundary $\partial\Delta$ of the simplex $\Delta$. As a result, the induced belief assigns a positive probability to each action in $S$. To present the main result on convergence in Bayesian potential games, we restrict attention to the following subset of Bayesian strategies. Fix $\alpha>0$, $M \subseteq \Delta^{\circ}$ compact, and $\Xi \subseteq \Omega$. Define
	\begin{equation}\label{Sigma_Alpha_M_Xi}
		\Sigma^{\alpha}_{M}(\Xi):=\{\sigma \in \Sigma: \sigma \text{ is } \alpha\text{-Lipschitz on } \Xi\} \cap \Sigma_M.
	\end{equation}

	Thus, $\Sigma^{\alpha}_{M}(\Xi)$ is the collection of Bayesian strategies which are $\alpha$-Lipschitz on the subset $\Xi$ with range ${\xi}$-a.s contained in the compact set $M \subseteq \Delta^{\circ}$. In the following lemma (Lemma \ref{ThmSigmaNormCpt}), we show that $\Sigma^{\alpha}_{M}(\Xi)$ is relatively norm compact subject to suitable conditions on the type measure ${\xi}$. We also show that $\Sigma^{\alpha}_{M}(\Xi)$ is forward invariant under the RBBR dynamic under some further conditions on the game $\mathcal{G}$ and on the Lipschitz coefficient in Definition \ref{DefLipBayStr}. As a result, the lemma provides a candidate space to prove convergence results for Bayesian potential games and Bayesian negative semidefinite games.

	\begin{lemma}\label{ThmSigmaNormCpt}
		Suppose that there exists a compact set $\Xi \subseteq \Omega$ such that ${\xi}(\Xi)=1$. Then the following statements hold true:
		\begin{enumerate}[(i)]
			\item   For every $\alpha>0$ and every $M \subseteq \Delta^{\circ}$ compact, the subset $\Sigma^{\alpha}_{M}(\Xi)$ (as defined in (\ref{Sigma_Alpha_M_Xi})) is relatively norm compact in $\widehat{\Sigma}$.
			\item Let $\epsilon,\gamma>0$, where $\textbf{v}$ is steep and $\gamma$-strongly convex. Suppose that for every $\sigma \in \Sigma$, the following assumptions are satisfied:
			\begin{itemize}
				\item $\|\mathcal{G}(\sigma,\omega)\| \leq c$, for $\xi$-{a.e} $\omega \in \Xi$, for some $c>0$.
				\item the mapping $\omega \mapsto \mathcal{G}(\sigma,\omega)$ is $\lambda$-Lipschitz on $\Xi$, for some $\lambda>0$.\footnote{We bring to notice to the readers that this condition is not very restrictive in the sense that the Bayesian aggregative population games (Example \ref{DefAggregateBayPopGame}) satisfies this condition with $\lambda=1$.}
			\end{itemize}
			Then there exists a compact set ${M} \subseteq \Delta^{\circ}$ (depending on $c$) such that $\Sigma^{\alpha}_{{M}}(\Xi)$ is forward invariant under the {RBBR} learning dynamic, provided $\alpha \geq \frac{\lambda}{\epsilon\gamma}$.
		\end{enumerate}
	\end{lemma}
	\begin{proof}
		See Appendix \ref{AppThmNormCpt}.
	\end{proof}

	\section{{Convergence in Bayesian potential games}}\label{SectionBayesianPotential}
	In this section we define the notion of Bayesian potential games and study the convergence of the {RBBR} dynamic to the regularized Bayesian equilibria in such games.{\footnote{See \cite{monderer1996potential}, \cite{sandholm2010population}, and \cite{lahkar2015logit} for a complete exposition of homogeneous potential games.} {For this purpose, we first require the notion of Fr\'echet derivative, which is a generalization of the usual derivative in Banach spaces. Let $\mathbb{V}$ and $\mathbb{W}$ be normed linear spaces. An operator  $\phi:\mathbb{V} \rightarrow \mathbb{W}$ is Fr\'echet differentiable at $x \in \mathbb{V}$ if there exists a bounded linear operator ${A}_x:\mathbb{V} \rightarrow \mathbb{W}$ such that} 
		
		{$$\lim_{\|h\|_{\mathbb{V}} \rightarrow 0}\dfrac{\|\phi(x+h)-\phi(x)-{A}_x(h)\|_{\mathbb{W}}}{\|h\|_{\mathbb{V}}}=0.$$ Thus, if ${A}$ is the Fr\'echet derivative of $\phi$, then  $\phi(x+h)=\phi(x)+{A}_x(h)+\textbf{o}(\|h\|_{\mathbb{V}})$ for all $h$ near zero in $\mathbb{V}$.}

		{In this paper, we consider Fr\'echet differentiability of functions $\phi: \widehat{\Sigma} \rightarrow \mathbb{R}$. The  Fr\'echet derivative of $\phi$ is denoted by $\text{D}\phi$. Let $\mathbb{R}^{n}_{0}$ be the collection of $n$-dimensional vectors such that the sum of the components is zero, that is, $\mathbb{R}^{n}_{0}:=\{{x}=({x}_1,\ldots,{x}_n)\in\mathbb{R}^{n}:{x}_1+\cdots+{x}_n=0\}$. Let $\Sigma_0:=\{\widehat{\sigma}\in\widehat{\Sigma}:\widehat{\sigma}(\omega)\in\mathbb{R}^{n}_{0} \text{ for } \xi-\text{a.e } \omega \in \Omega\}$. For a Fr\'echet differentiable function $\phi:\widehat{\Sigma} \rightarrow \mathbb{R}$, the \textit{Fr\'echet gradient} $\nabla^{\textbf{F}}\phi: \widehat{\Sigma} \rightarrow \mathcal{L}^{\infty}(\Omega,{\xi},{\mathbb{R}^n}^{*})$ of $\phi$ is a bounded measurable map defined as follows: for every $\widehat{\sigma} \in \widehat{\Sigma}$ and every ${\sigma}_{0} \in \Sigma_0$},
		\begin{equation}\label{GradientExist}
			{\text{D}\phi[\widehat{\sigma}](\sigma_0):=\int_{\Omega}\langle\sigma_0,\nabla^{\textbf{F}}\phi(\widehat{\sigma})\rangle(\omega){\xi}(d\omega)=\int_{\Omega}\langle \nabla^{\textbf{F}}\phi[\widehat{\sigma}](\omega),\sigma_0(\omega)\rangle{\xi}(d\omega),}
		\end{equation}
		where the inner product in above integral after the first equality denotes the bilinear pairing as in Definition \ref{DefWeakTopOnSigma} and the inner product in the integral after the second equality is the usual inner product in $\mathbb{R}^n$.
		
		\begin{definition}[\cite{zusai2018evolutionary}]
			{A mapping ${\mathcal{G}}:{\Sigma} \times \Omega \rightarrow \mathbb{R}^{n}$ is called a \textbf{Bayesian potential game} if there exists a Fr\'echet differentiable map $\varphi : \widehat{\Sigma} \rightarrow \mathbb{R}$ such that for every ${\sigma} \in {\Sigma}$, we have $$\nabla^{\textbf{F}} \varphi[{\sigma}](\omega)={\mathcal{G}}({\sigma},\omega), \hspace{3mm} {\xi}-a.s.$$ The map $\varphi$ is called the Bayesian potential function of the Bayesian potential game ${\mathcal{G}}$.} 
		\end{definition}	

	We then introduce the notion of {entropy adjusted Bayesian potential function}. For a compact set $M \subseteq \Delta^{\circ}$, we then define a map $\widetilde{\textbf{v}}:\Sigma_M \rightarrow \mathbb{R}$ such that $$\widetilde{\textbf{v}}(\sigma):=\int_{\Omega}\textbf{v}(\sigma(\omega)){\xi}(d\omega),\hspace{3mm} \text{ for all } \sigma \in \Sigma_M.$$ The fact that the regularizer $\textbf{v}$ is strongly convex on $\Delta^{\circ}$ implies that $\textbf{v}$ is continuous on $M$. Thus $\textbf{v}$ is bounded on $M$ and hence $\widetilde{\textbf{v}}$ is well defined. This leads us to the following important definition.
	
	\begin{definition}
		{For $M \subseteq \Delta^{\circ}$ compact, the \textbf{entropy adjusted Bayesian potential function} is a mapping $\widetilde{\varphi}:\Sigma_M \rightarrow \mathbb{R}$ defined as $$\widetilde{\varphi}(\sigma)=\varphi(\sigma)-\widetilde{\textbf{v}}(\sigma), \hspace{3mm} \text{ for all } \sigma \in \Sigma_M.$$}  
	\end{definition}
	
	We require an additional result for establishing convergence in Bayesian potential games. We shall say that the RBBR dynamic satisfies \textbf{Bayesian positive correlation} with respect to the game $\mathcal{G}:\Sigma \times \Omega \to \mathbb{R}^n$ if $\int_{\Omega}\langle \mathcal{G}(\sigma,\omega), \dot{\sigma}(\omega) \rangle \xi(d\omega) \geq 0$ for all $\sigma \in \Sigma$ with equality only if $\dot{\sigma}=0$. For the rest of our analysis on Bayesian potential and Bayesian negative semidefinite games, we assume that the regularizer is a $\mathcal{C}^1$ function. The reason for such an assumption will be clear from the following definition. Define the virtual payoff vector at a Bayesian strategy $\sigma \in \Sigma$ for a player of type $\omega \in \Omega$ as
	\begin{equation}\label{virtualpayoff}
		\widetilde{\mathcal{G}}(\sigma,\omega)=\mathcal{G}(\sigma,\omega)-\nabla^{\textbf{G}}\widetilde{\textbf{v}}[\sigma](\omega).\footnote{The fact that $\textbf{v}$ is smooth in the interior $\Delta^{\circ}$ and bounded on $M$ implies that $\widetilde{\textbf{v}}$ id G\^ateaux differentiable}
	\end{equation}
	
	The following lemma states that the RBBR dynamic satisfies Bayesian positive correlation with respect to the virtual payoff $\widetilde{\mathcal{G}}$. The lemma imposes a further condition (\ref{Decomposable}) on the regularizer. Let $\mathcal{C}^1((0,1))$ denote the collection of all continuously differentiable functions $\theta:(0,1) \to \mathbb{R}$.
	\begin{definition}
		{A regularizer $\textbf{v}:\Delta \to \mathbb{R}^n \cup \{\infty\}$ is said to be \textbf{decomposable} if there exists a $\mathcal{C}^1$ function $\theta$ such that the regularizer $\textbf{v}$ can be expressed as}
		\begin{equation}\label{Decomposable}
			\textbf{v}(\textbf{x}):=\sum\limits_{i=1}^{n}\theta(\textbf{x}_i), \hspace{3mm} \text{ for all } \textbf{x} \in \Delta^{\circ}.\footnote{Such regularizers have been consider earlier, for instance in \cite{coucheney2015penalty}, \cite{mertikopoulos2016learning}.}  
		\end{equation} 
	\end{definition}
	The three entropies we consider, namely the Shannon-Gibbs entropy, the Tsallis entropy, and the Burg entropy, all satisfy condition (\ref{Decomposable}). Following is an intermediate lemma which we use to prove convergence in Bayesian potential games. 
	
	\begin{lemma}\label{LemPositiveCorrelation}
		Let $\mathcal{G}:\Sigma \times \Omega \to \mathbb{R}^n$ be a strongly Lipschitz Bayesian potential game. Assume that the conditions of part (ii) of Lemma \ref{ThmSigmaNormCpt} are satisfied. Suppose that the regularizer satisfies (\ref{Decomposable}). Then the following statements hold:
		\begin{enumerate}[(i)]
			\item Suppose that $\theta'$ is strictly increasing on $(0,1)$. Then for every $\epsilon>0$, the RBBR dynamic satisfies Bayesian positive correlation with respect to the virtual payoff $\widetilde{\mathcal{G}}$ as defined in (\ref{virtualpayoff}). 
			\item Consider the norm compact set $\Sigma^{\alpha}_M(\Xi)$ obtained in Lemma \ref{ThmSigmaNormCpt}. Let $\delta=(\delta_1,\ldots,\delta_n) \in (0,\frac{1}{2})^n$ be such that $M \subseteq [\delta_1,1-\delta_1] \times \cdots \times [\delta_n,1-\delta_n]$. 
			Suppose that $\theta'$ is strictly increasing on the interval $[\min_{i=1,\ldots,n}\delta_i, 1- \min_{i=1,\ldots,n}\delta_i]$. Then the RBBR dynamic satisfies positive correlation with respect to the virtual payoff $\widetilde{\mathcal{G}}$ defined in (\ref{virtualpayoff}).
		\end{enumerate}
	\end{lemma}
	
	\begin{proof}
		See Appendix \ref{AppLemPositiveCorrelation}.
	\end{proof}
	
	As we shall see in Theorem \ref{bpgc} below that $\widetilde{\varphi}$ restricted to $\Sigma^{\alpha}_M(\Xi)$ will act as a Lyapunov function for the RBBR dynamic when the underlying game is a Bayesian potential game. Also, the fact that $\Sigma^{\alpha}_{M}(\Xi)$ is relatively norm compact and forward invariant under the {RBBR} learning dynamic will allow us to use results from the theory of dynamical systems to prove the desired convergence of the solutions to the regularized Bayesian equilibrium.  In what follows, we present the main theorem of this section.  
	\begin{theorem}
		\label{bpgc}
		Suppose that the assumptions of Theorem \ref{ThmSigmaNormCpt}, part (ii) of Lemma \ref{ThmSigmaNormCpt} and Lemma \ref{LemPositiveCorrelation} are satisfied. Let $\mathcal{G}:{\Sigma} \times \Omega \rightarrow \mathbb{R}^{n}$ be a Bayesian potential game with Bayesian potential function $\varphi:\Sigma \rightarrow \mathbb{R}$ and let $\widetilde{\varphi}^{\alpha}_{M}:\Sigma^{\alpha}_M(\Xi) \rightarrow \mathbb{R}$ be the entropy adjusted Bayesian potential function restricted to $\Sigma^{\alpha}_M(\Xi)$. Then the following statements hold true:
		\begin{enumerate}[(i)]
			\item $\widetilde{\varphi}^{\alpha}_M$ increases weakly along every solution trajectory to {RBBR} learning dynamic that originates in $\Sigma^{\alpha}_M(\Xi)$ and  increases strictly across every non-stationary solution trajectory.
			\item The set of omega--limit points (in the strong topology) of any trajectory to the {RBBR} learning dynamic is a non-empty connected compact set of perturbed Bayesian equilibria. Moreover, such limit points are local maximizers of the entropy adjusted Bayesian potential function $\widetilde{\varphi}^{\alpha}_M$ on $\Sigma^{\alpha}_{M}(\Xi)$.
		\end{enumerate}
	\end{theorem}
	\begin{proof}
		See Appendix \ref{AppSectionPotential}.
	\end{proof}

	\section{{Convergence in Bayesian negative semidefinite games}}\label{SectionBayesianNegative}
	In this section we seek to define the notion of Bayesian negative semidefinite games, which is a generalization of negative semidefinite population games on finite strategy spaces to heterogeneous populations. We also generalize the notion of self defeating externalities to Bayesian self defeating externalities. Following \cite{cheung2014pairwise}, we show the equivalence between Bayesian negative semidefinite games and Bayesian self defeating externalities. Similar to \cite{hofbauer2007evolution} in finite strategy negative semidefinite games and \cite{lahkar2015logit} in continuum strategy negative semidefinite games under homogeneous set-up, we show uniqueness of regularized equilibrium. Finally, we conclude the section by proving a convergence result which shows that the solution to the {RBBR} learning dynamic with initial condition in norm compact forward invariant subsets of $\Sigma$ converges to the unique regularized Bayesian equilibrium under the strong topology on $\Sigma$, when the underlying game is a Bayesian negative semidefinite game. We begin with the following definition.
	
	\begin{definition}\label{DefBayNegSemGame}
		A mapping ${\mathcal{G}}:{\Sigma}\times \Omega \rightarrow \mathbb{R}^{n}$ is called a \textbf{Bayesian negative semidefinite} game if 
		\begin{equation}\label{DefBayNeg}
			\int_{\Omega}\langle\mathcal{G}(\sigma,\omega)-\mathcal{G}({\rho},\omega),\sigma(\omega)-{\rho}(\omega)\rangle{\xi}(d\omega) \leq 0 \hspace{3mm} \text{ for all } \sigma, {\rho} \in {\Sigma}. 
		\end{equation}
	\end{definition}

	\cite{hofbauer2009stable} calls such games stable games. The above definition is a generalization of stable games to heterogeneous populations. Below we provide some examples of games which are Bayesian negative semi-definite.
	
	\begin{example}[\textbf{Bayesian aggregative RPS game}]
		\textit{In this example, we define the notion of a Bayesian aggregative Rock-Paper-Scissor game. We will use the abbreviation BA-RPS for the same. In standard RPS games, the players are allowed three actions, namely, `rock (R)', `paper (P)' and `scissors (S)'. The P covers the R, the R smashes the S, and the S cuts to the P. Now consider an agent of type $\omega$. Suppose that the Bayesian strategy exercised by the underlying population is $\Sigma$. Then the payoff vector of the agent is given by $A_{\omega}\mathcal{E}(\sigma)$, where $A_{\omega}$ is the payoff matrix specific to agents of type $\omega$ and is defined as}
		\begin{equation}
			A_{\omega}= 
			\begin{pmatrix}
				0 & -a(\omega) & b(\omega)\\
				b(\omega) & 0 & -a(\omega)\\
				-a(\omega) & b(\omega) & 0
			\end{pmatrix}
		\end{equation} 
		with $a, b : \Omega \to [0,\infty)$ are measurable functions. Formally, a Bayesian aggregate RPS is a mapping $\mathcal{G}_{\text{BA-RPS}}:\Sigma \times \Omega \to \mathbb{R}^n$ such that for all $\sigma \in \Sigma$,
		\begin{equation}
			\mathcal{G}_{BA-RPS}(\sigma,\omega):=A_{\omega}\mathcal{E}(\sigma), \hspace{3mm} \text{ for all } \omega \in \Omega.
		\end{equation}
		
		\textit{A special case of the above game is called a standard BA-RPS game if $a(\omega)=b(\omega)$ for $\xi$-a.e $\omega \in \Omega$. Now suppose that $b(\omega)-a(\omega) \geq 0$ for $\xi$-a.e $\omega \in \Omega$, then from Definition \ref{DefBayNegSemGame} it follows that $\mathcal{G}_{\text{BA-RPS}}$ is a Bayesian aggregative RPS game.}\footnote{This follows computations similar to \cite{hofbauer2009stable}.} 
	\end{example}
	
	\begin{example}[\textbf{Bayesian symmetric zero-sum game}]
		\textit{For every $\omega \in \Omega$, let $A_{\omega}$ be a skew symmetric matrix, that is, $A^{ij}_{\omega}=-A^{ji}_{\omega}$ for all $i,j \in S$. We then define the Bayesian symmetric zero-sum game as a mapping $\mathcal{G}_{\text{sym}}:\Sigma \times \Omega \to \mathbb{R}$ such that for every $\sigma \in \Sigma$,}
		\begin{equation}
			\mathcal{G}_{\text{sym}}(\sigma,\omega):=A_{\omega}\mathcal{E}(\sigma), \hspace{3mm} \text{ for all } \omega \in \Omega.
		\end{equation}
		\textit{Thus, in such games, the inequality (\ref{DefBayNeg}) is satisfied and hence $\mathcal{G}_{\text{sym}}$ is a Bayesian negative semidefinite game.}
	\end{example}

	We now turn to investigate the equivalence between Bayesian negative semidefiniteness and Bayesian self-defeating externalities. Recall from Section \ref{SectionPerturbedBestResponse} that for a fixed $\omega \in \Omega$, the slice game $\mathcal{G}_{\omega}:\Sigma \rightarrow \mathbb{R}^n$ is defined as $\mathcal{G}_{\omega}(\sigma):=\mathcal{G}(\sigma,\omega)$, for all $\sigma \in \Sigma$. We require this in the definition of Bayesian self defeating externalities.
	\begin{definition}
		{A Bayesian population game ${\mathcal{G}}:{\Sigma}\times \Omega \rightarrow \mathbb{R}^{n}$ satisfies \textbf{Bayesian self defeating externalities} if the following two conditions are satisfied:}
		\begin{enumerate}[(i)]
			\item For every $\omega \in \Omega$, the slice game $\mathcal{G}_{\omega}:\widehat{\Sigma} \rightarrow \mathbb{R}^n$ is $\mathcal{C}^{1}$-Fr\'echet differentiable for ${\xi}$-a.e $\omega \in \Omega$, and
			\item $\int_{\Omega}\langle\text{D}\mathcal{G}_{\omega}[{\sigma}](\sigma_0),{\sigma}_0(\omega)\rangle{\xi}(d\omega) \leq 0$, for all $\sigma \in \Sigma$, and all ${\sigma}_0 \in {\Sigma}_0$. 
		\end{enumerate}
	\end{definition}
	{\cite{hofbauer2009stable} first proved that the notion of negative semidefiniteness and self-defeating externalities are equivalent in the case of finite strategy homogeneous population games, followed by \cite{cheung2014pairwise} who proved the equivalence in the context of continuum strategy homogeneous population games. Our next lemma states that indeed, the same holds true in the case of finite strategy Bayesian population games also. }
	
	\begin{proposition}
		\label{sde}
		A Bayesian population game $\mathcal{G}:\Sigma \times \Omega \rightarrow \mathbb{R}^{n}$ is Bayesian negative semidefinite if and only if it satisfies Bayesian self defeating externalities.
	\end{proposition}
	\begin{proof}
		See Appendix \ref{AppSectionBayesianNegative}.
	\end{proof}	
	Following is the result on the uniqueness of regularized Bayesian equilibrium for Bayesian negative semidefinite games. This result extends the uniqueness result of \cite{hofbauer2009stable} in the heterogeneous set-up for finitely many strategies.

	\begin{proposition}\label{LemBayNegSemGameUnique}
		Fix a noise parameter $\epsilon>0$. Suppose that $\mathcal{G}:\Sigma \times \Omega \to \mathbb{R}^n$ is a Bayesian negative semidefinite game which is $\mathcal{C}^1$ in the sense of Fr\'echet differentiability. Suppose that the regularizer $\textbf{v}$ satisfies $\langle d\nabla \textbf{v}_{\epsilon}[\textbf{x}+t\textbf{z}_0](\textbf{z}_0),\textbf{z}_0 \rangle \geq 0$, for all $t \in \mathbb{R}$, $\textbf{x} \in \Delta$ and $\textbf{z}_0 \in \mathbb{R}^n_0$ such that $\textbf{x}+t\textbf{z}_0 \in \Delta$.
		Then the game $\mathcal{G}$ has a unique regularized Bayesian equilibrium.
	\end{proposition}
	\begin{proof}
		See Appendix \ref{AppLemBayNegSemGameUnique}.
	\end{proof}
	
	Now, we proceed to study the convergence of Bayesian negative semidefinite games under the {RBBR} dynamic. To do so we again restrict the dynamic to the norm compact forward invariant subset $\Sigma^{\alpha}_{M}(\Xi)$ as defined in Section \ref{SectionBayesianPotential}. On this restricted space we define the appropriate Lyapunov function as a mapping $\psi^{\alpha}_{M,\epsilon}:\Sigma^{\alpha}_{M}(\Xi) \rightarrow \mathbb{R}$ such that for all $\sigma \in \Sigma^{\alpha}_{M}(\Xi)$,
	\begin{align}\label{BayNegDefLyapounov}
		\psi^{\alpha}_{M,\epsilon}(\sigma)&:=\Big[\int_{\Omega}\langle\mathcal{G}(\sigma,\omega),{{\beta}}_{\epsilon}[\sigma](\omega)\rangle{\xi}(d\omega)-\widetilde{\textbf{v}}_{\epsilon}({{\beta}}_{\epsilon}(\sigma))\Big] \nonumber \\
		&\hspace{3cm}-\Big[\int_{\Omega}\langle\mathcal{G}(\sigma,\omega),\sigma(\omega)\rangle{\xi}(d\omega)-\widetilde{\textbf{v}}_{\epsilon}(\sigma)\Big].
	\end{align}

	\begin{theorem}
		\label{bndc}
		Let ${\mathcal{G}}:{\Sigma}\times\Omega \rightarrow \mathbb{R}^{n}$ be a Bayesian negative semidefinite game such that the slice map $\mathcal{G}_{\omega}$ is $\mathcal{C}^{1}$-Fr\'echet differentiable in the norm topology for ${\xi}$-a.e $\omega \in \Omega$. Then, the Lyapunov function $\psi^{\alpha}_{M,\epsilon}:\Sigma^{\alpha}_{M}(\Xi) \rightarrow \mathbb{R}$ as defined in (\ref{BayNegDefLyapounov}) decreases monotonically to zero along every solution trajectory to the {RBBR} learning dynamic that originates in $\Sigma^{\alpha}_{M}(\Xi)$, and decreases strictly along every non stationary solutions. {Hence, every solution trajectory to the {RBBR} learning dynamic in $\Sigma^{\alpha}_{M}(\Xi)$ converges with respect to the norm topology on $\Sigma^{\alpha}_{M}(\Xi)$ to the regularized Bayesian equilibrium.} 
	\end{theorem}
	
	\begin{proof}
		See Appendix \ref{AppBayNegDefConv}.
	\end{proof}	

	\begin{appendix}
		\section{Appendix}
		
		\subsection{{Preliminaries of B\"ochner Spaces}}\label{SubSecWeakTopRelWeakCpt}
		In this section, we provide a brief exposition on the preliminary theory of B\"ochner spaces which we use in our proofs.
		\begin{definition}
			Let $({\Omega},{\mathcal{A}},{\mu})$ be a probability space and let $\mathbb{X}$ be a Banach space.
			A function $f:\Omega \to \mathbb{X}$ is simple if there exists $x_1,x_2,\ldots,x_k \in \mathbb{X}$ and $E_1,E_2,\ldots,E_k \in \mathcal{A}$ such that $f(\omega)=\sum_{1 \leq i \leq k}x_i\chi_{E_i}(\omega)$ for all $\omega \in \Omega$, where $\chi_{E_i}$ denotes the indicator function on $E_i$ for all $i=1,2,\ldots,k$.
		\end{definition}
		\begin{definition}\label{Def_mu_mble}
			Let $({\Omega},{\mathcal{A}},{\mu})$ be a probability space and let $\mathbb{X}$ be a Banach space.
			A function $f:\Omega \to \mathbb{X}$ is $\mu$-measurable if there exists a sequence of simple functions $(f_n)_{n \geq 1}$ with $\lim_{n \to \infty}\|f_n-f\|_{\mathbb{X}}=0$, $\mu$-almost everywhere.
		\end{definition}
		\begin{definition}\label{DefWeaklyMuMeasurable}
			Let $({\Omega},{\mathcal{A}},{\mu})$ be a probability space and let $\mathbb{X}$ be a Banach space.	
			A function $f:\Omega \to \mathbb{X}$ is weakly $\mu$-measurable if for every $x^{*} \in \mathbb{X}^{*}$, the numerical function $x^{*}f$ is $\mu$-measurable.
		\end{definition}	
		
		Let $\mathcal{L}^{p}(\Omega,{\mu},\mathbb{X})$ denote the space of all (${\mu}$-a.s.) equivalence classes of B\"ochner integrable functions $f:{\Omega} \rightarrow \mathbb{X}$ with norm defined as $$\|f\|:=\int_{{\Omega}}\|f(\omega)\|^{p}_{\mathbb{X}}{\mu}(d{\omega})<\infty.$$ 
		The the B\"ochner space  $\mathcal{L}^{p}(\Omega,{\mu},\mathbb{X})$ is a generalization of standard $\mathcal{L}^{p}({\mu})$ spaces.\footnote{See \cite{diestel1978vector} and \cite{hille1996functional} for further details on B\"ochner spaces.} Note that the space of integrable signed Bayesian strategies as defined in (\ref{IntegrableSignedBayStr}) can be viewed as the B\"ochner space $\mathcal{L}^{1}(\Omega,{\mu},\mathbb{X})$, where ${\Omega}=\Omega$, ${\mathcal{A}}=\mathcal{B}_{\Omega}$, ${\mu}={\xi}$, and $\mathbb{X}=\mathbb{R}^{n}$.
		
		We now introduce the notion of weak topology on $\mathcal{L}^{p}(\Omega,{\mu},\mathbb{X})$ for $1 \leq p <\infty$. Let $f \in \mathcal{L}^{p}(\Omega,{\mu},\mathbb{X})$ and $g \in \mathcal{L}^{q}(\Omega,{\mu},\mathbb{X}^{*})$, where $p^{-1}+q^{-1}=1$.\footnote{Here $\mathbb{X}^{*}$ denotes the dual space of $\mathbb{X}$.} Define ${\langle} f,g{\rangle}({\omega}):=g({\omega})(f({\omega}))$ for all ${\omega} \in \Omega$.\footnote{Suppose that the Banach space $\mathbb{X}$ is reflexive. Then for $1 \leq p <\infty$, the dual of $\mathcal{L}^{p}(\widetilde{\mu},\mathbb{X})$ is identified by $\mathcal{L}^{q}(\widetilde{\mu},\mathbb{X}^{*})$, where $p^{-1}+q^{-1}=1$. That is, $\mathcal{L}^{p}(\widetilde{\mu},\mathbb{X})^{*} \cong \mathcal{L}^{q}(\widetilde{\mu},\mathbb{X}^{*})$, for $1 \leq p< \infty$. See \cite{diestel1978vector} for a complete exposition on weak topology on B\"ochner spaces.} 
		\begin{definition}\label{DefWeakTopOnSigma}
			Let $({\Omega},{\mathcal{A}},{\mu})$ be a probability space and let $\mathbb{X}$ be a Banach space. The weak topology on $ \mathcal{L}^{p}(\Omega,{\mu},\mathbb{X})$ is the topology induced by the convergence:
			$$f_{n} \rightarrow^{w} f \iff \int_{{\Omega}}\langle f_{n},g\rangle({\omega}){\mu}(d{\omega}) \rightarrow \int_{{\Omega}}\langle f,g\rangle({\omega}){\mu}(d{\omega}), \hspace{2mm} \text{ for all } g \in \mathcal{L}^{q}(\Omega,{\mu},\mathbb{X}^{*}).$$
		\end{definition}
		The following three definitions are required in the context of Theorem \ref{ThmSigmaNormCpt}.
		\begin{definition}\label{DefAvg}
			Let $f \in \mathcal{L}^{1}(\Omega,{\mu},\mathbb{X})$ and $A \in \mathcal{A}$. Then the average value of $f$ over $A$ is defined as $$\text{avg}(f;A):=\dfrac{1}{{\mu}(A)}\int_A f(\omega){\mu}(d\omega).$$
		\end{definition}
		
		\begin{definition}\label{DefBocce}
			Let $f \in \mathcal{L}^{1}(\Omega,{\mu},\mathbb{X})$ and $A \in \mathcal{A}$. Then the Bocce oscillation of $f$ over $A$ is defined as $$\text{Bocce-osc}(f;A):=\dfrac{1}{{\mu}(A)}\int_{A}\|f(\omega)-\text{avg}(f;A)\|_{\mathbb{X}}{\mu}(d\omega).$$
		\end{definition}
		
		\begin{definition}\label{DefSmallBocce}
			A subset $\mathbb{K}$ of $ \mathcal{L}^{1}(\Omega,{\mu},\mathbb{X})$ is said to satisfy small Bocce oscillation if for every $\epsilon>0$, there exists a finite measurable partition $\{A_1,\ldots,A_p\}\subseteq\Omega$ such that for all $f \in \mathbb{K}$, we have $$\sum_{1 \leq i \leq p}{\mu}(A_i)\text{Bocce-osc}(f;A_i)<\epsilon.$$ 
		\end{definition}

		\subsection{Proof of proposition \ref{Prop-OriGame-SliceGame}}\label{App-Prop-OriGame-SliceGame}
		
		We first prove the `only-if' part. 
		Note that the definition of Bayesian best response (\ref{EqBayesianBestResponse}) can be equivalently written as
		\begin{equation*}
			{\beta}(\sigma):=\Big\{\rho \in \Sigma: \rho(\omega) \in \arg\max_{\textbf{y} \in \Delta} \langle \textbf{y}, \mathcal{G}(\sigma,\omega)\rangle \text{ for } {\xi}-\text{a.e } \omega \in \Omega \Big\}.
		\end{equation*}   
		Now suppose that $\sigma^{\circ}$ is a Bayesian equilibrium of $\mathcal{G}$. Then we must have $\sigma^{\circ} \in {\beta}(\sigma^{\circ})$ which implies by the definition of Bayesian equilibrium that $$\sigma^{\circ}(\omega) \in \arg\max_{\textbf{y} \in \Delta} \langle \textbf{y}, \mathcal{G}(\sigma^{\circ},\omega)\rangle, \hspace{3mm} \text{ for } {\xi}-\text{a.e } \omega \in \Omega.$$ As a result, by the original definition of Bayesian best response correspondence, we have $\sigma^{\circ}(\omega) \in {\beta}[\sigma^{\circ}](\omega)$ for ${\xi}$-a.e $\omega \in \Omega$. Finally, by the definition of $\omega$-sliced Bayesian best response (\ref{EqSlicedBayesianBestResponse}), we have that $\sigma^{\circ}(\omega) \in \beta_{\omega}(\sigma^{\circ})$ for $\xi$-a.e $\omega \in \Omega$. This completes the proof of the only-if part.
		
		We now proceed with the proof of `if-part'. So let us suppose that $\sigma^{\circ}$ is an $\omega$-sliced Bayesian equilibrium of the game $\mathcal{G}_{\omega}$ for $\xi$-a.e $\omega \in \Omega$. Similarly, back tracing as in the proof of only if part, we have $\sigma^{\circ}$ is a Bayesian equilibrium of $\mathcal{G}$.

		\subsection{Proof of Proposition \ref{Prop}}\label{AppProp}
		To prove the proposition, we resort to the following theorem.
		\begin{theorem}[\textbf{Pettis's Measurability Theorem} (\cite{diestel1978vector})]\label{ThmPettis}
			Let $(\widehat{\Omega},\mathcal{A},\mu)$ be a probability space and let $\mathbb{X}$ be a Banach space. A function $f:\widehat{\Omega} \to \mathbb{X}$ is $\mu$-measurable if and only if 
			\begin{enumerate}[(i)]
				\item $f$ is $\mu$-essentially separable valued, that is, there exists $E \in \mathcal{A}$ with $\mu(E)=0$ and such that $f(\widehat{\Omega}\setminus E)$ is (norm) separable subset of $\mathbb{X}$, and
				\item $f$ is weakly $\mu$-measurable.
			\end{enumerate}
		\end{theorem}
		We now proceed with the proof of the proposition. We apply Theorem \ref{ThmPettis} with $\widehat{\Omega}=\Omega$, $\mu={\xi}$, and $\mathbb{X}=\mathbb{R}^n$. To show that the mapping $\sigma \mapsto {\beta}_{\epsilon}(\sigma)$ is well-defined, we need to show that for every $\sigma \in \Sigma$, ${\beta}_{\epsilon}(\sigma) \in \Sigma$. In other words, we need to show that ${\beta}_{\epsilon}[\sigma]$ is ${\xi}$-measurable. We thereby verify the conditions of Theorem \ref{ThmPettis}. Note that by definition, ${\beta}[\sigma]:\Omega\to\Delta$. Since $\Delta$ is compact, it is separable. This verifies condition (i). To verify condition (ii), we need to show that $\beta_{\epsilon}(\sigma):\Omega \to \Delta$ is weakly $\xi$-measurable in the sense of Definition \ref{DefWeaklyMuMeasurable}. Thus, fix $\textbf{u} \in \mathbb{R}^n$. We need to show that the map $\textbf{u}\circ{\beta}_{\epsilon}[\sigma]:\Omega\to\mathbb{R}$ defined by $\textbf{u}\circ{\beta}_{\epsilon}[\sigma](\omega):=\langle \textbf{u},{\beta}_{\epsilon}[\sigma](\omega)\rangle$ for all $\omega \in \Omega$ is weakly ${\xi}$-measurable. Since $\textbf{v}$ is a strongly convex function, we it follows that ${\beta}_{\epsilon}[\sigma](\omega)=\nabla \textbf{v}^{*}_{\epsilon}(\mathcal{G}(\sigma,\omega))$, in which case we have $$\langle \textbf{u},{\beta}_{\epsilon}[\sigma](\omega)\rangle=\langle \textbf{u},\nabla \textbf{v}^{*}_{\epsilon}(\mathcal{G}(\sigma,\omega))\rangle,\hspace{3mm} \text{ for all } \omega \in \Omega.$$ Also since $\mathcal{G}$ is a Bayesian population game, it follows from the definition that the map $\omega \mapsto \mathcal{G}(\sigma,\omega)$ is measurable. Also, it follows from \cite{rockafellar1970convex} that $\nabla \textbf{v}^{*}_{\epsilon}$ is $1/{{\gamma}\epsilon}$-Lipschitz (hence continuous). Therefore the map $\omega \mapsto \nabla \textbf{v}^{*}_{\epsilon}(\mathcal{G}(\sigma,\omega))$ is measurable. Thus $\omega \mapsto \langle \textbf{u},\nabla \textbf{v}^{*}_{\epsilon}(\mathcal{G}(\sigma,\omega))\rangle$ is measurable. As a result, by simple function approximation theorem, it follows that the map $\omega \mapsto \textbf{u}\circ{\beta}_{\epsilon}[\sigma](\omega)$ is weakly ${\xi}$-measurable. Since $\textbf{u} \in \mathbb{R}^n$ is arbitrary, condition (ii) is verified. This concludes the proof of Proposition \ref{ThmPettis}.

		\subsection{{Proof of Theorem \ref{ThmFixedPt}}}\label{AppFixedPoint}
		The proof of Theorem \ref{ThmFixedPt} is an application of Brouwer-Schauder-Tychonoff fixed point theorem (see \cite{charalambos2013infinite}) for details). We provide the details for completeness.
		\begin{theorem}[\textbf{Brouwer-Schauder-Tychonoff FPT}]
			\label{bst} Let $\mathbb{V}$ be a non-empty compact convex subset of a locally convex Hausdorff space, and let $f:\mathbb{V} \to \mathbb{V}$ be a continuous function. Then the set of fixed points of $f$ is compact and nonempty.
		\end{theorem}
		In order to apply Theorem \ref{bst}, recall from (\ref{IntegrableSignedBayStr}) that the space of integrable signed Bayesian strategies $\widehat{\Sigma}$ is identified by the B\"ochner space $\mathcal{L}^{1}(\Omega,{\xi},\mathbb{R}^{n})$. Since $\widehat{\Sigma}$ is endowed with the weak topology, it follows from \cite{kesavan2009functional} that it is a Hausdorff topological space. It also follows from the definition that $\widehat{\Sigma}$ is locally convex. Therefore, under the weak topology, the space of integrable signed Bayesian strategies is a locally convex Hausdorff topological space. We now proceed to show that the space of Bayesian strategies $\Sigma$ is compact under the weak topology. The fact that $\Sigma$ is convex is obvious.
		\begin{lemma}\label{LemSigmaCpt}
			The space of Bayesian strategies $\Sigma$ is compact under the weak topology.
		\end{lemma}
		\begin{proof}
			The proof of Lemma \ref{LemSigmaCpt} is an application of Dunford's Theorem (see \cite{diestel1978vector} for details) which characterizes relatively weakly compact subsets of $\mathcal{L}^{1}(\Omega,{\xi},\mathbb{R}^n)$. 
			\begin{theorem}[\textbf{Dunford's Theorem} (\cite{diestel1978vector})]\label{ThmDunford}
				Let $(\widehat{\Omega},{\mathcal{A}},{\mu})$ be a finite measure space and $\mathbb{X}$ be a Banach space such that both $\mathbb{X}$ and $\mathbb{X}^{*}$ satisfy the Radon-Nikodym property.\footnote{See \cite{diestel1978vector}, Chap III for the definition of Radon-Nikodym property. It is known that if $\mathbb{X}$ is a reflexive Banach space, then $\mathbb{X}$ has the Radon-Nikodym property. In our case, $\mathbb{X}=\mathbb{R}^n$. Thus $\mathbb{X}$ follows the Radon Nikodym property.} A subset $\mathbb{K}$ of $\mathcal{L}^{1}({\mu},\mathbb{X})$ is relatively weakly compact if 
				\begin{enumerate}[(i)]
					\item the subset $\mathbb{K}$ is bounded,
					\vspace{-2mm}
					\item the subset $\mathbb{K}$ is uniformly integrable, and 
					\vspace{-2mm}
					\item for each $E \in \mathcal{A}$, the set $\Big\{\bigintsss_{E}f({\omega}) d{\mu}(d{\omega}) : f \in \mathbb{K}\Big\}$ is relatively weakly compact.
				\end{enumerate}
			\end{theorem}
			We apply Theorem \ref{ThmDunford} with $\widehat{\Omega}=\Omega$, $\mu={\xi}$, $\mathbb{X}=\mathbb{R}^n$, and $\mathbb{K}=\Sigma$. We need to show that the space of Bayesian strategies $\Sigma$ satisfies the conditions stated in Theorem \ref{ThmDunford}. Note that $\|\sigma(\omega)\|\leq\sqrt{n}$, for all $\omega \in \Omega$. Therefore, we have $$\vertiii{\sigma}=\int_{\Omega}\|\sigma(\omega)\|{\xi}(d\omega)\leq\sqrt{n},$$ which in particular implies that $\sup_{\sigma \in \Sigma}\vertiii{\sigma}\leq \sqrt{n}$. Therefore the subset $\Sigma$ is bounded. 
			
			The fact that $\Sigma$ is uniformly integrable follows since every Bayesian strategy $\sigma \in \Sigma$ is uniformly bounded by $\sqrt{n}$, for every $\omega \in \Omega$, in which case we have $$\lim_{\epsilon \rightarrow \infty}\sup_{\sigma \in \Sigma}\int_{\|\sigma\|>\epsilon}\|\sigma(\omega)\|{\xi}(d\omega)\leq\sqrt{n}\lim_{\epsilon \rightarrow \infty}\sup_{\sigma \in \Sigma}{\xi}(\|\sigma\|>\epsilon) = 0.$$ 
			
			Finally we verify condition (iii) of Theorem \ref{ThmDunford}. To this end, let $E \in \mathcal{B}_{\Omega}$. Consider the set $$\mathcal{I}_{E}:=\Big\{\int_{E}\sigma(\omega){\xi}(d\omega):\sigma \in \Sigma\Big\}.$$ The fact that $\Delta$ is a compact convex subset of $\mathbb{R}^{n}$ and that $\text{Range}(\sigma) \subseteq \Delta$ for all $\sigma \in \Sigma$, together imply that $\mathcal{I}_{E}$ is relatively compact. Since $E$ is arbitrary, condition (iii) of Theorem \ref{ThmDunford} is satisfied. 
			
			This proves that $\Sigma$ is relatively compact under the weak topology. However since $\Sigma$ is itself closed under the weak topology, it is compact. This concludes the proof of lemma \ref{LemSigmaCpt}.
		\end{proof}
		
		To conclude the proof of Theorem \ref{ThmFixedPt}, we need to show that the mapping $\sigma \mapsto {\beta}_{\epsilon}(\sigma)$ is continuous in the weak topology. We show this in the following lemma.
		\begin{lemma}\label{LemBRWeakCont}
			The mapping $\sigma \mapsto {\beta}_{\epsilon}(\sigma)$ is continuous in the weak topology. 	
		\end{lemma}
		\begin{proof}
			To prove the lemma, consider a sequence $(\sigma_n)_{n \geq 1} \subseteq \Sigma$, such that $\sigma_n \rightarrow^{w} \sigma$. We need to show that ${{\beta}}_{\epsilon}(\sigma_n) \rightarrow^{w} {{\beta}}_{\epsilon}(\sigma)$. Since the regularizer is a ${\gamma}$-strongly convex function and $\mathcal{G}$ is a weakly continuous Bayesian population game, there exists $\bar{\Omega}\subseteq\Omega$ with ${\xi}({\bar{\Omega}})=1$ such that for every fixed $\omega \in \bar{\Omega}$, we have that
			\begin{align*}
				\|{{\beta}}_{\epsilon}[\sigma_n](\omega)-{{\beta}}_{\epsilon}[\sigma](\omega)\|&=\|\nabla \textbf{v}^{*}_{\epsilon}(\mathcal{G}(\sigma_n,\omega))-\nabla \textbf{v}^{*}_{\epsilon}(\mathcal{G}(\sigma,\omega))\|\\
				&\leq \dfrac{1}{\epsilon{\gamma}}\|\mathcal{G}(\sigma_n,\omega)-\mathcal{G}(\sigma,\omega)\|\rightarrow 0.
			\end{align*}
			This proves the fact that ${{\beta}}_{\epsilon}[\sigma_n](\omega)\rightarrow{{\beta}}_{\epsilon}[\sigma](\omega)$ for every $\omega \in \bar{\Omega}$. Now let $g \in \mathcal{L}^{\infty}(\Omega,{\xi},{\mathbb{R}^{n}}^{*})$. This implies that $g(\omega)$ is a bounded linear functional from $\mathbb{R}^n$ to $\mathbb{R}$ for every $\omega \in \Omega$. Therefore, for every $\omega \in \Omega$, we have $g[\omega](\textbf{x}_n) \rightarrow g[\omega](\textbf{x})$, whenever $\textbf{x}_n \rightarrow \textbf{x}$. Since ${{\beta}}_{\epsilon}[\sigma_n](\omega)\rightarrow{{\beta}}_{\epsilon}[\sigma](\omega)$ for every $\omega \in \bar{\Omega}$, we have that 
			\begin{align*}
				\int_{\Omega}\langle {{\beta}}_{\epsilon}(\sigma_n),g\rangle(\omega){\xi}(d\omega)&=\int_{\Omega}g(\omega)({{\beta}}_{\epsilon}[\sigma_n](\omega)){\xi}(d\omega)\\
				&=\int_{\bar{\Omega}}\langle g(\omega),{{\beta}}_{\epsilon}[\sigma_n](\omega) \rangle {\xi}(d\omega)\\
				&\rightarrow \int_{\bar{\Omega}}\langle g(\omega),{{\beta}}_{\epsilon}[\sigma](\omega) \rangle {\xi}(d\omega)\\
				&=\int_{\Omega}g(\omega)({{\beta}}_{\epsilon}[\sigma](\omega)){\xi}(d\omega)\\
				&=\int_{\Omega}\langle {{\beta}}_{\epsilon}(\sigma),g\rangle(\omega){\xi}(d\omega).
			\end{align*}
			Since $g \in \mathcal{L}^{\infty}(\Omega,\xi,{\mathbb{R}^n}^*)$ is arbitrary, this proves that the mapping $\sigma \mapsto {{\beta}}_{\epsilon}(\sigma)$ is continuous in the weak topology on $\Sigma$ which completes the proof of Lemma \ref{LemBRWeakCont}.
		\end{proof}
		
		Since $\Sigma$ is a non empty compact convex subset of $\widehat{\Sigma}$ and ${{\beta}}_{\epsilon}$ is a continuous map, the conditions of Theorem \ref{bst} are satisfied and the existence of a fixed point is established. This concludes the proof of Theorem \ref{ThmFixedPt}.

		\subsection{Proof of Corollary \ref{CorBayEqExist}}\label{AppCorBayEqExist}
		
		In order to prove the corollary, it is enough to show that the aggregate Bayesian population game $\mathcal{G}$ is weakly continuous. Note that from the definition of $\mathcal{G}$ it is enough to show that $\mathcal{E}$ is continuous. Let $\sigma_n \rightarrow \sigma$. We show that $\mathcal{E}(\sigma_n) \rightarrow \mathcal{E}(\sigma)$. To this end, fix $i \in S$. Define the projection map ${\pi}_{i} : \mathcal{C}_n(\Delta) \rightarrow {\mathbb{R}^{n}}^{*}$, such that for every $\omega \in \mathcal{C}_n(\Delta)$, $${\pi}_{i}[\omega](\textbf{x}):=\textbf{x}_{i}, \hspace{3mm} \text{ for all } \textbf{x} \in \mathbb{R}^n.$$ Note that ${\pi}_{i} \in \mathcal{L}^{\infty}(\mathcal{C}_n(\Delta),{\xi},{\mathbb{R}^{n}}^{*})$ for all $i \in S$. Since $\sigma_{n} \to^{w} \sigma$, we have from the definition of weak topology that 
		\begin{align*}
			\mathcal{E}^{i}(\sigma_n)&=\int_{\Omega}\sigma^{i}_{n}(\omega){\xi}(d\omega)\\
			&=\int_{\Omega}{\pi}_{i}(\omega)(\sigma_n(\omega)){\xi}(d\omega)\\
			&=\int_{\Omega}\langle \sigma_n,{\pi}_{i}\rangle(\omega){\xi}(d\omega)\\
			&\rightarrow \int_{\Omega}\langle \sigma,{\pi}_{i}\rangle(\omega){\xi}(d\omega)\\
			&=\int_{\Omega}{\pi}_{i}(\omega)(\sigma(\omega)){\xi}(d\omega)\\
			&=\int_{\Omega}\sigma^{i}(\omega){\xi}(d\omega)\\
			&=\mathcal{E}^{i}(\sigma).
		\end{align*}
		This proves that $\mathcal{E}(\sigma_n) \rightarrow \mathcal{E}(\sigma)$. Thus $\sigma_n \to^{w} \sigma$ implies that $\mathcal{G}(\sigma_n,\omega)=\omega(\mathcal{E}(\sigma_n)) \to \omega(\mathcal{E}(\sigma))=\mathcal{G}(\sigma,\omega)$ for all $\omega \in \mathcal{C}_n(\Delta)$. Hence $\mathcal{G}$ is weakly continuous.
		This concludes the proof of Corollary \ref{CorBayEqExist}.

		\subsection{Proof of Theorem \ref{ThmBayesianAccuPoints}}\label{AppThmBayesianAccuPoints}
		
		Let $(\epsilon_m)_{m \geq 1}$ be an arbitrary sequence of positive noise parameters such that $\epsilon_m \downarrow 0$. For every $m \geq 1$, let $\sigma^{\circ}_m$ denote the corresponding $\epsilon_m$-regularized Bayesian equilibrium, the existence of which is guaranteed by Theorem \ref{ThmFixedPt}. Note that by Lemma \ref{LemSigmaCpt}, the space of Bayesian strategies $\Sigma$ is compact relative to the weak topology. Thus there exists a subsequence $(\sigma_{m_k})_{k \geq 1} \subseteq \Sigma$ such that $\sigma^{\circ}_{m_k} \to^{w} \sigma^{\circ}$, for some $\sigma^{\circ} \in \Sigma$. We need to show that $\sigma^{\circ}$ is a Bayesian equilibrium of the game $\mathcal{G}$. We first prove the following lemma.

		\begin{lemma}\label{LemOriginalGame-SlicedGame}
			Fix a noise parameter $\epsilon>0$. Then $\sigma^{\circ}_{\epsilon}$ is a regularized Bayesian equilibrium of the game $\mathcal{G}$ iff $\sigma^{\circ}_{\epsilon}(\omega)$ is a regularized equilibrium of the $\omega$-sliced game $\mathcal{G}_{\omega}$ for ${\xi}$-a.e $\omega \in \Omega$.
		\end{lemma}
		
		\begin{proof}
			The proof of the lemma follows using the definitions of $\beta_{\epsilon}$ and $\beta^{\epsilon}_{\omega}$ and is left to the reader.
		\end{proof}
		
		We now proceed with the proof of the theorem. We shall prove this by the method of contradiction. Suppose on the contrary that $\sigma^{\circ}$ is not a Bayesian equilibrium. Then by definition, there exists $\bar{\Omega} \subseteq \Omega$ with ${\xi}(\bar{\Omega})>0$ such that $\sigma^{\circ}(\omega) \notin {\beta}[\sigma^{\circ}](\omega)$ for all $\omega \in \bar{\Omega}$, which in principle implies that $\sigma^{\circ}(\omega) \notin \arg\max_{\textbf{y} \in \Delta}\langle \textbf{y}, \mathcal{G}(\sigma^{\circ},\omega)\rangle$ for all $\omega \in \bar{\Omega}$. Thus, there exists $\textbf{i},\textbf{j}:\bar{\Omega}\to S$ such that for every $\omega \in \bar{\Omega}$, there exists $\textbf{j}({\omega}) \in S$, $\textbf{i}(\omega) \in \text{supp}(\sigma^{\circ}(\omega))$ and an $\eta(\omega)>0$ satisfying the following inequality:
		$$\mathcal{G}^{\textbf{j}(\omega)}(\sigma^{\circ},\omega) > 2\eta(\omega)+ \mathcal{G}^{\textbf{i}(\omega)}(\sigma^{\circ},\omega).$$
		Since $\mathcal{G}$ is a weakly continuous population game, there exists for every $\omega \in \bar{\Omega}$, a $K(\omega) \geq 1$ sufficiently large enough such that
		\begin{equation}\label{Contra1}
			\mathcal{G}^{\textbf{j}(\omega)}(\sigma^{\circ}_{m_k},\omega) > \eta(\omega)+ \mathcal{G}^{\textbf{i}(\omega)}(\sigma^{\circ}_{m_k},\omega), \hspace{3mm} \text{ for sufficiently large } k \geq K(\omega).
		\end{equation}
		
		We now make the following observation. By our assumptions, we have that $\sigma^{\circ}_{m_k} \to^{w} \sigma^{\circ}$. We show that $\sigma^{\circ}_{m_k}(\omega) \to \sigma^{\circ}(\omega)$ as $k \to \infty$ for every $\omega \in \Omega$. To prove the claim, fix an $\omega \in \Omega$ and $i \in S$. Define a function $g^i_{\omega}:\Omega \to \mathbb{R}^n$ as $g^i_{\omega}(\bar{\omega}):=\textbf{e}_i\mathbb{1}_{\omega}(\bar{\omega})$, for all $\bar{\omega} \in \Omega$. It then follows that $\|g^{i}_{\omega}\|_{\infty}=1$ and that $g^{i}_{\omega} \in \mathcal{L}^{\infty}(\Omega,\xi,{\mathbb{R}^{n}}^*)$, for every $\omega \in \Omega$ and every $i \in S$. Since $\sigma^{\circ}_{m_k} \to \sigma^{\circ}$, we have by definition of weak topology that,
		\begin{align*}
			\sigma^{\circ,i}_{m_k}(\omega)&=\langle \sigma^{\circ}_{m_k}(\omega), \textbf{e}_i \rangle\\
			&=\int_{\Omega}\langle \sigma^{\circ}_{m_k}, g^{i}_{\omega}\rangle (\bar{\omega}) \xi(d\bar{\omega})\\
			& \to \int_{\Omega}\langle \sigma^{\circ}, g^{i}_{\omega}\rangle (\bar{\omega}) \xi(d\bar{\omega})\\
			&=\langle \sigma^{\circ}(\omega), \textbf{e}_i \rangle\\
			&=\sigma^{\circ,i}(\omega).
		\end{align*} 

		Since $\textbf{i}(\omega) \in \text{supp}(\sigma^{\circ}(\omega))$ for all $\omega \in \bar{\Omega}$, we must have that $\sigma^{\circ,\textbf{i}(\omega)}(\omega)>0$ for all $\omega \in \bar{\Omega}$. However, we show that this cannot happen. Thus, in order to arrive at a contradiction, we show that $\sigma^{\circ,\textbf{i}(\omega)}_{m_k}(\omega) \to 0$ for some $\omega \in \bar{\Omega}$. We show this separately for regularizers satisfying Condition C1 followed by regularizers satisfying Condition C2.

		\begin{proof}[Proof for regularizers satisfying Condition C1]
			Proceeding, first let us suppose that the regularizer satisfies Condition C1 (Definition \ref{ConditionC1}). We prove a much stronger statement, that is, we show that $\sigma^{\circ,\textbf{i}(\omega)}_{m_k}(\omega) \to 0$ for all $\omega \in \bar{\Omega}$. To this end, fix $\omega \in \bar{\Omega}$. Then by (\ref{Contra1}), there exists $\textbf{i}(\omega), \textbf{j}(\omega) \in S$ such that $\textbf{u}^{\textbf{j}(\omega)}_{m_k}:=\mathcal{G}^{\textbf{j}(\omega)}(\sigma^{\circ}_{m_k},\omega)>\mathcal{G}^{\textbf{i}(\omega)}(\sigma^{\circ}_{m_k},\omega)=:\textbf{u}^{\textbf{i}(\omega)}_{m_k}$ for $k \geq K(\omega)$. Since $\textbf{v}$ satisfies Condition C1, we have 
			\begin{equation*}
				\nabla \textbf{v}^{*}_{\epsilon_{m_k},\textbf{j}(\omega)}(\mathcal{G}(\sigma^{\circ}_{m_k},\omega)) \geq \lambda(\epsilon_{m_k}) \textbf{v}^{*}_{\epsilon_{m_k},\textbf{i}(\omega)}(\mathcal{G}(\sigma^{\circ}_{m_k},\omega)), \hspace{3mm} \text{ for all } k \geq K(\omega).
			\end{equation*}
			Next, note that $\nabla \textbf{v}^{*}_{\epsilon_{m_k},\textbf{j}(\omega)}(\mathcal{G}(\sigma^{\circ}_{m_k},\omega))=\beta_{\epsilon_{m_k},\textbf{j}(\omega)}[\sigma^{\circ}_{m_k}](\omega)$ and $\nabla \textbf{v}^{*}_{\epsilon_{n_k},\textbf{i}(\omega)}(\mathcal{G}(\sigma^{\circ}_{m_k},\omega))=\beta_{\epsilon_{m_k},\textbf{i}(\omega)}[\sigma^{\circ}_{m_k}](\omega)$. Also since $\sigma^{\circ}_{m_k}$ is a $\epsilon_{m_k}$-regularized Bayesian equilibrium of the game $\mathcal{G}$, we have that $\nabla \textbf{v}^{*}_{\epsilon_{m_k},\textbf{j}(\omega)}(\mathcal{G}(\sigma^{\circ}_{m_k},\omega))=\sigma^{\circ,\textbf{i}(\omega)}_{m_k}(\omega)$ and $\nabla \textbf{v}^{*}_{\epsilon_{m_k},\textbf{i}(\omega)}(\mathcal{G}(\sigma^{\circ}_{m_k},\omega))=\sigma^{\circ,\textbf{i}(\omega)}_{m_k}(\omega)$. In what follows, we have
			\begin{equation*}
				\sigma^{\circ,\textbf{j}(\omega)}_{m_k}(\omega) \geq \lambda(\epsilon_{m_k})\sigma^{\circ,\textbf{i}(\omega)}_{m_k}(\omega), \hspace{3mm} \text{ for all } k \geq K(\omega).
			\end{equation*}
			Since $\epsilon_{m_k}$ converges to $0$ as $k \to \infty$ and that $\textbf{v}$ satisfies Condition C1, we have $\lambda(\epsilon_{m_k}) \to \infty$ as $k \to \infty$. Thus we can conclude that $\sigma^{\circ,\textbf{i}(\omega)}_{m_k}(\omega) \to 0$.  However, this leads to a contradiction since $\textbf{i}(\omega) \in \text{supp}(\sigma^{\circ}(\omega))$. Since $\omega \in \bar{\Omega}$ is arbitrary, we have that $\sigma^{\circ,\textbf{i}(\omega)}_{m_k}(\omega) \to 0$ for all $\omega \in \bar{\Omega}$. Thus, the required claim is proved. Therefore $\sigma^{\circ}$ must be a Bayesian equilibrium of the original game $\mathcal{G}$.
		\end{proof}
		
		\begin{proof}[Proof for regularizers satisfying Condition C2]
			
			Next we establish the same under Condition C2 (Definition \ref{ConditionC2}). We again prove this via the method of contradiction. However, contrary to Case 1, we show that $\sigma^{\circ,\textbf{i}(\omega)}_{m_k}(\omega) \to 0$ for some $\omega \in \bar{\Omega}$, which, as mentioned in the previous arguments, is enough to conclude that $\sigma^{\circ}$ is a Bayesian equilibrium. Suppose not. Then we must have that $\sigma^{\circ,\textbf{i}(\omega)}_{m_k}(\omega)$ converges to $c(\omega)$, where $c(\omega)>0$ for all $\omega \in \bar{\Omega}$.\footnote{We note that the convergence of $\sigma^{\circ,\textbf{i}(\omega)}_{m_k}(\omega)$ is guaranteed by weak convergence of $(\sigma^{\circ}_{m_k})_{k \geq 1}$.} As a result, $\sigma^{\circ,\textbf{i}(\omega)}_{m_k}(\omega)$ is bounded away from $0$ for all $\omega \in \bar{\Omega}$. Thus, we have via Condition C2 that there exists a sequence of shifted Bayesian strategies $(\widetilde{\sigma}^{\circ}_{m_k})_{k \geq 1}$, where $(\widetilde{\sigma}^{\circ}_{m_k})_{k \geq 1}$ is a shift of $(\sigma^{\circ}_{m_k})_{k \geq 1}$ over $\bar{\Omega}$, such that, after some algebra, the following condition holds true for all $\omega \in \bar{\Omega}$:
			\begin{align*} \sum_{\alpha \in \{\textbf{i}(\omega), \textbf{j}(\omega)\} }\mathcal{G}^{\alpha}({\sigma}^{\circ}_{m_k},\omega)\widetilde{\sigma}^{\circ,\alpha}_{m_k}(\omega)&\geq 2\eta(\omega)(\sigma^{\circ,\textbf{i}(\omega)}_{m_k}(\omega)-\widetilde{\sigma}^{\circ,\textbf{i}(\omega)}_{m_k}(\omega))+\sum_{\alpha \in \{\textbf{i}(\omega), \textbf{j}(\omega)\} }\mathcal{G}^{\alpha}({\sigma}^{\circ}_{m_k},\omega){\sigma}^{\circ,\alpha}_{m_k}(\omega)
			\end{align*}
			for sufficiently large $K_2(\omega)$. Now fix an $\epsilon>0$. Note that since $\epsilon_{m_k} \downarrow 0$, there exists $N \geq 1$ such that $\epsilon_{m_k}<\epsilon$ for all $k \geq N$.
			Also since the regularizer $\textbf{v}$ satisfies Condition C2. Thus, in particular, it satisfies Condition C2 with respect to the sequence of Bayesian strategies $(\sigma^{\circ}_{m_k})_{k \geq 1}$ over the subset $\bar{\Omega}$, we have for every $\omega \in \bar{\Omega}$,
			\begin{align*}
				|\textbf{v}_{\epsilon_{m_k}}(\widetilde{\sigma}^{\circ}_{m_k}(\omega))-\textbf{v}_{\epsilon_{m_k}}(\sigma^{\circ}_{m_k}(\omega))|&= |\epsilon_{m_k}||\textbf{v}_{\epsilon_{m_k}}(\widetilde{\sigma}^{\circ}_{m_k}(\omega))-\textbf{v}_{\epsilon_{m_k}}(\sigma^{\circ}_{m_k}(\omega))|\\
				& \leq \epsilon(\sigma^{\circ, \textbf{i}(\omega)}_{m_k}(\omega)-\widetilde{\sigma}^{\circ, \textbf{i}(\omega)}_{m_k}(\omega))
			\end{align*}	
			for sufficiently large $k \geq \max\{N,K_2(\omega)\}$. As a result, we have that $$\textbf{v}_{\epsilon_{m_k}}(\widetilde{\sigma}^{\circ}_{m_k}(\omega)) \leq \textbf{v}_{\epsilon_{m_k}}(\sigma^{\circ}_{m_k}(\omega))+ \epsilon(\sigma^{\circ, i(\omega)}_{m_k}(\omega)-\widetilde{\sigma}^{\circ, i(\omega)}_{m_k}(\omega))$$ for $k \geq \max\{N,K_2(\omega)\}$. Now note that for all $\omega \in \bar{\Omega}$ and $\alpha \in S \setminus \{\textbf{i}(\omega), \textbf{j}(\omega)\}$, we have $$\sum_{\alpha \in S \setminus \{\textbf{i}(\omega), \textbf{j}(\omega)\}}\mathcal{G}^{\alpha}(\sigma^{\circ}_{m_k},\omega)\widetilde{\sigma}^{\circ, \alpha}_{m_k}(\omega)=\sum_{\alpha \in S \setminus \{\textbf{i}(\omega), \textbf{j}(\omega)\}}\mathcal{G}^{\alpha}(\sigma^{\circ}_{m_k},\omega){\sigma}^{\circ, \alpha}_{m_k}(\omega).$$ 
			
			Combining all the above observations, we arrive at the following inequality:
			\begin{align}\label{PrevIneq}
				\langle \mathcal{G}({\sigma}^{\circ}_{m_k},\omega),\widetilde{\sigma}^{\circ}_{m_k}(\omega) \rangle - \epsilon_{m_k}\textbf{v}(\widetilde{\sigma}^{\circ}_{m_k}(\omega))  &\geq \eta(\omega)(\sigma^{\circ,\textbf{i}(\omega)}_{m_k}(\omega)-\widetilde{\sigma}^{\circ,\textbf{i}(\omega)}_{m_k}(\omega))\\ \nonumber
				& \hspace{2cm}+\langle \mathcal{G}({\sigma}^{\circ}_{m_k},\omega),{\sigma}^{\circ}_{m_k}(\omega) \rangle- \epsilon_{m_k}\textbf{v}(\sigma^{\circ}_{m_k}(\omega))
			\end{align}
			for large $k \geq \max\{N,K_2(\omega)\}$. Thus in terms of the sliced game, the above inequality can equivalently be written as
			\begin{align}\label{PrevIneq2}
				\langle \mathcal{G}_{\omega}({\sigma}^{\circ}_{m_k}),\widetilde{\sigma}^{\circ}_{m_k}(\omega) \rangle - \epsilon_{m_k}\textbf{v}(\widetilde{\sigma}^{\circ}_{m_k}(\omega))  &\geq \eta(\omega)(\sigma^{\circ,\textbf{i}(\omega)}_{m_k}(\omega)-\widetilde{\sigma}^{\circ,\textbf{i}(\omega)}_{m_k}(\omega))\\ \nonumber
				& \hspace{2cm}+\langle \mathcal{G}_{\omega}({\sigma}^{\circ}_{m_k}),{\sigma}^{\circ}_{m_k}(\omega) \rangle- \epsilon_{m_k}\textbf{v}(\sigma^{\circ}_{m_k}(\omega))
			\end{align}
			for large $k \geq \max\{N,K_2(\omega)\}$.
			However note that since $\sigma^{\circ}_{m_k}$ is an $\epsilon_{m_k}$-regularized Bayesian equilibrium, we have by Lemma \ref{LemOriginalGame-SlicedGame} that $\sigma^{\circ}_{m_k}(\omega)$ is an $\epsilon_{m_k}$-perturbed equilibrium of the sliced game $\mathcal{G}_{\omega}$ for every $k \geq 1$. Thus we have by definition that $\sigma^{\circ}_{m_k}(\omega)={\beta}_{\epsilon_{m_k}}[\sigma^{\circ}_{m_k}](\omega)$ for ${\xi}$-a.e $\omega \in \Omega$. As a result, it should be the case that $$\langle \mathcal{G}_{\omega}({\sigma}^{\circ}_{m_k}),{\sigma}^{\circ}_{m_k}(\omega) \rangle- \epsilon_{m_k}\textbf{v}(\sigma^{\circ}_{m_k}(\omega)) \geq \langle \mathcal{G}_{\omega}({\sigma}^{\circ}_{m_k}),\textbf{y} \rangle- \epsilon_{m_k}\textbf{v}(\textbf{y}), \hspace{3mm} \text{ for all } \textbf{y} \in \Delta.$$ 
			
			But we observe from (\ref{PrevIneq2}) that for all $\omega \in \bar{\Omega}$ the above inequality does not hold. Therefore, we have constructed a sequence of shifted Bayesian strategies $(\widetilde{\sigma}^{\circ}_{m_k})_{k \geq 1}$ which have a greater expected payoff under the $\omega$-sliced game $\mathcal{G}_{\omega}$ for every $\omega \in \bar{\Omega}$. We also note that the shifted sequence of Bayesian strategies differ from the original sequence of regularized Bayesian equilibria on the set $\bar{\Omega}$ of positive measure. Hence, arrive at a contradiction.
			Thus, $\sigma^{\circ, \textbf{i}(\omega)}_{m_k}(\omega) \to 0$ for some $\omega \in \bar{\Omega}$. Also since we have that $\sigma^{\circ}_{m_k} \to \sigma^{\circ}$ relative to the weak topology on $\Sigma$, we have 
			that $\sigma^{\circ,\textbf{i}(\omega)}(\omega)=0$. But again, we arrive at a contradiction since by our hypothesis, $\textbf{i}(\omega) \in \text{supp}(\sigma^{\circ}(\omega))$ for all $\omega \in \bar{\Omega}$. Therefore, $\sigma^{\circ}$ is a Nash equilibrium of the original game $\mathcal{G}$.
		\end{proof}

		\subsection{Regularizers satisfy Condition C1 or Condition C2}
		
		In this section, we shall show that the Shannon-Gibbs entropy satisfies Condition C1, whereas the Tsallis entropy and the Burg entropy satisfies Condition C2.
		
		\subsubsection{Shannon-gibbs entropy satisfies condition C1}
		In this subsection, we show that the Shannon-Gibbs entropy (Example \ref{ExampleShannon}) satisfies Condition C1 (Definition \ref{ConditionC1}). Fix $\textbf{u} \in \mathbb{R}^n$ and $i,j \in S$ with $i \neq j$ such that $\textbf{u}_j>\textbf{u}_i$. In that case, if follows from (\ref{LogitClosedForm}) that
		\begin{equation*}
			\frac{\nabla \textbf{v}^{*}_{\textbf{s},\epsilon,j}(\textbf{u})}{\nabla \textbf{v}^{*}_{\textbf{s},\epsilon,i}(\textbf{u})}=\frac{\exp(\epsilon^{-1}\textbf{u}_j)}{\exp(\epsilon^{-1}\textbf{u}_i)}=\exp(\epsilon^{-1}(\textbf{u}_j-\textbf{u}_i)).
		\end{equation*}
		This Condition C1 is satisfied with $\lambda(\epsilon):=\exp(\epsilon^{-1}(\textbf{u}_j-\textbf{u}_i))$. Note that since $\textbf{u}_j>\textbf{u}_i$, this implies that $\lambda(\epsilon) \to \infty$ as $\epsilon \to 0$.
		
		\subsubsection{Tsallis entropy satisfies Condition C2}
		In this subsection, we show that Tsallis entropy (Example \ref{ExampleTsallis}) satisfies Condition C2 (Definition \ref{ConditionC2}) with respect to the sequence $(\sigma^{\circ}_{m_k})_{k \geq 1}$, where $(\sigma^{\circ}_{m_k})_{k \geq 1}$ is sequence of Bayesian strategies obtained in Theorem \ref{ThmFixedPt}. Fix an arbitrary $\bar{\Omega} \subseteq \Omega$ such that $\xi(\bar{\Omega})>0$. Consider $\omega \in \bar{\Omega}$, distinct $\textbf{i},\textbf{j}:\bar{\Omega} \to S$ with $\lim_{k \to \infty}\sigma^{\circ,\textbf{i}(\omega)}_{m_k}(\omega) > 0$ and $\mathcal{G}^{\textbf{j}(\omega)}(\sigma^{\circ}_{m_k},\omega)>\mathcal{G}^{\textbf{i}(\omega)}(\sigma^{\circ}_{m_k},\omega)$ for sufficiently large $k$ (depending on $\omega$). The fact that $\lim_{k \to \infty}\sigma^{\circ,\textbf{i}(\omega)}_{m_k}(\omega) > 0$ implies that $\lim_{k \to \infty}\sigma^{\circ,\textbf{i}(\omega)}_{m_k}(\omega)$ is bounded away from $0$, in which case we can construct a corresponding sequence of Bayesian strategies $(\widetilde{\sigma}^{\circ}_{m_k})_{k \geq 1}$, where $(\widetilde{\sigma}^{\circ}_{m_k})_{k \geq 1}$ is a shift of $({\sigma}^{\circ}_{m_k})_{k \geq 1}$ over $\bar{\Omega}$ from $\textbf{i}$ to $\textbf{j}$. First we observe that for every $\omega \in \bar{\Omega}$, the induced belief under the shifted sequence $(\widetilde{\sigma}^{\circ}_{m_k})_{k \geq 1}$ differs from the original sequence $({\sigma}^{\circ}_{m_k})_{k \geq 1}$ only at the co-ordinates $\textbf{i}$ and $\textbf{j}$ in the sense that for all $\omega \in \bar{\Omega}$, we have $\widetilde{\sigma}^{\circ,k}_{m_k}(\omega)=\sigma^{\circ,k}_{m_k}(\omega)$ for all $k \in S \setminus \{\textbf{i}(\omega), \textbf{j}(\omega)\}$. Therefore, in order to verify Condition C2, it is enough to bound the difference in the entropy function over the strategies $\textbf{i}(\omega)$ and $\textbf{j}(\omega)$. Thus, by definition of Tsallis entropy, we have
		\begin{align}
			|\textbf{v}^{\textbf{t}}(\widetilde{\sigma}^{\circ}_{m_k}(\omega)-\textbf{v}^{\textbf{t}}(\sigma^{\circ}_{m_k}(\omega))|& \leq \frac{1}{q(1-q)} \Big [ \underbrace{|\widetilde{\sigma}^{\circ,\textbf{j}(\omega)}_{m_k}(\omega)-{\sigma}^{\circ,\textbf{j}(\omega)}_{m_k}(\omega)|}_{=:T^k_1(\omega)}+\underbrace{|\widetilde{\sigma}^{\circ,\textbf{j}(\omega)}_{m_k}(\omega)^q-{\sigma}^{\circ,\textbf{j}(\omega)}_{m_k}(\omega)^q|}_{=:T^k_2(\omega)}\Big] \nonumber \\
			& \hspace{2cm}+\frac{1}{q(1-q)}\Big[ |\widetilde{\sigma}^{\circ,\textbf{i}(\omega)}_{m_k}(\omega)-{\sigma}^{\circ,\textbf{i}(\omega)}_{m_k}(\omega)|+\underbrace{|\widetilde{\sigma}^{\circ,\textbf{i}(\omega)}_{m_k}(\omega)^q-{\sigma}^{\circ,\textbf{i}(\omega)}_{m_k}(\omega)^q|}_{=:T^k_3(\omega)}\Big].
		\end{align}
		
		We bound the terms $T^k_1(\omega), T^k_2(\omega)$ and $T^k_3(\omega)$ separately. Note that by the very definition of $\widetilde{\sigma}^{\circ}_{m_k}$ (via Definition \ref{DefShift}) that, 
		\begin{equation}\label{IneqT1}
			T^k_1(\omega)=|\widetilde{\sigma}^{\circ,\textbf{j}(\omega)}_{m_k}(\omega)-{\sigma}^{\circ,\textbf{j}}_{m_k}(\omega)|= |\widetilde{\sigma}^{\circ,\textbf{i}(\omega)}_{m_k}(\omega)-\sigma^{\circ,\textbf{i}(\omega)}_{m_k}(\omega)|.
		\end{equation}
		
		In order to bound the term $T^k_2(\omega)$, we first observe that $\widetilde{\sigma}^{\circ,\textbf{j}(\omega)}_{m_k}(\omega)$ can be expressed as
		\begin{equation}
			\widetilde{\sigma}^{\circ,\textbf{j}(\omega)}_{m_k}(\omega)=\delta^{\textbf{i}(\omega)\textbf{j}(\omega)}_{m_k}(\omega)\sigma^{\circ,\textbf{j}(\omega)}_{m_k}(\omega),
		\end{equation}
		where $\delta^{\textbf{i}(\omega)\textbf{j}(\omega)}_{m_k}(\omega):=1+\frac{\sigma^{\circ,\textbf{i}(\omega)}_{m_k}(\omega)-\widetilde{\sigma}^{\circ,\textbf{i}(\omega)}_{m_k}(\omega)}{\sigma^{\circ,\textbf{j}(\omega)}_{m_k}(\omega)}$. 
		Since by definition, $\delta^{\textbf{i}(\omega)\textbf{j}(\omega)}_{m_k}(\omega)>1$, and $0 < q <1$, we must have $\delta^{\textbf{i}(\omega)\textbf{j}(\omega)}_{m_k}(\omega)^q-1<\delta^{\textbf{i}(\omega)\textbf{j}(\omega)}_{m_k}(\omega)-1$. Thus we arrive at
		\begin{align*}
			T^k_2(\omega)&=|\widetilde{\sigma}^{\circ,\textbf{j}(\omega)}_{m_k}(\omega)^q-{\sigma}^{\circ,\textbf{j}(\omega)}_{m_k}(\omega)^q|\\
			&=|(\delta^{\textbf{i}(\omega)\textbf{j}(\omega)}_{m_k}(\omega)^q-1)\sigma^{\circ,\textbf{j}(\omega)}_{m_k}(\omega)^q|\\
			& \leq |(\delta^{\textbf{i}(\omega)\textbf{j}(\omega)}_{m_k}(\omega)-1)\sigma^{\circ,\textbf{j}(\omega)}_{m_k}(\omega)^q|\\
			&= \frac{\sigma^{\circ,\textbf{i}(\omega)}_{m_k}(\omega)-\widetilde{\sigma}^{\circ,\textbf{i}(\omega)}_{m_k}(\omega)}{\sigma^{\circ,\textbf{j}(\omega)}_{m_k}(\omega)^{1-q}}. 
		\end{align*}
		
		Thus, in order to bond the term $T^k_2(\omega)$, we need to lower bound the term $\sigma^{\circ,\textbf{j}(\omega)}_{m_k}(\omega)^{1-q}$. Note that since $\sigma^{\circ}_{m_k}(\omega)$ is a regularized equilibrium of the sliced game $\mathcal{G}_{\omega}$, we have that
		\begin{equation}
			\sigma^{\circ,\textbf{j}(\omega)}_{m_k}(\omega)=\Big[\frac{\epsilon_{m_k}}{1-q}\Big]^{1/(1-q)}\frac{1}{(\theta^{\textbf{t}}_{\epsilon}(\sigma^{\circ}_{m_k},\omega)-\mathcal{G}^{\textbf{j}(\omega)}_{\omega}(\sigma^{\circ}_{m_k}))^{1/(1-q)}}
		\end{equation} 
		Thus, letting $\eta(\omega) \downarrow 0$ in (\ref{Contra1}), we have that $\sigma^{\circ,\textbf{j}(\omega)}_{m_k}(\omega) \geq \sigma^{\circ,\textbf{i}(\omega)}_{m_k}(\omega)$ for all $k$ sufficiently large (depending on $\omega$). As a result, we have that $\sigma^{\circ,\textbf{j}(\omega)}_{m_k}(\omega)^{q-1} \leq \sigma^{\circ,\textbf{i}(\omega)}_{m_k}(\omega)^{q-1}$ for sufficiently large $k$ (depending on $\omega$). Note that since $\lim_{k \to \infty}\sigma^{\circ,\textbf{i}(\omega)}_{m_k}(\omega) > 0$, there exists $c_{\textbf{i}(\omega)}(\omega)>0$ such that $\sigma^{\circ,\textbf{i}(\omega)}_{m_k}(\omega) \geq \frac{c_{\textbf{i}(\omega)}(\omega)}{2}$ for sufficiently large $k$ (depending on $\omega$). This, in conjunction with the fact that $\sigma^{\circ,\textbf{j}(\omega)}_{m_k}(\omega) \geq \sigma^{\circ,\textbf{i}(\omega)}_{m_k}(\omega)$ for all $k$ sufficiently large (depending on $\omega$) implies that $\sigma^{\circ,\textbf{j}(\omega)}_{m_k}(\omega) \geq \frac{c_{\textbf{i}(\omega)}(\omega)}{2}$ for all $k$ sufficiently large (depending on $\omega$). Combining all the above observations, we arrive at
		\begin{equation}\label{IneqT2}
			T^k_2(\omega) \leq \frac{2^{1-q}}{c_{\textbf{i}(\omega)}(\omega)^{1-q}}(\sigma^{\circ,\textbf{i}(\omega)}_{m_k}(\omega)-\widetilde{\sigma}^{\circ,\textbf{i}(\omega)}_{m_k}(\omega)).
		\end{equation}

		Finally, we proceed to bound the term $T_3$. Recall that $\sigma^{\circ,\textbf{i}(\omega)}_{m_k}(\omega) \geq \frac{c_{\textbf{i}(\omega)}(\omega)}{2}$ for sufficiently large $k$ (depending on $\omega$). Now, by an application of the mean-value theorem, we have 
		\begin{align*}\label{IneqT3}
			T^k_3(\omega)&=|\widetilde{\sigma}^{\circ,\textbf{i}(\omega)}_{m_k}(\omega)^q-{\sigma}^{\circ,\textbf{i}(\omega)}_{m_k}(\omega)^q|\\
			&\leq \frac{q 2^{1-q}}{c_{\textbf{i}(\omega)}(\omega)^{1-q}}|\widetilde{\sigma}^{\circ,\textbf{i}(\omega)}_{m_k}(\omega)-{\sigma}^{\circ,\textbf{i}(\omega)}_{m_k}(\omega)|.
		\end{align*}
		
		Combining all the inequalities obtained above in (\ref{IneqT1}), (\ref{IneqT2}) and (\ref{IneqT3}) we have for sufficiently large $k$ that,
		\begin{align*}
			|\textbf{v}^{\textbf{t}}(\widetilde{\sigma}^{\circ}_{m_k}(\omega))-\textbf{v}^{\textbf{t}}(\sigma^{\circ}_{m_k}(\omega))|& \leq \Big(2+\frac{2^{1-q}}{c_{\textbf{i}(\omega)}(\omega)^{1-q}}+\frac{q 2^{1-q}}{c_{\textbf{i}(\omega)}(\omega)^{1-q}}\Big)|\widetilde{\sigma}^{\circ,\textbf{i}(\omega)}_{m_k}(\omega)-{\sigma}^{\circ,\textbf{i}(\omega)}_{m_k}(\omega)| \\
			&= \Big(\frac{2c_{\textbf{i}(\omega)}(\omega)^{1-q}+2^{1-q}(1+q)}{q(1-q)c_{\textbf{i}(\omega)}(\omega)^{1-q}}\Big)|\widetilde{\sigma}^{\circ,\textbf{i}(\omega)}_{m_k}(\omega)-{\sigma}^{\circ,\textbf{i}(\omega)}_{m_k}(\omega)|.
		\end{align*}
		
		Thus setting $\kappa_{\textbf{i}(\omega)}(\omega):=\frac{2c_{\textbf{i}(\omega)}(\omega)^{1-q}+2^{1-q}(1+q)}{q(1-q)c_{\textbf{i}(\omega)}(\omega)^{1-q}}$, and observing the fact that ${\sigma}^{\circ,\textbf{i}(\omega)}_{m_k}(\omega)>\widetilde{\sigma}^{\circ,\textbf{i}(\omega)}_{m_k}(\omega)$, proves the fact that the Tsallis entropy satisfies Condition C2.
		
		\subsubsection{Burg entropy satisfies Condition C2}
		In this subsection, we show that Burg entropy (Example \ref{ExampleBurg}) satisfies Condition C2. Fix $\omega \in \Omega$, distinct $\textbf{i},\textbf{j}:\bar{\Omega} \to S$ with $\lim_{k \to \infty}\sigma^{\circ,\textbf{i}(\omega)}_{m_k}(\omega) > 0$ and $\mathcal{G}^{\textbf{j}(\omega)}(\sigma^{\circ}_{m_k},\omega)>\mathcal{G}^{\textbf{i}(\omega)}(\sigma^{\circ}_{m_k},\omega)$ for sufficiently large $k$ (depending on $\omega$). As in the case of the proof for Tsallis entropy, it is enough to bound the Burg entropy at co-ordinates $\textbf{i}(\omega)$ and $\textbf{j}(\omega)$. 
		
		We proceed similar to the case of Tsallis entropy. Note that by definition of Burg entropy, we have
		\begin{align}
			|\textbf{v}^{\textbf{b}}(\widetilde{\sigma}^{\circ}_{m_k}(\omega))-\textbf{v}^{\textbf{b}}(\sigma^{\circ}_{m_k}(\omega))|& \leq   \underbrace{|\log \widetilde{\sigma}^{\circ,\textbf{j}(\omega)}_{m_k}(\omega)-\log{\sigma}^{\circ,\textbf{j}(\omega)}_{m_k}(\omega)|}_{=:U^k_1(\omega)}+\underbrace{|\log\widetilde{\sigma}^{\circ,\textbf{i}(\omega)}_{m_k}(\omega)-\log{\sigma}^{\circ,\textbf{i}(\omega)}_{m_k}(\omega)|}_{=:U^k_2(\omega)}.
		\end{align}
		
		We bound the terms $U^k_1(\omega), U^k_2(\omega)$ separately. We first observe that $\widetilde{\sigma}^{\circ,\textbf{j}(\omega)}_{m_k}(\omega)$ can be expressed as
		\begin{equation}
			\widetilde{\sigma}^{\circ,\textbf{j}(\omega)}_{m_k}(\omega)=\delta^{\textbf{i}(\omega)\textbf{j}(\omega)}_{m_k}(\omega)\sigma^{\circ,\textbf{j}(\omega)}_{m_k}(\omega),
		\end{equation}
		where $\delta^{\textbf{i}(\omega)\textbf{j}(\omega)}_{m_k}(\omega):=1+\frac{\sigma^{\circ,\textbf{i}(\omega)}_{m_k}(\omega)-\widetilde{\sigma}^{\circ,\textbf{i}(\omega)}_{m_k}(\omega)}{\sigma^{\circ,\textbf{j}(\omega)}_{m_k}(\omega)}$.
		Since $\log(1+x)<x$ for all $x>0$, we have that $\log \delta^{\textbf{i}(\omega)\textbf{j}(\omega)}_{m_k}(\omega) < \frac{\sigma^{\circ,\textbf{i}(\omega)}_{m_k}(\omega)-\widetilde{\sigma}^{\circ,\textbf{i}(\omega)}_{m_k}(\omega)}{\sigma^{\circ,\textbf{j}(\omega)}_{m_k}(\omega)}$. Now, by definition of Burg entropy, we have
		\begin{align*}
			U^k_1(\omega)&=|\log \widetilde{\sigma}^{\circ,\textbf{j}(\omega)}_{m_k}(\omega)-\log \sigma^{\circ,\textbf{j}(\omega)}_{m_k}(\omega)|\\
			&=|\log \big[\delta^{\textbf{i}(\omega)\textbf{j}(\omega)}_{m_k}(\omega){\sigma}^{\circ,\textbf{j}(\omega)}_{m_k}(\omega)\big]-\log \sigma^{\circ, \textbf{j}(\omega)}_{m_k}(\omega)|\\
			&=|\log\delta^{\textbf{i}(\omega)\textbf{j}(\omega)}_{m_k}(\omega)|\\
			&<\frac{\sigma^{\circ,\textbf{i}(\omega)}_{m_k}(\omega)-\widetilde{\sigma}^{\circ,\textbf{i}(\omega)}_{m_k}(\omega)}{\sigma^{\circ,\textbf{j}(\omega)}_{m_k}(\omega)}
		\end{align*}	
		
		Thus, similar to Tsallis entropy, we need to lower bound the term $\sigma^{\circ,\textbf{j}(\omega)}_{m_k}(\omega)$. Note that since $\sigma^{\circ}_{m_k}(\omega)$ is a regularized equilibrium of the sliced game $\mathcal{G}_{\omega}$, we have that
		\begin{equation}
			\sigma^{\circ,\textbf{j}}_{m_k}(\omega)=\frac{\epsilon_{m_k}}{(\theta^{\textbf{b}}_{\epsilon}(\sigma^{\circ}_{m_k},\omega)-\mathcal{G}^{\textbf{j}(\omega)}_{\omega}(\sigma^{\circ}_{m_k}))}.
		\end{equation} 
		Thus, letting $\eta(\omega) \downarrow 0$ in (\ref{Contra1}), we have that $\sigma^{\circ,\textbf{j}(\omega)}_{m_k}(\omega) \geq \sigma^{\circ,\textbf{i}(\omega)}_{m_k}(\omega)$ for all $k$ sufficiently large (depending on $\omega$). Note that since $\lim_{k \to \infty}\sigma^{\circ,\textbf{i}(\omega)}_{m_k}(\omega) > 0$, there exists $\widehat{c}_{\textbf{i}(\omega)}(\omega)>0$ such that $\sigma^{\circ,\textbf{i}(\omega)}_{m_k}(\omega) \geq \frac{\widehat{c}_{\textbf{i}(\omega)}(\omega)}{2}$ for sufficiently large $k$ (depending on $\omega$). This, in conjunction with the fact that $\sigma^{\circ,\textbf{j}(\omega)}_{m_k}(\omega) \geq \sigma^{\circ,\textbf{i}(\omega)}_{m_k}(\omega)$ for all $k$ sufficiently large (depending on $\omega$) implies that $\sigma^{\circ,\textbf{j}(\omega)}_{m_k}(\omega) \geq \frac{\widehat{c}_{\textbf{i}(\omega)}(\omega)}{2}$ for all $k$ sufficiently large (depending on $\omega$). Combining the above observations, we arrive at
		\begin{align}\label{IneqU1}
			U^k_1(\omega)& \leq \frac{2}{\widehat{c}_{\textbf{i}(\omega)}(\omega)}(\sigma^{\circ,\textbf{i}(\omega)}_{m_k}(\omega)-\widetilde{\sigma}^{\circ,\textbf{i}(\omega)}_{m_k}(\omega)).
		\end{align}

		We now proceed to bound the term $U^k_2(\omega)$. Notice that the mapping $x \mapsto \log x$ is Lipschitz on $[a,\infty)$, where $a>0$. The fact that $\sigma^{\circ,\textbf{i}(\omega)}_{m_k}(\omega)$ is bounded away from $0$, implies that the shifted sequence of Bayesian strategies can be chosen in such a way that $\widetilde{\sigma}^{\circ,\textbf{i}(\omega)}_{m_k}(\omega)$ is bounded away from $0$. More precisely, we can construct a shifted sequence of Bayesian strategies such that $\widetilde{\sigma}^{\circ,\textbf{i}(\omega)}_{m_k}(\omega) \geq \frac{\widehat{c}_{\textbf{i}(\omega)}(\omega)}{4}$. Thus, by an application of mean value theorem, we have 
		\begin{align}\label{Burg2}
			U^k_2(\omega)&=|\log \widetilde{\sigma}^{\circ,\textbf{i}(\omega)}_{m_k}(\omega)-\log \sigma^{\circ,\textbf{i}(\omega)}_{m_k}(\omega)| \nonumber\\
			& \leq \frac{4}{\widehat{c}_{\textbf{i}(\omega)}(\omega)}(\sigma^{\circ,\textbf{i}(\omega)}_{m_k}(\omega)-\widetilde{\sigma}^{\circ,\textbf{i}(\omega)}_{m_k}(\omega)).
		\end{align}
		
		Therefore, combining (\ref{IneqU1}) and (\ref{Burg2}), we arrive at
		\begin{align*}
			|\textbf{v}^{\textbf{b}}(\widetilde{\sigma}^{\circ}_{m_k}(\omega))-\textbf{v}^{\textbf{b}}(\sigma^{\circ}_{m_k}(\omega))|& \leq U^k_1(\omega)+U^k_2(\omega)\\
			& \Big(\frac{2}{\widehat{c}_{\textbf{i}(\omega)}(\omega)}+\frac{4}{\widehat{c}_i(\omega)}\Big)(\sigma^{\circ,\textbf{i}(\omega)}_{m_k}(\omega)-\widetilde{\sigma}^{\circ,\textbf{i}(\omega)}_{m_k}(\omega)) \\
			&\leq \frac{6}{\widehat{c}_{\textbf{i}(\omega)}(\omega)}(\sigma^{\circ,\textbf{i}(\omega)}_{m_k}(\omega)-\widetilde{\sigma}^{\circ,\textbf{i}(\omega)}_{m_k}(\omega)).
		\end{align*}
		Thus, setting $\widehat{\kappa}_{\textbf{i}(\omega)}(\omega):=\frac{6}{\widehat{c}_{\textbf{i}(\omega)}(\omega)}$ and observing the fact that ${\sigma}^{\circ,\textbf{i}(\omega)}_{m_k}(\omega)>\widetilde{\sigma}^{\circ,\textbf{i}(\omega)}_{m_k}(\omega)$, proves the fact that the Burg entropy satisfies Condition C2.

		\subsection{{Proof of Theorem \ref{ThmSolnExistence}}}\label{AppThmSolnExistence}
		We use the infinite dimensional Picard-Lindeloff Theorem (see \cite{zeidler1986nonlinear} for details) to prove Theorem \ref{ThmSolnExistence}.
		\begin{theorem}[\textbf{Picard-Lindel\"off Theorem}, \cite{zeidler1986nonlinear}]
			\label{pl}
			Let $\mathbb{Y}$ be a Banach space. Let $y_0 \in \mathbb{Y}, t_0 \in \mathbb{R}$, and $$\mathcal{Q}_{b}:=\{(t,y) \in \mathbb{R} \times \mathbb{Y} \mid \hspace{1mm} |t-t_0| \leq a \mbox{ and } \|y-y_0\| \leq b\}$$ for fixed $a,b>0$. Suppose $f:\mathcal{Q}_{b} \to \mathbb{Y}$ is continuous and that $$\|f(t,x)-f(t,y)\| \leq L\|x-y\| \text{ for all } (t,x),(t,y) \in \mathcal{Q}_{b}, \mbox{ and}$$
			$$\|f(t,x)\|<M \text{ for all } (t,x) \in \mathcal{Q}_{b}$$ where $L \geq 0$ and $M>0$ are fixed. Choose $c$ such that $0<c<a$ and $Mc<b$. Then the initial value problem $$\dot{x}(t)=f(t,x(t)), \hspace{1cm} x(t_0)=y_0,$$ where $x:[t_0-c,t_0+c] \to \mathbb{Y}$, has exactly one continuously differentiable solution on the interval $[t_0-c,t_0+c]$.
		\end{theorem}
		
		Now we proceed with the proof of Theorem \ref{ThmSolnExistence}. To prove Theorem \ref{ThmSolnExistence}, it is by Theorem \ref{pl} enough to show that the mapping $\sigma \mapsto {\beta}_{\epsilon}(\sigma)$ is Lipschitz continuous with respect to the strong topology on $\Sigma$. We show this in the following lemma.
		\begin{lemma}\label{Lip}
			The mapping $\sigma \mapsto {\beta}_{\epsilon}(\sigma)$ is Lipschitz continuous with respect to the strong norm on $\Sigma$.
		\end{lemma}
		
		\begin{proof}
			To prove the lemma, we need to show that there exists a constant $\kappa>0$, such that 
			\begin{equation}\label{LipRequirement}
				\vertiii{{\beta}_{\epsilon}(\sigma)-{\beta}_{\epsilon}(\rho)}\leq \bar{\kappa}\|\sigma-\rho\|, \hspace{3mm} \text{ for all } \sigma,\rho \in \Sigma.
			\end{equation}
			Fix $\sigma, \rho \in \Sigma$.
			By (\ref{RBBRMap}) and the definition of strong norm, we have that 
			\begin{align*}
				\vertiii{{\beta}_{\epsilon}(\sigma)-{\beta}_{\epsilon}(\rho)}&=\int_{\Omega}\|{\beta}_{\epsilon}[\sigma](\omega)- {\beta}_{\epsilon}[\rho](\omega)\|{\xi}(d\omega)\\
				&=\int_{\Omega}\|\nabla \textbf{v}^{*}_{\epsilon}(\mathcal{G}(\sigma,\omega))-\nabla \textbf{v}^{*}_{\epsilon}(\mathcal{G}(\rho,\omega))\|{\xi}(d\omega)\\
				&\leq\dfrac{1}{\epsilon{\gamma}}\int_{\Omega}\|\mathcal{G}(\sigma,\omega)-\mathcal{G}(\rho,\omega)\|{\xi}(d\omega)\\
				&\leq \dfrac{ \kappa}{\epsilon{\gamma}}\vertiii{\sigma-\rho}.
			\end{align*}
			The second last inequality occurs since $\nabla \textbf{v}^{*}_{\epsilon}$ is $(\epsilon{\gamma})^{-1}$-Lipschitz, and the last inequality occurs since the Bayesian population game $\mathcal{G}$ is $\kappa$-strongly Lipschitz. Finally we have that
			\begin{align*}
				\vertiii{{\beta}_{\epsilon}(\sigma)-{\beta}_{\epsilon}(\rho)}
				& \leq \bar{\kappa}\vertiii{\sigma-\rho}, \hspace{3mm} \text{ for all } \sigma, \rho \in \Sigma.
			\end{align*}	
			Therefore (\ref{LipRequirement}) is satisfied with $\bar{\kappa}=\frac{\kappa}{\epsilon{\gamma}}$. This completes the proof of Lemma \ref{Lip}.
			
			We now apply Gr\"onwall's Lemma (see \cite{zeidler1986nonlinear} for details) to prove the continuity of the semiflow, that is, part (ii) of Theorem \ref{ThmSolnExistence}. We state the lemma below.
			
			\begin{lemma}[\textbf{Gr\"onwalls's Lemma}, \cite{zeidler1986nonlinear}]
				\label{gl}
				Let $u,v:[t_0-a,t_0+a] \to \mathbb{V}$ be maps into the subspace $\mathbb{V}$ of a Banach space $\mathbb{Y}$. Suppose $f:[t_0-a,t_0+a] \times \mathbb{V} \to \mathbb{Y}$ is continuous and Lipschitz continuous with respect to $u$, that is, $$\|f(t,u)-f(t,v)\|_{\mathbb{Y}} \leq L\|u-v\|_{\mathbb{V}} \text{ for all } t \in [t_0-a,t_0+a] \text{ and } u,v \in \mathbb{V} \mbox{ and fixed }L>0.$$ Then $\|u(t)-v(t)\|_{\mathbb{V}} \leq e^{L|t-t_0|}\|u(0)-v(0)\|$ for all $t \in [t_0-a,t_0+a]$.
			\end{lemma}
			
			We apply Lemma \ref{gl} with $\mathbb{Y}=\widehat{\Sigma}, \mathbb{V}=\Sigma$, and $f={{\beta}}_{\epsilon}$. By Lemma \ref{Lip}, the conditions of Lemma \ref{gl} are satisfied. Therefore, there exists $k>0$ such that for any two solutions $(\sigma_{t})_{t \geq 0}$ and $(\widetilde{\sigma}_{t})_{t \geq 0}$ to the Generalized Dynamic with initial conditions $\sigma_{0}$ and $\widetilde{\sigma}_{0}$, respectively, we have $\vertiii{\sigma_{t}-\widetilde{\sigma}_{t}} \leq e^{kt}\vertiii{\sigma_{0}-\widetilde{\sigma}_{0}}$. This establishes continuity of the semiflow of the Generalized Dynamic. This concludes the proof of Theorem \ref{ThmSolnExistence}.

			\subsection{{Proof of Lemma \ref{ThmSigmaNormCpt}}}\label{AppThmNormCpt}
			To prove part (i) of the theorem, we require the following lemma.
			\begin{lemma}[\cite{balder1994weak}]\label{LemNormCpt}
				A subset $\mathbb{K}$ of $\mathcal{L}^{1}({\mu},\mathbb{X})$ is relatively norm compact if and only if the following conditions are satisfied: 
				\begin{enumerate}[(i)]
					\item The subset $\mathbb{K}$ is relatively weakly compact.
					\vspace{-2mm}
					\item The subset $\mathbb{K}$ satisfies small Bocce oscillation.
					\vspace{-2mm}
					\item The subset $\{\text{avg}(f;A):f \in \mathbb{K}\}$ is relatively norm compact in $\mathbb{X}$, for every $A \in \mathcal{B}_{\Omega}$, such that ${\mu}(A)>0$.
				\end{enumerate}
			\end{lemma}
			Throughout the proof, we fix a compact set $\Xi \subseteq \Omega$ with $\xi(\Xi)=1$. To verify condition (i) of Lemma \ref{LemNormCpt}, note that for every $\alpha>0$ and $M \subseteq \Delta^{\circ}$, we have $\Sigma^{\alpha}_M(\Xi)$ is a subset of $\Sigma$. Therefore, the fact that $\Sigma$ is relatively weakly compact implies that $\Sigma^{\alpha}_M(\Xi)$ is also relatively weakly compact. 
			
			We now proceed to verify condition (ii) of Lemma \ref{LemNormCpt}, that is, we need to show that $\Sigma^{\alpha}_M(\Xi)$ satisfies small Bocce oscillation. To this end, fix $\eta>0$. Since $\Sigma^{\alpha}_M(\Xi)$ is a family of Bayesian strategies which are $\alpha$-Lipschitz on $\Xi$, it follows that $\Sigma^{\alpha}_M(\Xi)$ is equicontinuous on $\Xi$. Therefore, by the definition, there exists $\delta>0$, such that $\omega,\bar{\omega} \in \Xi$ and $d_{\Omega}(\omega,\bar{\omega})<\delta$ imply that 
			\begin{equation}\label{SigmaEquiCont}
				\|\sigma(\omega)-\sigma(\bar{\omega})\|<\eta, \hspace{3mm} \text{ for all } \sigma \in \Sigma^{\alpha}_M(\Xi).
			\end{equation}
			Fix $\omega \in \Xi$, and let $B(\omega;\delta)$ denote the open ball of radius $\delta$ about $\omega$. Then, for every $\sigma \in \Sigma^{\alpha}_M(\Xi)$, we have that $\|\sigma(\omega)-\sigma(\bar{\omega})\|<\eta$, for all $\bar{\omega} \in B(\omega;\delta)$. We now proceed to show that $\text{Bocce-osc}(\sigma,B(\omega;\delta))<\eta$, for every $\sigma \in \Sigma^{\alpha}_M(\Xi)$. Fix $\sigma \in \Sigma^{\alpha}_M(\Xi)$. Note that for every $\bar{\omega} \in B(\omega;\delta)$, we have $\|\sigma(\omega)-\sigma(\bar{\omega})\|<\eta$. This implies that
			\begin{align*}
				\|\sigma(\omega)-\text{avg}(\sigma,B(\omega;\delta))\|&=\Big\|\sigma(\omega)-\frac{1}{{\xi}(B(\omega;\delta))}\int_{B(\omega;\delta)}\sigma(\bar{\omega})\xi(d\bar{\omega})\Big\|\\
				&=\frac{1}{{\xi}(B(\omega;\delta))}\Big\|\int_{B(\omega;\delta)}\sigma(\omega){\xi}(d\bar{\omega})-\int_{B(\omega;\delta)}\sigma(\bar{\omega}){\xi}(d\bar{\omega})\Big\|\\
				&\leq\frac{1}{{\xi}(B(\omega;\delta))}\int_{B(\omega;\delta)}\|\sigma(\omega)-\sigma(\bar{\omega})\|{\xi}(d\bar{\omega})\\
				&<\eta.
			\end{align*}
			
			This implies $\text{Bocc-osc}(\sigma,B(\omega,\delta))<\eta$. Since $\Sigma^{\alpha}_M(\Xi)$ is equicontinuous on $\Xi$, we have from (\ref{SigmaEquiCont}), that $$\text{Bocce-osc}(\sigma;B(\omega;\delta))<\eta, \hspace{3mm} \text{ for all } \sigma \in\Sigma^{\alpha}_M(\Xi).$$
			
			Note that $(B(\omega;\delta))_{\omega \in \Omega}$ is an open cover of $\Omega$, and hence an open cover of $\Xi$. Since $\Xi$ is compact, it follows that there exist $\omega_1,\ldots,\omega_p \in \Xi$, such that $\Xi \subseteq \cup_{1 \leq i \leq p}B(\omega_i;\delta)$. Suppose that the collection $(B(\omega_i;\delta))_{1 \leq i \leq p}$ is disjoint. Then, $$\{B(\omega_i;\delta):1 \leq i \leq p\}\cup(\Omega\setminus\Xi)$$ forms a finite measurable partition of $\Omega$. However, if it happens to be the case that the collection $(B(\omega_i;\delta))_{1 \leq i \leq p}$ is not disjoint, we disjointify the collection and extract a collection $(G(\omega_i;\delta))_{1 \leq i \leq p}$ such that $\cup_{1 \leq i \leq p}G(\omega;\delta)=\cup_{1 \leq i \leq p}B(\omega_i;\delta)$ and that $G(\omega_i;\delta) \cap G(\omega_j;\delta)=\emptyset$ for all $i \neq j$. In such a case, we have that $$\{G(\omega_i;\delta):1 \leq i \leq p\}\cup(\Omega\setminus\Xi)$$ forms a measurable partition of $\Omega$. The fact that the elements of the new collection $(G(\omega_i;\delta))_{1 \leq i \leq p}$ are measurable subsets follows since both open sets and closed sets are measurable, and the Borel sigma-algebra is closed under intersection of finitely many open and closed sets. Therefore we finally have that $$\sum_{1 \leq i \leq p}{\xi}(G(\omega_i;\delta))\text{Bocc-osc}(\sigma;G(\omega_i;\delta))<\eta.$$ This proves that $\Sigma^{\alpha}_M(\Xi)$ has small Bocce oscillation thereby verifying condition (ii) of Lemma \ref{LemNormCpt}.
			
			Finally we verify condition (iii) of Lemma \ref{LemNormCpt}. Fix $A \in \mathcal{B}_{\Omega}$, such that ${\xi}(A)>0$. By the definition of Bayesian strategies, we have that $\{\text{avg}(\sigma;A):\sigma\in\Sigma^{\alpha}_M(\Xi)\}\subseteq\Delta$. Since $\Delta$ is compact in the Euclidean topology, we have that $\{\text{avg}(\sigma;A):\sigma\in\Sigma^{\alpha}_M(\Xi)\}$ is relatively norm compact in $\mathbb{R}^{n}$. Since $A \in \mathcal{B}_{\Omega}$ is arbitrary, Condition (iii) of Lemma \ref{LemNormCpt} is satisfied. This concludes the proof of part (i) of Theorem \ref{ThmSigmaNormCpt}.
			
			We now proceed to prove part (ii) of Lemma \ref{ThmSigmaNormCpt}. We first observe that the game $\mathcal{G}$ is uniformly bounded in the sense that there exists $c>0$ such that for every $\sigma \in \Sigma$, $\|\mathcal{G}(\sigma,\omega)\| \leq c$ for $\xi$-a.e $\omega \in \Omega$. Let $B_{c}:=\{\textbf{u} \in \mathbb{R}^n: \|\textbf{u}\| \leq c\}$. Since the mapping $\textbf{u} \mapsto \nabla \textbf{v}^{*}_{\epsilon}(\textbf{u})$ is continuous and $B_c$ is a compact set, we have that the image of $B_c$ is a compact set. Also since the regularized best response maps payoff-vectors to the interior of $\Delta$, we have $\nabla \textbf{v}^{*}_{\epsilon}(\textbf{u}) \in \Delta^{\circ}$ for every $\textbf{u} \in \mathbb{R}^n$, we have that the image of $B_c$ is a compact subset of $\Delta^{\circ}$. Thus, for a given $c>0$, consider the set $B_c$ and define $M_c$ to be the image of $B_c$ under the map $\textbf{u} \mapsto \nabla \textbf{v}^{*}_{\epsilon}(\textbf{u})$. Therefore, we have that $\beta_{\epsilon}[\sigma] \in \Sigma_{M_c}$ whenever $\sigma \in \Sigma$. This implies, in particular, that $\beta_{\epsilon}[\sigma] \in \Sigma_{M_c}$ whenever $\sigma \in \Sigma_{M_c}$. Also, by assumption of lemma, we have for every $\sigma \in \Sigma$, the map $\omega \mapsto \mathcal{G}(\sigma,\omega)$ is $\lambda$-Lipschitz for $\xi$-a.e $\omega \in \Xi$. As a result, fixing $\omega,\bar{\omega} \in \Xi$, we have
			\begin{align*}
				\|{\beta}_{\epsilon}[\sigma](\omega)-{\beta}_{\epsilon}[\sigma](\bar{\omega})\|&=\|\nabla \textbf{v}^{*}_{\epsilon}(\mathcal{G}(\sigma,\omega))-\nabla \textbf{v}^{*}_{\epsilon}(\mathcal{G}(\sigma,\bar{\omega}))\|\\
				&\leq \dfrac{1}{\epsilon{\gamma}}\|\mathcal{G}(\sigma,{\omega})-\mathcal{G}(\sigma,\bar{\omega})\|\\
				&\leq \dfrac{\lambda}{\epsilon{\gamma}}{d}_{\Omega}(\omega,\bar{\omega})\\
				&\leq \alpha {d}_{\Omega}(\omega,\bar{\omega}).
			\end{align*}
			
			Combining the above observations, we have that $\beta_{\epsilon}[\sigma] \in \Sigma^{\alpha}_{M_c}(\Xi)$ whenever $\sigma \in \Sigma$. This again, in particular, implies that $\beta_{\epsilon}[\sigma] \in \Sigma^{\alpha}_{M_c}(\Xi)$ whenever $\sigma \in \Sigma^{\alpha}_{M_c}(\Xi)$. Thus, we have proved that given $c>0$, there exists a compact set $M_c \subseteq \Delta^{\circ}$ such that $\beta_{\epsilon}[\sigma] \in \Sigma^{\alpha}_{M_c}(\Xi)$ whenever $\sigma \in \Sigma^{\alpha}_{M_c}(\Xi)$, provided the Lipschitz coefficient of the Bayesian strategy $\alpha \geq \frac{\lambda}{\epsilon\gamma}$. Hence $\Sigma^{\alpha}_{M_c}(\Xi)$ is forward invariant under the RBBR learning dynamic. This completes the proof of part (ii) of Lemma \ref{ThmSigmaNormCpt}.
			
			\subsection{Proof of Lemma \ref{LemPositiveCorrelation}}\label{AppLemPositiveCorrelation}
			
			In order to prove the lemma, we need to show that $\int_{\Omega} \langle \widetilde{\mathcal{G}}(\sigma,\omega), \dot{\sigma}(\omega) \rangle \xi(d\omega) \geq 0$ for all $\sigma \in \Sigma$ with equality only if $\dot{\sigma}=0$. We first prove the inequality.
			We require the following lemma in this regard.
			
			\begin{lemma}\label{LemPopolnGame=GateauxDeri}
				For every $\sigma \in \Sigma^{\alpha}_M$, the following equality holds: $$\int_{\Omega}\langle \nabla^{\textbf{G}}\widetilde{\textbf{v}}_{\epsilon}[{\beta}_{\epsilon}(\sigma)](\omega),\dot{\sigma}(\omega)\rangle {\xi}(d\omega)=\int_{\Omega}\langle \mathcal{G}(\sigma,\omega), \dot{\sigma}(\omega)\rangle{\xi}(d\omega).$$
			\end{lemma}	
			\begin{proof}
				To prove the lemma, fix $\sigma \in \Sigma^{\alpha}_M$ and $\omega \in \Omega$ and consider a mapping $\Phi^{\epsilon}_{\sigma,\omega}:\Delta^{\circ} \to \mathbb{R}$ defined as
				\begin{equation}
					\Phi^{\epsilon}_{\sigma,\omega}(\textbf{y}):=\langle \textbf{y}, \mathcal{G}(\sigma,\omega) \rangle - \epsilon \textbf{v}(\textbf{y}), \hspace{3mm} \text{ for all } \textbf{y} \in \Delta^{\circ}.
				\end{equation}			
				
				Now note that by the definition of regularized Bayesian best response map and the fact that $\textbf{v}$ is strongly convex, we have that $\beta_{\epsilon}[\sigma](\omega)$ is the unique unique maximizer of $\Phi^{\epsilon}_{\sigma,\omega}$. Therefore, we must have ${d}\Phi^{\epsilon}_{\sigma,\omega}[\beta_{\epsilon}[\sigma](\omega)](\textbf{z}_0)=0$ for all $\textbf{z}_0 \in \mathbb{R}^n_0$. As a result, for $\textbf{z}_0 \in \mathbb{R}^n_0$, we have 
				\begin{align}
					0&={d}\Phi^{\epsilon}_{\sigma,\omega}[\beta_{\epsilon}[\sigma](\omega)](\textbf{z}_0) \nonumber\\
					&=\lim_{\eta \to 0}\frac{\Phi^{\epsilon}_{\sigma,\omega}(\beta_{\epsilon}[\sigma](\omega)+\eta \textbf{z}_0)-\Phi^{\epsilon}_{\sigma,\omega}(\beta_{\epsilon}[\sigma](\omega))}{\eta}\nonumber\\
					&=\langle \textbf{z}_0,\mathcal{G}(\sigma,\omega) \rangle -\lim_{\eta \to 0}\frac{ \textbf{v}_{\epsilon}(\beta_{\epsilon}[\sigma](\omega)+\eta \textbf{z}_0)-\textbf{v}_{\epsilon}(\beta_{\epsilon}[\sigma](\omega))}{\eta}\nonumber\\
					&=\langle \textbf{z}_0,\mathcal{G}(\sigma,\omega) \rangle-d\textbf{v}_{\epsilon}[\beta_{\epsilon}[\sigma](\omega)](\textbf{z}_0).  \label{eqnInLemmaA7}
				\end{align}
				Thus, we have $d\textbf{v}_{\epsilon}[\beta_{\epsilon}[\sigma](\omega)](\textbf{z}_0)=\langle \mathcal{G}(\sigma,\omega), \textbf{z}_0 \rangle$ for all $\textbf{z}_0 \in \mathbb{R}^n_0$. Fix $\sigma \in \Sigma^{\alpha}_M$ and $\sigma_0 \in \Sigma_0$. Recall from the definition of $\widetilde{\textbf{v}}$ that for fixed $\eta>0$, we have $$\widetilde{\textbf{v}}_{\epsilon}({\beta}_{\epsilon}[\sigma]+\eta\sigma_0)=\int_{\Omega}\textbf{v}_{\epsilon}({\beta}_{\epsilon}[\sigma](\omega)+\eta\sigma_0(\omega)){\xi}(d\omega),$$ and that $$\widetilde{\textbf{v}}_{\epsilon}({\beta}_{\epsilon}[\sigma])=\int_{\Omega}\textbf{v}_{\epsilon}({\beta}_{\epsilon}[\sigma](\omega)){\xi}(d\omega).$$ Since $\sigma \in \Sigma^{\alpha}_M$, we have that $\textbf{v}(\sigma(\omega))$ is uniformly bounded for ${\xi}$-a.e $\omega \in \Omega$, in which case, we have
				\begin{equation}\label{DerivativeEquality}
					{d}\widetilde{\textbf{v}}_{\epsilon}[{\beta}_{\epsilon}(\sigma)](\sigma_0)=\int_{\Omega}{d}\textbf{v}_{\epsilon}({\beta}_{\epsilon}[\sigma](\omega))(\sigma_0(\omega)){\xi}(d\omega).
				\end{equation}
				However using (\ref{eqnInLemmaA7}) and the definition of G\^ateaux derivative, we have 
				\begin{align*}
					\int_{\Omega}\nabla^{\textbf{G}}\widetilde{\textbf{v}}_{\epsilon}[{\beta}_{\epsilon}(\sigma)](\omega)(\sigma_0(\omega)){\xi}(d\omega)&=\int_{\Omega}\langle\sigma_0(\omega),\nabla^{\textbf{G}}\widetilde{\textbf{v}}_{\epsilon}[{\beta}_{\epsilon}(\sigma)](\omega) \rangle{\xi}(d\omega)\\
					&=\int_{\Omega}{d}\textbf{v}_{\epsilon}({\beta}_{\epsilon}[\sigma](\omega))(\sigma_0(\omega)){\xi}(d\omega)\\
					&=\int_{\Omega}\langle\mathcal{G}(\sigma,\omega),\sigma_0(\omega)\rangle{\xi}(d\omega).
				\end{align*}
				
				Note that the above equality holds true for arbitrary $\sigma_0 \in \Sigma_0$. Since, $\dot{\sigma}=\beta_{\epsilon}(\sigma)-\sigma \in \Sigma_0$, we have that 
				\begin{equation}
					\int_{\Omega}\nabla^{\textbf{G}}\widetilde{\textbf{v}}_{\epsilon}[{\beta}_{\epsilon}(\sigma)](\omega)(\dot{\sigma}(\omega)){\xi}(d\omega)=\int_{\Omega}\langle\mathcal{G}(\sigma,\omega),\dot{\sigma}(\omega)\rangle{\xi}(d\omega).
				\end{equation}
				
				This completes the proof of the Lemma \ref{LemPopolnGame=GateauxDeri}.
			\end{proof}
			
			We now proceed with the proof of Lemma \ref{LemPositiveCorrelation}. Note that, as a consequence of result of Lemma \ref{LemPopolnGame=GateauxDeri}, we have the following equality: 
			\begin{align*}
				\int_{\Omega}\langle\widetilde{\mathcal{G}}(\sigma,\omega),\dot{\sigma}(\omega)\rangle{\xi}(d\omega)&=\int_{\Omega}\langle\mathcal{G}(\sigma,\omega)-\nabla^{\textbf{G}}\widetilde{\textbf{v}}[\sigma](\omega),\dot{\sigma}(\omega)\rangle{\xi}(d\omega)\\
				&=\int_{\Omega}\langle\nabla^{\textbf{G}}\widetilde{\textbf{v}}[{\beta}_{\epsilon}(\sigma)](\omega)-\nabla^{\textbf{G}}\widetilde{\textbf{v}}[\sigma](\omega),\dot{\sigma}(\omega)\rangle{\xi}(d\omega).
			\end{align*}
			Therefore, in order to prove Theorem \ref{bpgc}, firstly we need to show that 
			\begin{equation}\label{reqdineq}
				\int_{\Omega}\langle\nabla^{\textbf{G}}\widetilde{\textbf{v}}[{\beta}_{\epsilon}(\sigma)](\omega)-\nabla^{\textbf{G}}\widetilde{\textbf{v}}[\sigma](\omega),\dot{\sigma}(\omega)\rangle{\xi}(d\omega)\geq 0.
			\end{equation}
			Since $\textbf{v}$ is strongly convex (and hence convex), we have that $\widetilde{\textbf{v}}$ is also convex. As a result, we have from \cite{rockafellar1970convex}, that for every $\sigma,\rho \in \Sigma$,
			\begin{align*}
				\widetilde{\textbf{v}}(\rho)&\geq\widetilde{\textbf{v}}(\sigma)+\int_{\Omega}\langle\rho-\sigma,\nabla^{\textbf{G}}\widetilde{\textbf{v}}(\sigma)\rangle(\omega){\xi}(d\omega)\\
				&=\widetilde{\textbf{v}}(\sigma)+\int_{\Omega}\langle \nabla^{\textbf{G}}\widetilde{\textbf{v}}[\sigma](\omega),\rho(\omega)-\sigma(\omega)\rangle {\xi}(d\omega).
			\end{align*}
			Putting $\rho={\beta}_{\epsilon}[\sigma]$ in the above inequality we have that 
			\begin{equation}\label{BayPotIneq1}
				\widetilde{\textbf{v}}({\beta}_{\epsilon}[\sigma])\geq\widetilde{\textbf{v}}(\sigma)+\int_{\Omega}\langle\nabla^{\textbf{G}}\widetilde{\textbf{v}}[\sigma](\omega),{\beta}_{\epsilon}[\sigma](\omega)-\sigma(\omega)\rangle{\xi}(d\omega).
			\end{equation}
			Now interchanging $\rho$ and $\sigma$, we have that
			\begin{equation}\label{BayPotIneq2}
				\widetilde{\textbf{v}}(\sigma)\geq\widetilde{\textbf{v}}({\beta}_{\epsilon}[\sigma])+\int_{\Omega}\langle\nabla^{\textbf{G}}\widetilde{\textbf{v}}({\beta}_{\epsilon}[\sigma])(\omega),\sigma(\omega)-{\beta}_{\epsilon}[\sigma](\omega)\rangle{\xi}(d\omega). 	
			\end{equation}
			Combining inequalities (\ref{BayPotIneq1}) and (\ref{BayPotIneq2}) above, we arrive at the required inequality (\ref{reqdineq}). Hence this completes the proof of Lemma \ref{LemPositiveCorrelation}.
			
			It now remains to show that equality in (\ref{reqdineq}) holds only if $\beta_{\epsilon}(\sigma)=\sigma$, that is, $\dot{\sigma}=0$. Since $\textbf{v}$ is a decomposable regularizer, there exists a $\mathcal{C}^1$-function $\theta:(0,1) \to \mathbb{R}$ such that $\textbf{v}(\textbf{x})=\sum_{1 \leq i \leq n} \theta (\textbf{x}_i)$, for all $\textbf{x} \in \Delta^{\circ}$. We first require the following lemma.
			\begin{lemma}\label{LemDerivativeBayRegu}
				Suppose that $\textbf{v}$ is a regularizer decomposable via the function $\theta:(0,1) \to \mathbb{R}$. Then for every $\sigma \in \Sigma^{\alpha}_M$, we have 
				\begin{equation*}
					\text{d}\widetilde{\textbf{v}}_{\epsilon}[\sigma](\sigma_0):=\epsilon\sum\limits_{i=1}^{n}\int_{\Omega}\theta{'}(\sigma^{i}(\omega))\sigma^{i}_0(\omega) \xi(d\omega), \hspace{3mm} \text{ for all } \sigma_0 \in \Sigma_0.
				\end{equation*}
			\end{lemma}	
			\begin{proof}
				For $\textbf{x} \in \Delta^{\circ}$ and $\textbf{z}_0 \in \mathbb{R}^n_0$, we have
				\begin{align*}
					{d}\textbf{v}_{\epsilon}[\textbf{x}](\textbf{z}_0)&=\lim_{\eta \to 0}\frac{\textbf{v}_{\epsilon}(\textbf{x}+\eta \textbf{z}_0)-\textbf{v}_{\epsilon}(\textbf{x})}{\eta}\\
					&=\epsilon\lim_{\eta \to 0}\frac{\sum_{1 \leq i \leq n}\theta(\textbf{x}_i+\eta \textbf{z}^0_i)-\sum_{1 \leq i \leq n}\theta(\textbf{x}_i)}{\eta}\\
					&=\epsilon \sum\limits_{i=1}^{n}\textbf{z}^i_0 \theta'(\textbf{x}_i).
				\end{align*}
				It now follows from (\ref{DerivativeEquality}) and the previous arguments that for $\sigma \in \Sigma^{\alpha}_M$ and $\sigma_0 \in \Sigma_0$, we have
				\begin{align*}
					{d}\widetilde{\textbf{v}}_{\epsilon}[\sigma](\sigma_0)&=\int_{\Omega}{d}\textbf{v}_{\epsilon}(\sigma(\omega))(\sigma_0(\omega))\xi(d\omega)\\
					&=\epsilon\sum\limits_{i=1}^{n} \int_{\Omega}\theta'(\sigma^{i}(\omega))\sigma^{i}_0(\omega)\xi(d\omega).
				\end{align*}
				This completes the proof of Lemma \ref{LemDerivativeBayRegu}.
			\end{proof}
			
			We now proceed with the proof of part (i) of Lemma \ref{LemPositiveCorrelation} in the context of decomposable regularizers. By Lemma \ref{LemDerivativeBayRegu}, we have that the G\^ateaux gradient $\nabla^{\textbf{G}}\widetilde{\textbf{v}}_{\epsilon}(\beta_{\epsilon}[\sigma])$ at the type $\omega$ is identified by the vector $\epsilon(\theta'(\beta^{1}_{\epsilon}[\sigma](\omega)),\ldots,\theta'(\beta^{n}_{\epsilon}[\sigma](\omega)))$. Similarly, we have that $\nabla^{\textbf{G}}\widetilde{\textbf{v}}_{\epsilon}[\sigma](\omega)=\epsilon(\theta'(\sigma^{1}(\omega)),\ldots,\theta'(\sigma^{n}(\omega)))$ for all $\omega \in \Omega$. This in particular implies that 
			\begin{align}\label{EqualityGateauxGradientAndDecomposable}
				\int_{\Omega}\langle\nabla^{\textbf{G}}\widetilde{\textbf{v}}[{\beta}_{\epsilon}(\sigma)](\omega)-\nabla^{\textbf{G}}\widetilde{\textbf{v}}[\sigma](\omega),\dot{\sigma}(\omega)\rangle{\xi}(d\omega) \nonumber\\
				&\hspace{-3cm}=\epsilon\sum\limits_{i=1}^{n}\int_{\Omega}[\theta'(\beta^i_{\epsilon}[\sigma](\omega))-\theta'(\sigma^{i}(\omega))][\beta^{i}_{\epsilon}[\sigma](\omega)-\sigma^{i}(\omega)]\xi(d\omega).
			\end{align}
			
			By assumption of the lemma, we have that the function $\theta'$ is strictly increasing on $(0,1)$.  Thus, we have that the RHS of the equality (\ref{EqualityGateauxGradientAndDecomposable}) is always non-negative for $\xi$-a.e $\omega \in \Omega$. This implies that the RHS of (\ref{EqualityGateauxGradientAndDecomposable}) equals $0$ iff $\beta_{\epsilon}[\sigma]=\sigma$, that is equality holds iff $\dot{\sigma}=0$.
			
			Finally we proceed with the proof of part (ii) of Lemma \ref{LemPositiveCorrelation}. We note  by Lemma \ref{ThmSigmaNormCpt}, that there exists a compact set $M \subseteq \Delta^{\circ}$ such that $\Sigma^{\alpha}_{M}(\Xi)$ is forward invariant under the RBBR dynamic. Thus, we can find $\delta \in (0,\frac{1}{2})^n$ such that $M \subseteq [\delta^1, 1-\delta^1] \times \cdots \times [\delta^n,1-\delta^n]$.\footnote{We clarify the reader that the compact set $M$ and $\delta \in (0,\frac{1}{2})^n$ depends on the underlying parameters stated precisely in Lemma \ref{ThmSigmaNormCpt}. We remove the dependence of the parameters for a clean presentation of the proof.} Now suppose that $\sigma \in \Sigma^{\alpha}_{M}(\Xi)$. By forward invariance, $\beta_{\epsilon}[\sigma] \in \Sigma^{\alpha}_{{M}}(\Xi)$. Also by assumption in part (ii) of the lemma, we have that $\theta'$ is either strictly increasing on the interval $[\min_{1 \leq i \leq n} \delta^i, 1- \min_{1 \leq i \leq n}\delta^{i}]$. Thus, in such a scenario, the RHS of the equality (\ref{EqualityGateauxGradientAndDecomposable}) is non-negative or non-positive for $\xi$-a.e $\omega \in \Omega$. This implies that the RHS of (\ref{EqualityGateauxGradientAndDecomposable}) equals $0$ iff $\beta_{\epsilon}[\sigma]=\sigma$, that is equality holds iff $\dot{\sigma}=0$. Thus, the RBBR learning dynamic satisfies positive correlation with respect to the virtual payoff $\widetilde{\mathcal{G}}$.

			\subsection{{Proof of Theorem \ref{bpgc}}}\label{AppSectionPotential}
			We first observe that by Theorem \ref{ThmSigmaNormCpt}, the subset $\Sigma^{\alpha}_{M}(\Xi)$ is norm compact and forward invariant under the RBBR learning dynamic. Thus, in order to prove the theorem, we need to show the Bayesian entropy adjusted potential function $\widetilde{\varphi}^{\alpha}_{M}$ increases along every solution trajectory to the regularized Bayesian best response learning dynamic which originates in $\Sigma^{\alpha}_{M}(\Xi)$.
			Note that 
			\begin{align*}
				\dot{\widetilde{\varphi}}(\sigma)&=\dot{\varphi}(\sigma)-\dot{\widetilde{\textbf{v}}}(\sigma)\\
				&=\text{D}\varphi[\sigma](\dot{\sigma})-{d}\widetilde{\textbf{v}}[\sigma](\dot{\sigma})\\
				&=\int_{\Omega}\langle \dot{\sigma},\nabla^{\textbf{F}}\varphi(\sigma)\rangle(\omega){\xi}(d\omega)-\int_{\Omega}\langle\dot{\sigma},\nabla^{\textbf{G}}\widetilde{\textbf{v}}(\sigma)\rangle(\omega){\xi}(d\omega)\\
				&=\int_{\Omega}\nabla^{\textbf{F}}\varphi[\sigma](\omega)(\dot{\sigma}(\omega)){\xi}(d\omega)-\int_{\Omega}\nabla^{\textbf{G}}\widetilde{\textbf{v}}[\sigma](\omega)(\dot{\sigma}(\omega)){\xi}(d\omega)\\
				&=\int_{\Omega}\langle\nabla^{\textbf{F}}\varphi[\sigma](\omega),\dot{\sigma}(\omega)\rangle{\xi}(d\omega)-\int_{\Omega}\langle \nabla^{\textbf{G}}\widetilde{\textbf{v}}[\sigma](\omega),\dot{\sigma}(\omega)\rangle{\xi}(d\omega)\\
				&=\int_{\Omega}\langle\mathcal{G}(\sigma,\omega),\dot{\sigma}(\omega)\rangle{\xi}(d\omega)-\int_{\Omega}\langle \nabla^{\textbf{G}}\widetilde{\textbf{v}}[\sigma](\omega),\dot{\sigma}(\omega)\rangle{\xi}(d\omega)\\
				&=\int_{\Omega}\langle\mathcal{G}(\sigma,\omega)-\nabla^{\textbf{G}}\widetilde{\textbf{v}}[\sigma](\omega),\dot{\sigma}(\omega)\rangle{\xi}(d\omega)\\
				&=\int_{\Omega}\langle\widetilde{\mathcal{G}}(\sigma,\omega),\dot{\sigma}(\omega)\rangle{\xi}(d\omega)
			\end{align*}

			Now we complete the proof by showing that the solution trajectory to the RBBR dynamic converges to the set of regularized Bayesian equilibria under the strong topology. By assumption of the theorem, it follows from Theorem \ref{ThmSigmaNormCpt} that $\Sigma^{\alpha}_{M}(\Xi)$ is norm compact and forward invariant. Also, since the vector field is Lipschitz under the norm $\vertiii{\hspace{0.5mm}.\hspace{0.5mm}}$, it follows from Theorem \ref{ThmSolnExistence} that a continuous and unique solution exists, given any initial condition $\sigma_0 \in \Sigma$. It then follows from standard results of dynamical systems theory (see \cite{bhatia2006dynamical} for details) that the omega--limit sets of the trajectory $(\sigma_{t})_{t \geq 0}$ is non--empty, compact, and connected in the norm topology. Moreover this set is contained in the set of rest points of the {RBBR} learning dynamic. Hence any element in the omega--limit set is a rest point of the {RBBR} learning dynamic and hence a regularized Bayesian equilibrium. Also since the entropy adjusted Bayesian
			potential function is increasing along every solution trajectory, any point in the omega--limit set will be a maximizer of the function. This concludes the proof of Theorem \ref{bpgc}.
			
			\subsection{{Proof of Lemma \ref{sde}}}\label{AppSectionBayesianNegative}
			The proof of the lemma follows along the lines of \cite{cheung2014pairwise}. Fix $\omega \in \Omega$,
			$\sigma \in \Sigma$, and ${\sigma}_0 \in {\Sigma_{0}}$. Since the slice map $\mathcal{G}_{\omega}:\widehat{\Sigma} \rightarrow \mathbb{R}^n$ is $\mathcal{C}^1$-Fr\'echet differentiable for ${\xi}$-a.e $\omega \in \Omega$, we have  $$\mathcal{G}_{\omega}(\sigma+\epsilon{\sigma}_0)=\mathcal{G}_{\omega}({\sigma})+\text{D}\mathcal{G}_{\omega}[\sigma](\epsilon{\sigma}_0)+\textbf{o}(\epsilon\|{\sigma}_0\|).$$ This implies that $$\mathcal{G}_{\omega}(\sigma+\epsilon{\sigma}_0)-\mathcal{G}_{\omega}({\sigma})=\text{D}\mathcal{G}_{\omega}[\sigma](\epsilon{\sigma}_0)+\textbf{o}(\epsilon\|{\sigma}_0\|).$$ Taking inner product with respect to $\epsilon\sigma_0(\omega)$ on both sides of the above equation and integrating with respect to ${\xi}$, we have $$\int_{\Omega}\langle\mathcal{G}_{\omega}(\sigma+\epsilon{\sigma}_0)-\mathcal{G}_{\omega}(\sigma),\epsilon\sigma_0(\omega)\rangle{\xi}(d\omega)=\int_{\Omega}\langle\text{D}\mathcal{G}_{\omega}[\sigma](\epsilon {\sigma}_0),\sigma_0(\omega)\rangle{\xi}(d\omega)+ \textbf{o}(\epsilon^2\|{\sigma}_0\|^2).$$ Since $\mathcal{G}$ is a negative semidefinite game, we have $$\epsilon^2\int_{\Omega}\langle\text{D}\mathcal{G}_{\omega}(\sigma)({\sigma}_0),\sigma_0(\omega)\rangle{\xi}(d\omega)+\textbf{o}(\epsilon^2\|{\sigma}_0\|^2) \leq 0.$$ Since $\epsilon$ is arbitrary, we have by dividing both sides of the above inequality by $\epsilon^2$ and allowing $\epsilon \rightarrow 0$, that $$\int_{\Omega}\langle\text{D}\mathcal{G}_{\omega}[\sigma]({\sigma}_0),\sigma_0(\omega)\rangle{\xi}(d\omega) \leq 0.$$ This proves that $\mathcal{G}$ satisfies self defeating externalities if and only if $\mathcal{G}$ is a negative semidefinite game, which concludes the proof of Lemma \ref{sde}.

			\subsection{Proof of Lemma \ref{LemBayNegSemGameUnique}}\label{AppLemBayNegSemGameUnique}
			\begin{proof}
				\cite{hofbauer2007evolution} show uniqueness of logit equilibrium for finite strategy negative semidefinite games under homogeneous population set-up. We extend this result to finite strategy negative semidefinite games under heterogeneous set-up, which we call Bayesian negative semidefinite games.	We prove the lemma along the lines of \cite{lahkar2015logit}, \cite{lahkar2022generalized}. Note that since the space $\Sigma^{\alpha}_M(\Xi)$ is forward invariant under the RBBR dynamic, it is enough to prove the uniqueness result on $\Sigma^{\alpha}_{M}(\Xi)$. Recall from Section \ref{SectionBayesianPotential} that the virtual payoff vector is such that for all $\sigma \in \Sigma^{\alpha}_{M}(\Xi)$,
				\begin{equation}
					\widetilde{\mathcal{G}}(\sigma,\omega)=\mathcal{G}(\sigma,\omega)-\nabla^{\textbf{G}}\widetilde{\textbf{v}}[\sigma](\omega), \hspace{3mm} \text{ for all } \omega \in \Omega.
				\end{equation}
				Now, for every $\widehat{\sigma}_0 \in \widehat{\Sigma}_0$ and $t \in \mathbb{R}$ such that $\sigma+t\widehat{\sigma}_0 \in \Sigma^{\alpha}_{M}(\Xi)$, we define the following quantity:
				\begin{equation}
					\lambda(\sigma,\widehat{\sigma}_0,t):=\int_{\Omega} \langle \widetilde{\mathcal{G}}(\sigma+t\widehat{\sigma}_0,\omega),\widehat{\sigma}_0(\omega)  \rangle \xi(d\omega).
				\end{equation}
				We would now like to calculate the G\^ateaux derivative of the function $\lambda(\sigma,\widehat{\sigma}_0,\cdot)$ with respect to $t$. Notice that 
				\begin{align*}
					\lambda(\sigma,\widehat{\sigma}_0,t+\eta)&=\int_{\Omega} \langle \widetilde{\mathcal{G}}(\sigma+(t+\eta)\widehat{\sigma}_0,\omega),\widehat{\sigma}_0(\omega)  \rangle \xi(d\omega)\\
					&= \int_{\Omega} [\langle \widetilde{\mathcal{G}}(\sigma+t\widehat{\sigma}_0,\omega)+d\widetilde{\mathcal{G}}_{\omega}[\sigma+t\widehat{\sigma}_0](\eta\widehat{\sigma}_0),\widehat{\sigma}_0(\omega) \rangle+ \textbf{o}(|\eta|)]\xi(d\omega)\\
					&=\lambda(\sigma,\widehat{\sigma}_0,t) +\int_{\Omega}\langle d\widetilde{\mathcal{G}}_{\omega}[\sigma+t\widehat{\sigma}_0](\eta\widehat{\sigma}_0),\widehat{\sigma}_0(\omega) \rangle \xi(d\omega) + \textbf{o}(|\eta|).
				\end{align*}
				Thus, we have that 
				\begin{equation*}
					\frac{\partial \lambda(\sigma, \widehat{\sigma}_0, t)}{\partial t}=\int_{\Omega}\langle d\widetilde{\mathcal{G}}_{\omega}[\sigma+t\widehat{\sigma}_0](\widehat{\sigma}_0),\widehat{\sigma}_0(\omega) \rangle \xi(d\omega).
				\end{equation*}
				In order for us to complete the proof of the lemma, we need to compute the right hand side of the above equation. Note that from the definition of $\widetilde{\mathcal{G}}$, we have
				\begin{align}\label{GradReguPos}
					\int_{\Omega}\langle d\widetilde{\mathcal{G}}_{\omega}[\sigma+t\widehat{\sigma}_0](\widehat{\sigma}_0),\widehat{\sigma}_0(\omega) \rangle \xi(d\omega)& = \int_{\Omega}\langle	d	{\mathcal{G}}_{\omega}[\sigma+t\widehat{\sigma}_0](\widehat{\sigma}_0),\widehat{\sigma}_0(\omega) \rangle \xi(d\omega) \nonumber\\
					& \hspace{2cm} - \int_{\Omega}\langle d\nabla^{\textbf{G}} \widetilde{\textbf{v}}[\sigma+t\widehat{\sigma}_0](\widehat{\sigma}_0),\widehat{\sigma}_0(\omega) \rangle \xi(d\omega).
				\end{align}
				
				Since the game $\mathcal{G}$ is Fr\'echet differentiable, $\mathcal{G}$ is G\^ateaux differentiable too in which case, using the fact that $\mathcal{G}$ satisfies Bayesian self-defeating externalities, we have the following inequality:
				\begin{equation*}
					\int_{\Omega}\langle d{\mathcal{G}}_{\omega}(\sigma+t\widehat{\sigma}_0)(\widehat{\sigma}_0),\widehat{\sigma}_0(\omega) \rangle \xi(d\omega)= \int_{\Omega}\langle	D{\mathcal{G}}_{\omega}(\sigma+t\widehat{\sigma}_0)(\widehat{\sigma}_0),\widehat{\sigma}_0(\omega) \rangle \xi(d\omega) \leq 0.
				\end{equation*}
				
				We now show that the second term in the RHS of (\ref{GradReguPos}) is positive. For this, we note that for every $\sigma \in \Sigma^{\alpha}_M$, the following equality holds:
				\begin{equation}
					\nabla^{\textbf{G}} \widetilde{\textbf{v}}_{\epsilon}[\sigma](\omega)=\nabla \textbf{v}_{\epsilon}(\sigma(\omega)), \hspace{3mm} \text{ for all } \omega \in \Omega. 
				\end{equation}
				
				With the above equality in hand, we have that
				\begin{equation}\label{EqualityDerivativeGradient}
					\int_{\Omega}\langle d\nabla^{\textbf{G}} \widetilde{\textbf{v}}_{\epsilon}[\sigma+t\widehat{\sigma}_0](\widehat{\sigma}_0),\widehat{\sigma}_0(\omega)\xi(d\omega)=\int_{\Omega}\langle d\nabla \textbf{v}_{\epsilon}[\sigma(\omega)+t\widehat{\sigma}_0(\omega)](\widehat{\sigma}_0(\omega)),\widehat{\sigma}_0(\omega)\rangle\xi(d\omega).
				\end{equation}
				
				By assumption of the lemma, $\langle d\nabla \textbf{v}_{\epsilon}[\textbf{x}+t\textbf{z}_0](\textbf{z}_0),\textbf{z}_0 \rangle \geq 0$, for all $t \in \mathbb{R}$, $\textbf{x} \in \Delta$ and $\textbf{z}_0 \in \mathbb{R}^n_0$ such that $\textbf{x}+t\textbf{z}_0 \in \Delta$. Thus, noticing the fact that $\sigma(\omega)+t\widehat{\sigma}_0(\omega) \in \Delta$, $\widehat{\sigma}_0(\omega) \in \mathbb{R}^n_0$ for all $\omega \in \Omega$ and using the equality relation in (\ref{EqualityDerivativeGradient}), we have that 
				\begin{equation}
					\int_{\Omega}\langle d\nabla^{\textbf{G}} \widetilde{\textbf{v}}_{\epsilon}[\sigma+t\widehat{\sigma}_0](\widehat{\sigma}_0),\widehat{\sigma}_0(\omega)\xi(d\omega)=\int_{\Omega}\langle d\nabla \textbf{v}_{\epsilon}[\sigma(\omega)+t\widehat{\sigma}_0(\omega)](\widehat{\sigma}_0(\omega)),\widehat{\sigma}_0(\omega)\rangle\xi(d\omega). \geq 0.
				\end{equation}
				
				Therefore, the LHS of (\ref{GradReguPos}), and thus $\frac{\partial \lambda(\sigma, \widehat{\sigma}_0, t)}{\partial t}<0$. Now, let us fix a regularized Bayesian equilibrium ${\sigma}^{\circ}$. It then follows from standard theory of gradient of convex conjugate (see for instance \cite{rockafellar1970convex}) that $\widetilde{\mathcal{G}}^{i}(\sigma^{\circ},\omega)$ is independent of $i$ for every $\omega \in \Omega$. In other words, given the equilibrium $\sigma^{\circ}$, we have for every $\omega \in \Omega$, a $c(\sigma^{\circ},\omega)$ such that $\widetilde{\mathcal{G}}^i(\sigma^{\circ},\omega)=c(\sigma^{\circ},\omega)$ for all $\omega \in \Omega$ and $i \in S$. 
				As a result, we have that $\lambda(\sigma^{\circ},\widehat{\sigma}_0,0)=\int_{\Omega}\langle \mathcal{G}(\sigma^{\circ},\omega),\widehat{\sigma}_0(\omega)\rangle \xi(d\omega)=0$ for all $\widehat{\sigma}_0 \in \widehat{\Sigma}$. Now define $\bar{\sigma}:=\sigma^{\circ}+t\widehat{\sigma}_0$ for some non-zero $\widehat{\sigma}_0 \in \widehat{\Sigma}$ such that $\bar{\sigma}$ is distinct from $\sigma^{\circ}$. By definition, we have $\int_{\Omega}\langle \widetilde{\mathcal{G}}(\bar{\sigma},\omega),\widehat{\sigma}_0(\omega)  \rangle \xi(d\omega)=\lambda(\sigma^{\circ},\widehat{\sigma}_0,t)$.Also, from the previous arguments, we have that $\lambda(\sigma^{\circ},\widehat{\sigma}_0,0)=0$. This, in conjunction with the fact that $\frac{\partial \lambda(\bar{\sigma}, \widehat{\sigma}_0, t)}{\partial t}<0$, implies that $\lambda(\bar{\sigma},\widehat{\sigma}_0,t)<0$. Therefore, we have that $\int_{\Omega}\langle \widetilde{\mathcal{G}}(\bar{\sigma},\omega),\widehat{\sigma}_0(\omega)  \rangle \xi(d\omega)<0$, and hence $\bar{\sigma}$ cannot be a regularized equilibrium. This completes the proof of Lemma \ref{LemBayNegSemGameUnique}.
			\end{proof}

			\subsection{{Proof of Theorem \ref{bndc}}}\label{AppBayNegDefConv}

			Similar to the proof of Theorem \ref{bpgc}, we need to show that the Lyapunov function $\psi^{\alpha}_{M,\epsilon}$ decreases along every solution trajectory to the RBBR learning dynamic which originates in $\Sigma^{\alpha}_{M}(\Xi)$. To this end let us define the following quantities:
			\begin{itemize}
				\item $\psi^{1}_{\epsilon}(\sigma):=\int_{\Omega}\langle\mathcal{G}(\sigma,\omega),{\beta}_{\epsilon}[\sigma](\omega)\rangle{\xi}(d\omega)$, for all $\sigma \in \Sigma^{\alpha}_{M}(\Xi)$, 
				\item $\psi^{2}_{\epsilon}(\sigma):=\widetilde{\textbf{v}}_{\epsilon}({\beta}_{\epsilon}(\sigma))$, for all $\Sigma^{\alpha}_{M}(\Xi)$, 
				\item $\psi^{3}_{\epsilon}(\sigma):=\int_{\Omega}\langle\mathcal{G}(\sigma,\omega),\sigma(\omega)\rangle{\xi}(d\omega)$, for all $\Sigma^{\alpha}_{M}(\Xi)$, and \item $\psi^{4}_{\epsilon}(\sigma):=\widetilde{\textbf{v}}_{\epsilon}(\sigma)$, for all $\Sigma^{\alpha}_{M}(\Xi)$.
			\end{itemize}
			
			To prove the theorem, it is enough to show that $$\dot{\psi}_{\epsilon}(\sigma)=\dot{\psi}^1_{\epsilon}(\sigma)+\dot{\psi}^2_{\epsilon}(\sigma)+\dot{\psi}^3_{\epsilon}(\sigma)+\dot{\psi}^4_{\epsilon}(\sigma)\leq 0.$$ We now compute the derivatives of $\psi^{i}_{\epsilon}$, $i=1,\ldots,4$ separately. We have 
			\begin{align*}
				\text{D}{\psi}^{1}_{\epsilon}[\sigma](\dot{\sigma})&=\int_{\Omega}\langle\text{D}\mathcal{G}_{\omega}[\sigma](\dot{\sigma}),{\beta}_{\epsilon}[\sigma](\omega)\rangle{\xi}(d\omega)+\int_{\Omega}\langle\mathcal{G}_{\omega}(\sigma),\dot{{\beta}}_{\epsilon}[\sigma](\omega)\rangle{\xi}(d\omega).
			\end{align*}
			Since the regularizer $\widetilde{\textbf{v}}_{\epsilon}$ is G\^ateaux differentiable, we have that 
			\begin{align*}
				\text{D}{\psi}^{2}_{\epsilon}[\sigma](\dot{\sigma})&=\text{d}\widetilde{\textbf{v}}_{\epsilon}({{\beta}}_{\epsilon}(\sigma))(\dot{{{\beta}}}_{\epsilon}(\sigma))\\
				&=\int_{\Omega}\langle\dot{{{\beta}}}_{\epsilon}(\sigma),\nabla^{\textbf{G}}\widetilde{\textbf{v}}_{\epsilon}({{\beta}}_{\epsilon}(\sigma)) \rangle{\xi}(d\omega)\\
				&=\int_{\Omega}\nabla^{\textbf{G}}\widetilde{\textbf{v}}_{\epsilon}[{{\beta}}_{\epsilon}(\sigma)](\omega)(\dot{{{\beta}}}_{\epsilon}[\sigma](\omega)){\xi}(d\omega).
			\end{align*}
			For the third term, we have 
			\begin{align*}
				\text{D}\psi^{3}_{\epsilon}[\sigma](\dot{\sigma})&=\int_{\Omega}\langle\text{D}\mathcal{G}_{\omega}[\sigma](\dot{\sigma}),\sigma(\omega)\rangle{\xi}(d\omega)+\int_{\Omega}\langle\mathcal{G}_{\omega}(\sigma),\dot{\sigma}(\omega)\rangle{\xi}(d\omega).
			\end{align*}
			Finally for the fourth term we have 
			\begin{align*}
				\text{D}\psi^{4}_{\epsilon}[\sigma](\dot{\sigma})&=\text{d}\widetilde{\textbf{v}}_{\epsilon}[\sigma](\dot{\sigma})\\
				&=\int_{\Omega}\langle\dot{\sigma},\nabla^{\textbf{G}}\widetilde{\textbf{v}}_{\epsilon}(\sigma)\rangle{\xi}(d\omega)\\
				&=\int_{\Omega}\nabla^{\textbf{G}}\widetilde{\textbf{v}}_{\epsilon}[\sigma](\omega)(\dot{\sigma}(\omega)){\xi}(d\omega).
			\end{align*}
			Combining all the four equations above, we have after simplification that
			\begin{align*}
				\dot{\psi}^{\alpha}_{M,\epsilon}(\sigma)&=\int_{\Omega}\langle\text{D}\mathcal{G}_{\omega}[\sigma](\dot{\sigma}),{{\beta}}_{\epsilon}(\omega)(\omega)\rangle{\xi}(d\omega)+\int_{\Omega}\langle\mathcal{G}_{\omega}(\sigma),\dot{{{\beta}}}_{\epsilon}[\sigma](\omega)\rangle{\xi}(d\omega)\\
				&-\int_{\Omega}\nabla^{\textbf{G}}\widetilde{\textbf{v}}_{\epsilon}[{{\beta}}_{\epsilon}(\sigma)](\omega)(\dot{{{\beta}}}_{\epsilon}[\sigma](\omega)){\xi}(d\omega)-\int_{\Omega}\langle\text{D}\mathcal{G}_{\omega}[\sigma](\dot{\sigma}),\sigma(\omega)\rangle{\xi}(d\omega)\\
				&-\int_{\Omega}\langle\mathcal{G}_{\omega}(\sigma),\dot{\sigma}(\omega)\rangle{\xi}(d\omega)+\int_{\Omega}\nabla^{\textbf{G}}\widetilde{\textbf{v}}_{\epsilon}[\sigma](\omega)(\dot{\sigma}(\omega)){\xi}(d\omega)\\
				&=\int_{\Omega}\langle\text{D}\mathcal{G}_{\omega}[\sigma](\dot{\sigma}),\dot{\sigma}(\omega)\rangle{\xi}(d\omega)+\int_{\Omega}\langle\mathcal{G}_{\omega}(\sigma),\dot{{{\beta}}}_{\epsilon}[\sigma](\omega)-\dot{\sigma}(\omega)\rangle{\xi}(d\omega)\\
				&-\int_{\Omega}\nabla^{\textbf{G}}\widetilde{\textbf{v}}_{\epsilon}[{{\beta}}_{\epsilon}(\sigma)](\omega)(\dot{{{\beta}}}_{\epsilon}[\sigma](\omega)){\xi}(d\omega)+\int_{\Omega}\nabla^{\textbf{G}}\widetilde{\textbf{v}}_{\epsilon}[\sigma](\omega)(\dot{\sigma}(\omega)){\xi}(d\omega).
			\end{align*}
			
			Since $\mathcal{G}$ is a Bayesian negative semidefinite game and $\dot{\sigma} \in \Sigma_0$, it follows from Lemma \ref{sde} that the first term in the RHS of the above equation is less than or equal to zero, that is, $$\int_{\Omega}\langle\text{D}\mathcal{G}_{\omega}[\sigma](\dot{\sigma}),\dot{\sigma}(\omega)\rangle{\xi}(d\omega)\leq 0.$$ Using this fact and rearranging the other terms in the RHS of the above equation, we have 
			\begin{align*}
				\dot{\psi}^{\alpha}_{M,\epsilon}(\sigma) &\leq \int_{\Omega}\langle\mathcal{G}_{\omega}(\sigma)-\nabla^{\textbf{G}}\widetilde{\textbf{v}}_{\epsilon}[{{\beta}}_{\epsilon}(\sigma)](\omega),\dot{{{\beta}}}_{\epsilon}[\sigma](\omega)\rangle{\xi}(d\omega)\\
				&\hspace{3cm}-\int_{\Omega}\langle\mathcal{G}_{\omega}(\sigma)-\nabla^{\textbf{G}}\widetilde{\textbf{v}}_{\epsilon}[\sigma](\omega),\dot{\sigma}(\omega)\rangle{\xi}(d\omega).
			\end{align*}
			
			It now follows from Lemma (\ref{LemPopolnGame=GateauxDeri}) that the first term in the RHS of the above inequality is $0$. It also follows that the second term in the RHS of the above inequality is non-negative. As a result, we have $\dot{\psi}^{\alpha}_{M,\epsilon}\leq 0$ along every solution which originates in $\Sigma^{\alpha}_{M}(\Xi)$. In particular, this implies that the Lyapounov function defined in (\ref{BayNegDefLyapounov}) decreases along every non--stationary solution which originates in $\Sigma^{\alpha}_{M}(\Xi)$. By part (i) of Theorem \ref{ThmSigmaNormCpt}, we have that $\Sigma^{\alpha}_{M}(\Xi)$ is compact under the norm topology on $\widehat{\Sigma}$ and $\dot{\psi}^{\alpha}_{M,\epsilon}\leq 0$ with equality only at the rest points, it follows from the standard theory of dynamical systems (\cite{bhatia2006dynamical}) that the omega--limit of every solution trajectory satisfies $\dot{\psi}^{\alpha}_{M,\epsilon}=0$, which is the set of all rest points of the RBBR dynamic. We therefore have that every non--stationary solution arising in $\Sigma^{\alpha}_{M}(\Xi)$ converges to a regularized Bayesian equilibrium.

	\end{proof}

\end{appendix}

\bibliographystyle{plainnat}
\setcitestyle{numbers}
\bibliography{mybib}

\end{document}